\documentclass[10pt]{article}
\usepackage{eurosym}
\usepackage{amsfonts}
\usepackage{amssymb}
\usepackage{amsthm}
\usepackage{amsmath}
\usepackage{graphicx}
\usepackage{empheq}
\usepackage{indentfirst}
\usepackage{cite}
\usepackage{mathrsfs}
\usepackage{cases}
\usepackage{graphics}
\usepackage{xcolor}
\usepackage{bm}
\usepackage{makeidx}
\usepackage[T1]{fontenc}
\usepackage{times}
\newtheorem{theorem}{\textbf{Theorem}}[section]
\newtheorem{lemma}{\textbf{Lemma}}[section]
\newtheorem{proposition}{\textbf{Proposition}}[section]
\newtheorem{corollary}{\textbf{Corollary}}[section]
\newtheorem{remark}{\textbf{Remark}}[section]
\newtheorem{definition}{\textbf{Definition}}[section]

\allowdisplaybreaks[4]

\def\be{\begin{equation}}
\def\ee{\end{equation}}
\def\bea{\begin{eqnarray}}
\def\eea{\end{eqnarray}}
\def\bt{\begin{theorem}}
\def\et{\end{theorem}}
\def\bl{\begin{lemma}}
\def\el{\end{lemma}}
\def\br{\begin{remark}}
\def\er{\end{remark}}
\def\bp{\begin{proposition}}
\def\ep{\end{proposition}}
\def\bc{\begin{corollary}}
\def\ec{\end{corollary}}
\def\bd{\begin{definition}}
\def\ed{\end{definition}}

\begin{document}

\title{On the Cahn-Hilliard equation with kinetic rate dependent \\
dynamic boundary conditions and non-smooth potentials: \\
Well-posedness and asymptotic limits}
\author{
Maoyin Lv \thanks{%
School of Mathematical Sciences, Fudan University, Shanghai
200433, P.R. China. Email: \texttt{mylv22@m.fudan.edu.cn} }, \ \ 
Hao Wu \thanks{%
Corresponding author. School of Mathematical Sciences, Fudan University, Shanghai 200433, P.R. China. Email: \texttt{haowufd@fudan.edu.cn} } }
\date{\today }
\maketitle

\begin{abstract}
\noindent We consider a class of Cahn-Hilliard equation with kinetic rate dependent dynamic boundary conditions that describe possible short-range interactions between the binary mixture and the solid boundary. In the presence of surface diffusion on the boundary, the initial boundary value problem can be viewed as a transmission problem consisting of Cahn-Hilliard type equations both in the bulk and on the boundary. We first prove existence, uniqueness and continuous dependence of global weak solutions. In the construction of solutions, an explicit convergence rate in terms of the parameter for the Yosida approximation is established. Under some additional assumptions, we also obtain the existence and uniqueness of global strong solutions.
Next, we study the asymptotic limit as the coefficient of the boundary diffusion goes to zero and show that the limit problem with a forward-backward dynamic boundary condition is well-posed in a suitable weak formulation. Besides, we investigate the asymptotic limit as the kinetic rate tends to zero and infinity, respectively. Our results are valid for a general class of bulk and boundary potentials with double-well structure, including the physically relevant logarithmic potential and the non-smooth double-obstacle potential.


\end{abstract}

\tableofcontents

\section{Introduction}
The Cahn-Hilliard equation was proposed in \cite{CH} as a phenomenological model to describe spinodal decomposition in binary alloys.
It characterizes the fundamental process of phase separation due to certain non-Fickian diffusion driven by gradient of the chemical potential.
The Cahn-Hilliard equation belongs to the so-called diffuse interface models, in which the free interface between two components of the mixture is represented by a thin layer with finite thickness. The diffuse interface methodology avoids tracking free interfaces explicitly as in the classical free boundary problems, and it provides a thermodynamically consistent description for the evolution of complex geometries. In particular, topological changes of free interfaces can be handled in a natural and efficient way. For further information, we refer to, e.g., \cite{An97,BDGP,DF} and the references therein.
As a representative of the diffuse interface models, the Cahn-Hilliard equation has become a useful tool for the study of a wide variety of segregation-like phenomena arising, for instance, in material science, image inpainting, biology and fluid mechanics. In recent years, the study of boundary effects in the phase separation process of binary mixtures has attracted a lot of attention. To describe short-range interactions of the binary mixture with the solid wall, several types of dynamic boundary conditions for the Cahn-Hilliard equation have been proposed and investigated in the literature, see, for instance, \cite{JW23,Mi19,Wu22} and the references therein.

\subsection{Model description}

Let $T\in(0,\infty)$ be an arbitrary but fixed final time and $\Omega\subset\mathbb{R}^{d}$ ($d\in\{2,3\}$) be a smooth bounded domain with boundary $\Gamma:=\partial\Omega$. In this paper, we consider the following
initial boundary value problem of the Cahn-Hilliard equation subject to a class of dynamic boundary conditions \cite{GMS,KLLM,LW}:
\begin{align}
&\partial_{t}\varphi=\Delta\mu,&&\text{in }\Omega\times(0,T),
\label{1.1}\\
&\mu=-\Delta \varphi+F'(\varphi),&&\text{in }\Omega\times(0,T),
\label{1.2}\\
&\begin{cases}
L\partial_{\mathbf{n}}\mu=\theta-\mu,\quad\ \ \text{if}\ L\in [0,\infty),\\
\partial_{\mathbf{n}}\mu=0,\qquad\qquad  \text{if}\ L=\infty,
\end{cases}
&&\text{on }\Gamma\times(0,T),
\label{1.4}\\
&\partial_{t}\varphi=\Delta_{\Gamma}\theta-\partial_{\mathbf{n}}\mu,&&\text{on }\Gamma\times(0,T),
\label{1.5}\\
&\theta=\partial_{\mathbf{n}}\varphi-\delta\Delta_{\Gamma}\varphi+G'(\varphi),&&\text{on }\Gamma\times(0,T),
\label{1.6}\\
&\varphi(0)=\varphi_{0},&&\text{in }\overline{\Omega}
\label{1.7}.
\end{align}
Here, the phase function $\varphi: \overline{\Omega}\times (0,T)\to \mathbb{R}$ is related to local concentrations
of the two components of a binary mixture.
The total free energy functional associated with the system \eqref{1.1}--\eqref{1.6} is given by
\begin{align}
E(\varphi):=\underbrace{\int_{\Omega}\Big(\frac{1}{2}|\nabla \varphi|^{2}+F(\varphi)\Big)\,\mathrm{d}x}_{\text{bulk free energy}}
+\underbrace{\int_{\Gamma}\Big(\frac{\delta}{2}|\nabla_{\Gamma}\varphi|^{2}+G(\varphi)\Big)\,\mathrm{d}S}_{\text{surface free energy}}.
\label{eq1.16}
\end{align}
In particular, the surface free energy on the boundary is introduced to describe possible short-range interactions between the solid wall and components of the mixture \cite{FM97}. The nonlinear potential functions $F$ and $G$ denote free energy densities in the bulk and on the boundary, respectively.
To describe the phase separation process, they usually present a double-well structure, that is, with two minima and a local unstable maximum in between.
Typical and physically significant examples of such potentials include the logarithmic potential \cite{CH} and the double obstacle potential \cite{BE91}:
\begin{align}
&F_{\text{log}}(r):=(1+r)\mathrm{ln}(1+r)+(1-r)\mathrm{ln}(1-r)-c_{1}r^{2},\;\;r\in(-1,1),
\label{eq1.14}\\
&F_{\text{2obs}}(r):=
\begin{cases}
c_{2}(1-r^{2}),\quad\text{if}\ |r|\leq1,\\
\infty,\qquad\qquad \ \, \text{if}\ |r|>1,
\end{cases}
\label{eq1.15}
\end{align}
where the constants satisfy $c_{1}>1$ and $c_{2}>0$, so that $F_{\text{log}}$, $F_{\text{2obs}}$ are nonconvex.
In practice, regular double-well potential of polynomial type such as
\begin{align}
&F_{\text{reg}}(r):=\frac{1}{4}(r^{2}-1)^{2},\;\;r\in\mathbb{R},\label{eq1.13}
\end{align}
and its generalizations are widely used (see \cite{Mi19}).
In \eqref{1.2}, $\mu: \Omega\times(0,T)\to \mathbb{R}$ stands for the chemical potential in the bulk, while in \eqref{1.6},
$\theta: \Gamma\times(0,T)\to \mathbb{R}$ stands for the chemical potential on the boundary.
They can be expressed as Fr\'echet derivatives of the bulk and surface free energies in \eqref{eq1.16}, respectively.
The symbol $\Delta$ denotes the Laplace operator in $\Omega$, $\Delta_\Gamma$ denotes
the Laplace-Beltrami operator on $\Gamma$ and $\nabla_\Gamma$ denotes the tangential (surface) gradient operator.
In the boundary conditions \eqref{1.4}--\eqref{1.6}, the symbol $\mathbf{n}$ denotes the outward normal vector on $\Gamma$
and $\partial_{\mathbf{n}}$ means the outward normal derivative on $\Gamma$. The nonlinearities
$F'$ and $G'$ in \eqref{1.2} and \eqref{1.6} simply denote derivatives of the related potentials.
Nevertheless, when non-smooth potentials are taken into account, $F'$ and $G'$
correspond to the subdifferential of the convex part (may be multivalued
graphs) plus the derivative of the smooth concave contribution.
For example, we have
$F_{\text{2obs}}'(r)=\partial I_{[-1,1]}(r) - 2c_{2}r$, where $\partial I_{[-1,1]}(r)$ is the subdifferential of the indicator function of $[-1, 1]$.
In this case, one should replace the equality in \eqref{1.2} and \eqref{1.6} by inclusion.
The boundary conditions \eqref{1.4}--\eqref{1.6} allow descriptions of the physically realistic scenario with
possible mass transfer between bulk and boundary as well as a dynamic angle between
the free interface (separating components of the binary mixture) and the solid boundary at contact line.
In this aspect, we refer to \cite{GK}, where a general thermodynamically consistent Navier-Stokes-Cahn-Hilliard system with dynamic boundary
conditions for incompressible two-phase flows with non-matched densities was introduced and analyzed
(see \cite{GLW} for the existence of global weak solutions in a more general setting).
Besides, we refer to \cite{CKS} for a related Cahn-Hilliard-Brinkman model on two-phase flows through porous media.

In our problem \eqref{1.1}--\eqref{1.7}, we maintain two parameters $L\in [0,\infty]$ and $\delta\in [0,\infty)$ that are important in the subsequent analysis. Other coefficients are set to be one for the sake of simplicity.

We first explain the role of $L\in [0,\infty]$. The bulk and boundary chemical potentials $\mu$, $\theta$ are coupled through the boundary condition \eqref{1.4}, which accounts for possible adsorption or desorption processes between the materials in the bulk and on the boundary \cite{KLLM}.
The mass flux $\partial_{\mathbf{n}}\mu$ that describes the motion of materials towards and away from the boundary,
is driven by the difference in the chemical potentials. In this sense, the coefficient $1/L$ can be interpreted as a kinetic rate.
The value of $L$ distinguishes different types of bulk-boundary interactions.
The case $L=0$ was introduce by Goldstein, Miranville and Schimperna \cite{GMS} (GMS in short)
for phase separation of a binary mixture confined to a bounded domain with porous walls.
The GMS model extends the Cahn-Hilliard equation with Wentzell type boundary conditions proposed by Gal \cite{Gal06} (see \cite{GalWu,Wu07} for related mathematical analysis). Taking $L=0$ in \eqref{1.4}, we obtain the Dirichlet boundary condition $\mu=\theta$ on $\Gamma$, which implies that the chemical potentials $\mu$ and $\theta$ are always in the chemical equilibrium (see \cite{GMS} for more general situations with a factor that can be a uniformly bounded positive function). The case $L=\infty$ was introduced by Liu and Wu \cite{LW} (LW in short) based on an energetic variational approach that combines the least action principle and Onsager's principle of maximum energy dissipation.
Then the homogeneous Neumann boundary condition for $\mu$ implies that there is no mass transfer between the bulk and boundary.
In such a situation, the chemical potentials $\mu$ and $\theta$ are not directly coupled.
Nevertheless, interactions between the bulk and the surface materials take place through the phase function $\varphi$.
Indeed, let us introduce a new unknown
\begin{align}
\psi=\varphi|_{\Gamma},&& \text{on}\ \Gamma\times(0,T),
\label{trace1}
\end{align}
where $\varphi|_{\Gamma}$ denote the trace of $\varphi$ (cf. \cite{Mi19,MZ04}).
Then we find that the system \eqref{1.1}--\eqref{1.6} yields a sort of transmission problem between the dynamics in the bulk $\Omega$ and the one on the boundary $\Gamma$. Finally, the case $L\in (0,\infty)$ was recently introduced by Knopf, Lam, Liu and Metzger \cite{KLLM} (KLLM in short).
The corresponding Robin type boundary condition \eqref{1.4} describes the situation that the chemical potentials $\mu$ and $\theta$ are not in equilibrium and they are related through the mass flux. Formally speaking, the KLLM model (with $0<L<\infty$) can be regarded as an interpolation between the GMS model ($L=0$, instantaneous mass transfer) and the LW model ($L=\infty$, no mass transfer) via a finite, positive relaxation parameter $L$ (see \cite{KLLM} for a rigorous verification).

Next, let us comment on the parameter $\delta\in [0,\infty)$,
which acts as a weight for surface diffusion effects on the boundary $\Gamma$.
When $\delta>0$, \eqref{1.5} together with \eqref{1.6} leads to a Cahn-Hilliard type dynamic boundary condition.
The case $\delta=0$ is closely related to the evolution of a free interface in contact with the solid
boundary, that is, the moving contact line problem \cite{CWX,QWS}.
From the mathematical point of view, without surface diffusion,
the boundary conditions \eqref{1.5}--\eqref{1.6} (formally) reduce to
\begin{align}
&\partial_{t}\varphi - G''(\varphi)\Delta_\Gamma\varphi =G^{(3)}(\varphi)|\nabla_\Gamma \varphi|^2 + \Delta_{\Gamma}(\partial_{\mathbf{n}}\varphi) -\partial_{\mathbf{n}}\mu,&&\text{on }\Gamma\times(0,T).
\label{1.6b}
\end{align}
In the regime that the potential $G$ is non-convex, in particular, $G''(\varphi)\leq 0$, we obtain a backward heat equation on $\Gamma$, whose
well-posedness is usually a delicate issue. As pointed out in \cite{CFS,CFSJEE}, \eqref{1.6b} yields a forward-backward dynamic boundary condition,
complemented with a Cahn-Hilliard equation \eqref{1.1}--\eqref{1.2} in the bulk.

The values of $L$ and $\delta$ also lead to differences in some basic properties
for the thermodynamically consistent problem \eqref{1.1}--\eqref{1.7} such as mass conservation and energy dissipation (see \cite{GMS,KLLM,LW,Wu22}).
For a sufficiently regular solution, we have the energy dissipation law
\begin{align}
& \frac{\mathrm{d}}{\mathrm{d}t}E(\varphi(t))+\int_{\Omega}|\nabla\mu(t)|^{2}\,\mathrm{d}x
+\int_{\Gamma}|\nabla_{\Gamma}\theta(t)|^{2}\,\mathrm{d}S +\chi(L)\int_\Gamma |\mu(t)-\theta(t)|^2\,\mathrm{d}S=0,
\quad \forall\,t\in (0,T),
\notag
\end{align}
where $\chi(L)=1/L$ if $L\in (0,\infty)$ and $\chi(L)=0$ if $L=0,\infty$.
This implies that for all $L\in [0,\infty]$, the total free energy $E(\varphi)$ is decreasing as time evolves.
Nevertheless, the form of the free energy $E$ depends on $\delta$ (recall \eqref{eq1.16})
and the energy dissipation varies according to $L$.
On the other hand, we obtain the conservation of total mass for $L\in [0,\infty)$,
\begin{align}
\int_{\Omega}\varphi(t)\,\mathrm{d}x
+\int_{\Gamma}\varphi(t)\,\mathrm{d}S
=\int_{\Omega}\varphi_0\,\mathrm{d}x
+\int_{\Gamma}\varphi_0\,\mathrm{d}S,
\quad\,\forall\,t\in[0,T],
\notag
\end{align}
while for $L=\infty$, mass conservation laws hold in the bulk and on the boundary separately:
\begin{align}
\int_{\Omega}\varphi(t)\,\mathrm{d}x = \int_{\Omega}\varphi_0\,\mathrm{d}x,\quad
\int_{\Gamma}\varphi(t)\,\mathrm{d}S = \int_{\Gamma}\varphi_0\,\mathrm{d}S,
\quad\,\forall\,t\in[0,T].
\notag
\end{align}

\subsection{A brief overview of related literature}
The Cahn-Hilliard equation with different types of dynamic boundary conditions
has been extensively studied from various viewpoints \cite{Mi19,Wu22}.
For the case with a dynamic boundary condition of Allen-Cahn type, that is,
$\partial_{t}\varphi -\delta\Delta_{\Gamma}\varphi+G'(\varphi) + \partial_{\mathbf{n}}\varphi= 0$ on $\Gamma$,
we quote \cite{CWX,CF20,CGS14,CMZ,GMS09,GMS10,MZ,RZ03,WZ04} among the vast literature.
We note that the dynamic boundary condition of Allen-Cahn type corresponds to an $L^2$-relaxation of the surface energy (see \cite{QWS}),
while in our problem \eqref{1.1}--\eqref{1.7}, the dynamic boundary condition of Cahn-Hilliard type accounts for the
mass transport on $\Gamma$ with a $(H^1)'$-relaxation dynamics.

For mathematical analysis of the GMS model ($L=0$), we refer to \cite{CSM,CP,CF15,CFS,CGS18,FW,GMS} and the references therein.
In \cite{GMS}, existence, uniqueness, regularity and long-time behavior of global weak solutions
were established under general assumptions on the nonlinearities, with $G$ being regular.
On the other hand, well-posedness of the GMS model with non-smooth bulk and boundary potentials was proved in \cite{CF15}, where the boundary potential was assumed to dominate the bulk one. See also \cite{FW} for the separation property and longtime behavior, and \cite{CGS18} when convection effects are taken into account. In \cite{CFS}, the asymptotic analysis as $\delta\to 0$,
i.e., the surface diffusion term on the dynamic boundary condition tends to $0$, was carried out in a very general setting with nonlinear terms admitting maximal monotone graphs both in the bulk and on the boundary.

Concerning the LW model ($L=\infty$), well-posedness and long-time behavior were first established in \cite{LW}
when $F$ and $G$ are suitable regular potentials. By introducing a slightly weaker notion of the solution,
the authors of \cite{GP} proved existence of weak solutions via a gradient flow approach,
removing the additional geometric assumption imposed in \cite{LW} for the case $\delta=0$.
Well-posedness for the LW model with general non-smooth potentials has been obtained in \cite{CFW}.
In \cite{CFSJEE}, the asymptotic analysis as $\delta\to 0$ was investigated with non-smooth potentials in the bulk and on the boundary.
We also mention \cite{MW} for the existence of global attractors and \cite{Me21} for the numerical analysis.

For the KLLM model $(0<L<\infty$), weak and strong well-posedness was established in \cite{KLLM} with suitable regular potentials $F$ and $G$.
Under similar setting, long-time behavior such as existence of a global attractor/exponential attractors and convergence to a single equilibrium were
obtained in \cite{GKY}. Asymptotic limits as $L\to 0$ and $L\to \infty$ were rigorously verified in \cite{KLLM} with regular potentials, see also \cite{GKY} for further properties in the limit $L\to 0$. A nonlocal variant of the KLLM model (including a nonlocal dynamic boundary condition) was proposed and investigated in \cite{KS}. Besides, for simulations and numerical analysis, we refer to \cite{BZ,KLLM}.
In the recent contribution \cite{KS24}, a class of more general bulk-surface convective Cahn-Hilliard systems with dynamic boundary conditions and regular potentials $F$, $G$ were investigated. There, the trace relation \eqref{trace1} for the phase function is further relaxed as follows (cf. \eqref{1.4} for the chemical potentials):
\begin{align}
    \begin{cases}
    J\partial_{\mathbf{n}}\varphi=\psi-\varphi|_{\Gamma},\quad\ \  \text{if}\ J\in [0,\infty),\\
    \partial_{\mathbf{n}}\varphi=0,\qquad \qquad \quad \text{if}\ \,J= \infty,
    \end{cases}
    && \text{on }\Gamma\times(0,T).
    \notag
\end{align}
The Robin approximation with $J\in (0,\infty)$ describes a scenario, where the boundary phase variable and the trace of the bulk phase variable are not proportional (see \cite{KLam20} for a similar consideration for the LW model).
Existence of weak solutions for $J,L\in(0,\infty)$ was proved \cite{KS24} by means of a Fadeo-Galerkin approach. For other cases, existence results were obtained by studying the asymptotic limit as sending $J,L$ to $0$ and $\infty$, respectively.
It is worth mentioning that all the related results mentioned above for the KLLM model were achieved for regular potentials $F$, $G$, singular potentials such as the logarithmic potential \eqref{eq1.14} and the double-obstacle potential \eqref{eq1.15} are unfortunately not admissible.
The only known result on the existence of weak solutions to the KLLM model with singular potentials including \eqref{eq1.14} can be obtained as a consequence of \cite{GLW} on a generalized Navier-Stokes-Cahn-Hilliard system with dynamic boundary conditions. However, uniqueness and regularity properties were not available there due to the presence of fluid interaction.

\subsection{Goal of the paper}
In this paper, we aim to study well-posedness and asymptotic limits as $\delta\to 0$, $L\to 0$ or $L\to \infty$ of the initial boundary value problem \eqref{1.1}--\eqref{1.7} with a wide class of bulk/boundary potentials $F$, $G$ that have a double-well structure, in particular, including the non-smooth
logarithmic potential \eqref{eq1.14} and the double-obstacle potential \eqref{eq1.15}.
\begin{itemize}
\item[(1)] Well-posedness of the KLLM model ($0<L<\infty$) with surface diffusion ($\delta>0$).
We prove the existence of global weak solutions (see Theorem \ref{weakexist}) and their continuous dependence on the data that yields the uniqueness (see Theorem \ref{contidepen}).
Thanks to the solvability of a second-order elliptic problem with bulk-surface coupling \cite{KL},
we are allowed to apply the approach in \cite{CF15} for the GMS model to conclude the existence result.
To this end, we consider a regularized problem by adding viscous terms in the Cahn-Hilliard equation
as well as the dynamic boundary condition and substituting the maximal monotone graphs with their Yosida regularizations.
The regularized problem can be solved by the abstract theory of doubly nonlinear evolution inclusions (see Proposition \ref{approexist}).
After that, we derive suitable estimates for the approximate solutions, which are uniform with
respect to the parameter $\varepsilon$ for the Yosida regularization.
Passing to the limit as $\varepsilon\to 0$, we can construct a weak solution by the compactness argument.
The continuous dependence estimate can be proved by the energy method.
Moreover, we establish an $O(\varepsilon^{1/2})$-estimate for the convergence of approximate phase functions in $L^{\infty}(0,T;(\mathcal{H}^{1})')\cap L^{2}(0,T;\mathcal{V}^{1})$ (see Proposition \ref{rate}), which seems to be the first result of this kind for problem \eqref{1.1}--\eqref{1.7}.
Finally, after deriving some higher order (in time) uniform estimates for the approximate solutions, we obtain the existence of a unique strong solution (see Theorem \ref{strongexist}). Our contribution extends previous works on well-poseness of the KLLM model with regular potentials \cite{KLLM,KS24}.

\item[(2)] Asymptotic limit as $\delta\to 0$ for the KLLM model with fixed $L\in (0,\infty)$.
We show that weak solutions to problem \eqref{1.1}--\eqref{1.7} obtained in Theorems \ref{weakexist}, \ref{contidepen}
converge as $\delta\to 0$ (in the sense of a subsequence), and thus prove the existence of weak solutions to the limit problem in which
the boundary condition \eqref{1.6} is replaced by the one with $\delta=0$ (see Theorem \ref{existthm}).
Besides, we establish a continuous dependence result for the phase function with respect to the data (see Theorem \ref{continuousdepen}).
This implies that the coupling of a forward-backward type boundary condition with the Cahn--Hilliard equation in the bulk
can be well-posed in a suitable sense. Analogous conclusions have been obtained for the GMS model and the LW model with general bulk/boundary potentials respectively in \cite{CFS} and \cite{CFSJEE}. Hence, our results fill the gap left by \cite{CFS,CFSJEE}.
Like in \cite{CFS,CFSJEE}, the solution of the limit problem with vanishing surface diffusion looses some spatial regularity, so several terms on the boundary
including $\partial_\mathbf{n}\varphi$ should be understood in a weaker sense.

\item[(3)] Asymptotic limit as $L\to 0$ or $L\to\infty$ for the KLLM model with fixed $\delta\in (0,\infty)$.
In presence of the surface diffusion in the KLLM model, we rigorously justify the limit case of instantaneous reaction as $L\to 0$ (see Theorem \ref{asymptotic0}), where the chemical potentials are in equilibrium, and a vanishing reaction rate as $L\to \infty$ (see Theorem \ref{asymptoticinfinity}), where the chemical potentials are not directly coupled. Comparing with the previous results in \cite{KLLM} that are only valid for suitable regular potentials, the main novelty in our analysis is the treatment of non-smooth bulk and boundary potentials that include logarithmic potential \eqref{eq1.14} and the double-obstacle potential \eqref{eq1.15}.
In order to derive uniform estimates with respective to the parameter $L$, different approaches have to be applied in the regime of vanishing (or large) kinetic rate. We remark that the asymptotic limit with respect to $L$ turns out to be more involved in the case of vanishing surface diffusion and general potentials. This issue will be studied in a future work.
\end{itemize}

\textbf{Plan of the paper.}
The remaining part of this paper is organized as follows.
In Section \ref{sec2}, we introduce our notations, assumptions and then state the main results.
In Section \ref{sec3}, we consider the KLLM model with surface diffusion, proving the existence of weak solutions, the continuous dependence and the existence of strong solutions. Besides, we obtain a convergence rate for the approximate phase functions.
In Section \ref{sec4}, we investigate the asymptotic limit $\delta\to 0$, which yields the existence of weak solutions to the limit problem with vanishing surface diffusion. A continuous dependence estimate is also derived.
In Section \ref{sec5}, we study asymptotic limits with respect to the kinetic rate $L\to 0$, $L\to \infty$, in presence of the surface diffusion.
In the Appendix, we list some useful tools that are frequently used in this paper.

\section{Main Results}
\label{sec2}
\setcounter{equation}{0}
In this section, we first recall some notations for the functional settings,
then we describe our problem and state the main results.

\subsection{Preliminaries}
For any real Banach space $X$, we denote its norm by $\|\cdot\|_X$, its dual space by $X'$
and the duality pairing between $X'$ and $X$ by
$\langle\cdot,\cdot\rangle_{X',X}$. If $X$ is a Hilbert space,
its inner product will be denoted by $(\cdot,\cdot)_X$.
The space $L^q(0,T;X)$ ($1\leq q\leq \infty$)
denotes the set of all strongly measurable $q$-integrable functions with
values in $X$, or, if $q=\infty$, essentially bounded functions.
The space $C([0,T];X)$ denotes the Banach space of all bounded and
continuous functions $u:[ 0,T] \rightarrow X$ equipped with the supremum
norm, while $C_{w}([0,T];X)$ denotes the topological vector space of all
bounded and weakly continuous functions.

Let $\Omega$ be a bounded domain in $\mathbb{R}^d$ ($d\in \{2,3\}$) with sufficiently smooth boundary $\Gamma:=\partial \Omega$ (at least Lipschitz).
The associated outward unit normal vector field on $\Gamma$ is denoted by $\mathbf{n}$. We use $|\Omega|$ and $|\Gamma|$
 to denote the Lebesgue measure of $\Omega$ and the Hausdorff measure of $\Gamma$, respectively.
For any $1\leq q\leq \infty$, $k\in \mathbb{N}$, the standard Lebesgue and Sobolev spaces on $\Omega$ are denoted by $L^{q}(\Omega )$
and $W^{k,q}(\Omega)$. Here, we use $\mathbb{N}$ for the set of natural numbers including zero.
For $s\geq 0$ and $q\in [1,\infty )$, we denote by $H^{s,q}(\Omega )$ the Bessel-potential spaces and by $W^{s,q}(\Omega )$ the Slobodeckij spaces.
If $q=2$, it holds $H^{s,2}(\Omega)=W^{s,2}(\Omega )$ for all $s$ and these spaces are Hilbert spaces.
We shall use the notations $H^s(\Omega)=H^{s,2}(\Omega)=W^{s,2}(\Omega )$ and $H^0(\Omega)$ can be identified with $L^2(\Omega)$.
The Lebesgue spaces, Sobolev spaces and Slobodeckij spaces on the boundary $\Gamma$ can be defined analogously,
provided that $\Gamma$ is sufficiently regular.
We write $H^s(\Gamma)=H^{s,2}(\Gamma)=W^{s,2}(\Gamma)$ and identify $H^0(\Gamma)$ with $L^2(\Gamma)$.
Hereafter, the following shortcuts will be applied:
  \begin{align*}
  &H:=L^{2}(\Omega),\quad H_{\Gamma}:=L^{2}(\Gamma),\quad V:=H^{1}(\Omega),\quad V_{\Gamma}:=H^{1}(\Gamma).
  \end{align*}
For every $y\in (H^1(\Omega))'$, we denote by $%
\langle y\rangle_\Omega=|\Omega|^{-1}\langle
y,1\rangle_{(H^1(\Omega))',\,H^1(\Omega)}$ its generalized mean
value over $\Omega$. If $y\in L^1(\Omega)$, then its spatial mean is simply
given by $\langle y\rangle_\Omega=|\Omega|^{-1}\int_\Omega y \,\mathrm{d}x$.
The spatial mean for a function $y_\Gamma$ on $\Gamma$, denoted by $\langle
y_\Gamma\rangle_\Gamma$, can be defined in a similar manner.
Then we introduce the spaces for functions with zero mean:
\begin{align*}
&V_{0}:=\big\{y\in V:\;\langle y\rangle_\Omega =0\big\},&& V_{0}^*:=\big\{y^*\in V':\;\langle y^*\rangle_\Omega =0\big\},
 \\
&V_{\Gamma,0}:=\big\{y_{\Gamma}\in V_{\Gamma}:\; \langle y_\Gamma\rangle_\Gamma =0\big\},&& V_{\Gamma,0}^*:=\big\{y_\Gamma^* \in V_{\Gamma}':\;\langle y_\Gamma^*\rangle_\Gamma =0\big\}.
\end{align*}
The following Poincar\'{e}-Wirtinger inequalities in $\Omega$ and on $\Gamma$ hold (see e.g., \cite[Theorem 2.12]{DE13} for the case on $\Gamma$):
\begin{align}
&\|u-\langle u\rangle_\Omega\|_{H} \leq C_\Omega\|\nabla u\|_{H},\qquad\quad\quad \ \ \ \forall\,
u\in V, \label{Po1}\\
&\|u_\Gamma-\langle u_\Gamma\rangle_\Gamma\|_{H_\Gamma} \leq C_\Gamma\|\nabla_\Gamma u_\Gamma\|_{H_\Gamma},\qquad \forall\,
u_\Gamma\in V_\Gamma,
\label{Po2}
\end{align}
where $C_\Omega$ (resp. $C_\Gamma$) is a positive constant depending only on $\Omega$ (resp. $\Gamma$).
In \eqref{Po2}, we denote by $\nabla_\Gamma$ the surface (tangential) gradient on $\Gamma$.
For basic facts of calculus on surfaces, see, e.g., \cite[Section 2]{DE13}.

Consider the Neumann problem
\begin{align}
\begin{cases}
-\Delta u=y, \,\quad \text{in} \ \Omega,\\
\partial_{\mathbf{n}} u=0, \quad \ \ \, \text{on}\ \Gamma.
\end{cases}
\label{P-n}
\end{align}
Owing to the Lax-Milgram theorem, for every $y\in V_0^*$, problem \eqref{P-n} admits a unique weak solution $u\in V_{0}$ satisfying
\begin{align}
\int_{\Omega}\nabla u\cdot\nabla \zeta\,\mathrm{d}x
= \langle y,\zeta \rangle_{V',V},\quad\forall\,\zeta\in V.
\notag
\end{align}
Thus, we can define the solution operator $\mathcal{N}_{\Omega}: V_{0}^*\rightarrow V_{0}$ such that $u=\mathcal{N}_{\Omega}y$.
Analogously, let us consider the surface Poisson equation
\begin{align}
-\Delta_\Gamma u_\Gamma = y_\Gamma,\quad \text{on}\ \ \Gamma,
\label{P-s}
\end{align}
where $\Delta_\Gamma=\mathrm{div}_\Gamma\nabla_\Gamma$ denotes the Laplace-Beltrami operator on $\Gamma$.
For every $y_\Gamma\in V_{\Gamma,0}^*$, problem \eqref{P-s} admits a unique weak solution $u_\Gamma\in V_{\Gamma,0}$ satisfying
\begin{align}
\int_{\Gamma}\nabla_{\Gamma} u_{\Gamma}\cdot\nabla_{\Gamma} \zeta_{\Gamma}\,\mathrm{d}S
= \langle y_{\Gamma},\zeta_{\Gamma} \rangle_{V_{\Gamma}',V_{\Gamma}},\quad\forall\,\zeta_{\Gamma}\in V_{\Gamma}.
\notag
\end{align}
Then we can define the solution operator $\mathcal{N}_{\Gamma}:V_{\Gamma,0}^*\rightarrow V_{\Gamma,0}$ such that $u_{\Gamma}=\mathcal{N}_{\Gamma}y_{\Gamma}$.
By virtue of these definitions, we can introduce the following equivalent norms
\begin{align*}
&\Vert y\Vert_{V_{0}^*}:=\Big(\int_{\Omega}|\nabla\mathcal{N}_{\Omega}y|^{2}\,\mathrm{d}x\Big)^{1/2},
&& \forall\, y\in V_{0}^*,\\
&\Vert y\Vert_{V'}:=\Big(\Vert y - \langle y\rangle_\Omega \Vert_{V_{0}^*}^2+ |\langle y\rangle_\Omega|^2\Big)^{1/2},
&& \forall\, y\in V',\\
&\Vert y_{\Gamma}\Vert_{V_{\Gamma,0}^*}:=\Big(\int_{\Gamma}|\nabla_{\Gamma}\mathcal{N}_{\Gamma}y_{\Gamma}|^{2}\,\mathrm{d}S\Big)^{1/2},
&& \forall\,y_{\Gamma}\in V_{\Gamma,0}^*,\\
&\Vert y_{\Gamma}\Vert_{V_{\Gamma}'}:=\Big(\Vert y_{\Gamma}- \langle y_{\Gamma}\rangle_\Gamma \Vert_{V_{\Gamma,0}^*}^2 + |\langle y_{\Gamma}\rangle_\Gamma|^2\Big)^{1/2},
&& \forall\,y_{\Gamma}\in V_{\Gamma}'.
\end{align*}

Next, we introduce the product spaces
$$\mathcal{L}^{q}:=L^{q}(\Omega)\times L^{q}(\Gamma)\quad\mathrm{and}\quad\mathcal{H}^{k}:=H^{k}(\Omega)\times H^{k}(\Gamma),$$
for $q\in [1,\infty]$ and $k\in \mathbb{N}$.
Like before, we can identify  $\mathcal{H}^{0}$ with $\mathcal{L}^{2}$.
 For any $k\in \mathbb{N}$, $\mathcal{H}^{k}$ is a Hilbert spaces endowed with the standard inner product
  $$
  \big((y,y_{\Gamma}),(z,z_{\Gamma})\big)_{\mathcal{H}^{k}}:=(y,z)_{H^{k}(\Omega)}+(y_{\Gamma},z_{\Gamma})_{H^{k}(\Gamma)},\quad\forall\, (y,y_{\Gamma}), (z,z_{\Gamma})\in\mathcal{H}^{k}
  $$
  and the induced norm $\Vert\cdot\Vert_{\mathcal{H}^{k}}:=(\cdot,\cdot)_{\mathcal{H}^{k}}^{1/2}$.
  We introduce the duality pairing
  \begin{align*}
  \big\langle (y,y_\Gamma),(\zeta, \zeta_\Gamma)\big\rangle_{(\mathcal{H}^1)',\mathcal{H}^1}
  = (y,\zeta)_{L^2(\Omega)}+ (y_\Gamma, \zeta_\Gamma)_{L^2(\Gamma)},
  \quad \forall\, (y,y_\Gamma)\in \mathcal{L}^2,\ (\zeta, \zeta_\Gamma)\in \mathcal{H}^1.
  \end{align*}
  By the Riesz representation theorem, this product
can be extended to a duality pairing on $(\mathcal{H}^1)'\times \mathcal{H}^1$.

For any $k\in\mathbb{Z}^+$, we introduce the Hilbert space
  $$
  \mathcal{V}^{k}:=\big\{(y,y_{\Gamma})\in\mathcal{H}^{k}\;:\;y|_{\Gamma}=y_{\Gamma}\ \ \text{a.e. on }\Gamma\big\},
  $$
  endowed with the inner product $(\cdot,\cdot)_{\mathcal{V}^{k}}:=(\cdot,\cdot)_{\mathcal{H}^{k}}$ and the associated norm $\Vert\cdot\Vert_{\mathcal{V}^{k}}:=\Vert\cdot\Vert_{\mathcal{H}^{k}}$.
  Here, $y|_{\Gamma}$ stands for the trace of $y\in H^k(\Omega)$ on the boundary $\Gamma$, which makes sense for $k\in \mathbb{Z}^+$.
  The duality pairing on $(\mathcal{V}^1)'\times \mathcal{V}^1$ can be defined in a similar manner.
  For convenience, we also use the notation
  $$
  \widetilde{\mathcal{V}}^{k}:=\big\{(y,y_{\Gamma})\in H^{k}(\Omega)\times H^{k-1/2}(\Gamma)\;:\;y|_{\Gamma}=y_{\Gamma}\ \ \text{a.e. on }\Gamma\big\},
  \quad  k\in \mathbb{Z}^+.
  $$
  Thanks to the trace theorem, for every $y\in H^k(\Omega)$, $k\in \mathbb{Z}^+$, it holds $(y,y|_{\Gamma})\in  \widetilde{\mathcal{V}}^{k}$.

 For any given $m\in\mathbb{R}$, we set
 $$
 \mathcal{L}^{2}_{(m)}:=\big\{(y,y_{\Gamma})\in\mathcal{L}^{2}\;:\;|\Omega|\langle y\rangle_{\Omega}+|\Gamma|\langle y_{\Gamma}\rangle_{\Gamma}=m\big\}.
 $$
 The closed linear subspaces
 $$
 \mathcal{H}_{(0)}^k=\mathcal{H}^{k}\cap\mathcal{L}_{(0)}^{2},
 \qquad \mathcal{V}_{(0)}^k=\mathcal{V}^{k}\cap\mathcal{L}_{(0)}^{2},
 \qquad k\in \mathbb{Z}^+,
  $$
 are Hilbert spaces endowed with the inner products $(\cdot,\cdot)_{\mathcal{H}^{k}}$
 and the associated norms $\Vert\cdot\Vert_{\mathcal{H}^{k}}$, respectively.
 For $L\in [0,\infty)$ and $k\in \mathbb{Z}^+$,  we introduce the notations
  $$
  \mathcal{H}^{k}_{L}:=
  \begin{cases}
  \mathcal{H}^k,\quad \text{if}\ L\in (0,\infty),\\
  \mathcal{V}^{k},\quad \ \text{if}\ L=0,
  \end{cases}\qquad
  \mathcal{H}^{k}_{L,0}:=
  \begin{cases}
  \mathcal{H}_{(0)}^k,\quad \text{if}\ L\in (0,\infty),\\
  \mathcal{V}^{k}_{(0)},\quad \ \text{if}\ L=0.
  \end{cases}
  $$
  Consider the bilinear form
  \begin{align}
  a_{L}\big((y,y_{\Gamma}),(z,z_{\Gamma})\big) :=\int_{\Omega}\nabla y\cdot\nabla z \,\mathrm{d}x +\int_{\Gamma}\nabla_{\Gamma}y_{\Gamma}\cdot\nabla_{\Gamma}z_{\Gamma}\,\mathrm{d}S
  +\chi(L)\int_{\Gamma}(y-y_{\Gamma})(z-z_{\Gamma})\,\mathrm{d}S,
  \notag
  \end{align}
  for all $(y,y_{\Gamma}), (z,z_{\Gamma})\in \mathcal{H}^{1}$, where
  $$
  \chi(L)=\begin{cases}
  1/L,\quad \text{if}\ L\in (0,\infty),\\
  0,\qquad \ \text{if}\ L=0.
  \end{cases}
  $$
 For any $(y,y_{\Gamma})\in \mathcal{H}^{1}_{L,0}$, $L\in [0,\infty)$, we define
  \begin{align}
  \Vert(y,y_{\Gamma})\Vert_{\mathcal{H}^{1}_{L,0}}:=\big((y,y_{\Gamma}),(y,y_{\Gamma})\big)_{\mathcal{H}^{1}_{L,0}}^{1/2}= \big[ a_{L}\big((y,y_{\Gamma}),(y,y_{\Gamma})\big)\big]^{1/2}.
  \label{norm-hL}
  \end{align}
  We note that for $(y,y_{\Gamma})\in \mathcal{V}_{(0)}^1\subseteq\mathcal{H}^{1}_{L,0}$, $\Vert(y,y_{\Gamma})\Vert_{\mathcal{H}^{1}_{L,0}}$ does not depend on $L$, since the third term in $a_L$ simply vanishes.
  The following Poincar\'{e} type inequality has been proved in \cite[Lemma A.1]{KL}:
\begin{lemma}
  There exists a constant $c_P>0$ depending only on $L\in [0,\infty)$ and $\Omega$ such that
 \begin{align}
  \|(y,y_{\Gamma})\|_{\mathcal{L}^2}\leq c_P \Vert(y,y_{\Gamma})\Vert_{\mathcal{H}^{1}_{L,0}},\quad \forall\, (y,y_{\Gamma})\in \mathcal{H}^{1}_{L,0}.
  \label{Po3}
  \end{align}
\end{lemma}
  \noindent
  Hence, for every $L\in [0,\infty)$, $\mathcal{H}^{1}_{L,0}$ is a Hilbert space with the inner product $(\cdot,\cdot)_{\mathcal{H}^{1}_{L,0}}^{1/2}$. The induced norm $\Vert\cdot\Vert_{\mathcal{H}^{1}_{L,0}}$ prescribed in \eqref{norm-hL} is equivalent to the standard one $\Vert\cdot\Vert_{\mathcal{H}^{1}}$ on $\mathcal{H}^{1}_{L,0}$.
  Besides, for any fixed $\delta>0$, we define the bilinear form
  $$
 b_{\delta}\big((y,y_\Gamma),(z,z_\Gamma)\big)
:=\int_{\Omega}\nabla y \cdot\nabla z\,\mathrm{d}x
+\delta\int_{\Gamma}\nabla_{\Gamma}y \cdot\nabla_{\Gamma} z_{\Gamma}\,\mathrm{d}S,\quad \forall\, (y,y_\Gamma),(z,z_\Gamma)\in \mathcal{H}^1.
  $$
It is easy to check that
  \begin{align}
  \Vert(y,y_{\Gamma})\Vert_{\mathcal{V}_{(0)}^{1}}:=\big[b_{\delta}\big((y,y_\Gamma),(y,y_\Gamma)\big)\big]^{1/2},
  \quad \forall\, (y,y_{\Gamma})\in \mathcal{V}^{1}_{(0)},
  \label{nv01}
  \end{align}
  yields a norm equivalent to $\Vert\cdot\Vert_{\mathcal{H}^{1}}$ on $\mathcal{V}_{(0)}^{1}$.

   For $L\in[0,\infty)$, let us consider the following elliptic boundary value problem
      \begin{align}
      \left\{
      \begin{array}{rl}
      -\Delta u = y&\quad \text{in }\Omega,\\
      -\Delta_{\Gamma}u_\Gamma +\partial_{\mathbf{n}}u= y_{\Gamma}&\quad \text{on }\Gamma,\\
      L\partial_{\mathbf{n}}u =u_\Gamma-u|_{\Gamma} &\quad \mathrm{on\;}\Gamma.
      \end{array}\right.
      \label{2.2}
      \end{align}
   Define the space
   $$
   \mathcal{H}_{(0)}^{-1}:=\big\{(y,y_{\Gamma})\in(\mathcal{H}^{1})'\;:\;|\Omega|\langle y\rangle_{\Omega}+|\Gamma|\langle y_{\Gamma}\rangle_{\Gamma}=0\big\}.
   $$
   The chain of inclusions holds
   $$
   \mathcal{H}^{1}_{L,0}\subseteq \mathcal{H}^{1}_{(0)} \subset \mathcal{L}_{(0)}^2\subset \mathcal{H}_{(0)}^{-1} \subset (\mathcal{H}^{1})'\subseteq (\mathcal{H}_L^1)'\subset (\mathcal{H}^{1}_{L,0})'.
   $$
   It has been shown in \cite[Theorem 3.3]{KL} that for every $(y,y_{\Gamma})\in\mathcal{H}_{(0)}^{-1}$, problem \eqref{2.2} admits a unique weak solution $(u,u_\Gamma)\in\mathcal{H}_{L,0}^{1}$ satisfing the weak formulation
      \begin{align}
      a_L\big((u,u_{\Gamma}),(\zeta,\zeta_{\Gamma})\big) = \big\langle(y,y_{\Gamma}),(\zeta,\zeta_{\Gamma})\big\rangle_{(\mathcal{H}_L^{1})',\mathcal{H}_L^{1}},
      \quad \forall\, (\zeta,\zeta_{\Gamma})\in\mathcal{H}_L^{1},
      \notag
      \end{align}
      and the estimate
      \begin{align}
      \|(u,u_{\Gamma})\|_{\mathcal{H}^1}\leq C\|(y,y_{\Gamma})\|_{(\mathcal{H}^{1}_L)'},\label{2.2a}
      \end{align}
      for some constant $C>0$ depending only on $L$ and $\Omega$. Furthermore, if the domain $\Omega$ is of class $C^{k+2}$ and $(y,y_{\Gamma})\in \mathcal{H}_{L,0}^{k}$, $k\in \mathbb{N}$, the following regularity estimate holds
      \begin{align}
      \|(u,u_{\Gamma})\|_{\mathcal{H}^{k+2}}\leq C\|(y,y_{\Gamma})\|_{\mathcal{H}^{k}}.
      \label{2.2b}
      \end{align}
      The above facts enable us to define the solution operator $$\mathfrak{S}^{L}:\mathcal{H}_{(0)}^{-1}\rightarrow\mathcal{H}_{L,0}^{1},\quad(y,y_{\Gamma})\mapsto (u,u_\Gamma)=\mathfrak{S}^{L}(y,y_{\Gamma})=\big(\mathfrak{S}^{L}_{\Omega}(y,y_{\Gamma}),\mathfrak{S}^{L}_{\Gamma}(y,y_{\Gamma})\big).
      $$
      We mention that similar results for the special case $L=0$ have also been presented in \cite{CF15}. A direct calculation yields that
      $$
      \big((u,u_{\Gamma}), (z,z_\Gamma)\big)_{\mathcal{L}^2}
      =\big((u,u_{\Gamma}), \mathfrak{S}^{L}(z,z_\Gamma)\big)_{\mathcal{H}^1_{L,0}},\quad \forall\, (u,u_{\Gamma})\in \mathcal{H}_{(0)}^1,\ (z,z_\Gamma)\in \mathcal{L}^2_{(0)}.
      $$
      Thanks to \cite[Corollary 3.5]{KL}, we can introduce the inner product on $\mathcal{H}_{(0)}^{-1}$  as
      \begin{align}
      \big((y,y_{\Gamma}),(z,z_{\Gamma})\big)_{0,*}&:=\big(\mathfrak{S}^{L}(y,y_{\Gamma}),\mathfrak{S}^{L}(z,z_{\Gamma})\big)_{\mathcal{H}^{1}_{L,0}},
      \quad \forall\, (y,y_{\Gamma}), (z,z_{\Gamma})\in \mathcal{H}_{(0)}^{-1}.\notag
      \end{align}
      The associated norm $\Vert(y,y_{\Gamma})\Vert_{0,*} :=\big((y,y_{\Gamma}),(y,y_{\Gamma})\big)_{0,*}^{1/2}$
       is equivalent to the standard dual norm $\|\cdot\|_{(\mathcal{H}^1)'}$ on $\mathcal{H}_{(0)}^{-1}$.
      For any $(y,y_{\Gamma})\in (\mathcal{H}^{1})'$, we define the generalized mean
      \begin{align}
\overline{m}(y,y_{\Gamma}):=\frac{|\Omega|\langle y\rangle_{\Omega}+|\Gamma|\langle y_{\Gamma}\rangle_{\Gamma}}{|\Omega|+|\Gamma|}.\label{gmean}
      \end{align}
      Then it follows that
      \begin{align}
      \|(y,y_{\Gamma})\|_{*}&:=\left(\Vert(y,y_{\Gamma})-\overline{m}(y,y_{\Gamma}) \mathbf{1}\Vert_{0,*}^2+ |\overline{m}(y,y_{\Gamma})|^2\right)^{1/2},
      \quad \forall\, (y,y_{\Gamma})\in (\mathcal{H}^{1})',\notag
      \end{align}
      is equivalent to the usual dual norm $\|\cdot\|_{(\mathcal{H}^1)'}$ on $(\mathcal{H}^{1})'$.

Define the closed linear subspaces $\widetilde{\mathcal{V}}_{(0)}^{k}:=\widetilde{\mathcal{V}}^{k}\cap \mathcal{L}_{(0)}^{2}$.
The following Poincar\'{e} type inequality has been proved in \cite[Lemma A.1]{CFS}:
\begin{lemma}
  There exists a constant $\widetilde{c}_P>0$ depending only on $\Omega$ such that
  \begin{align}
   \Vert y\Vert_{H}\leq \widetilde{c}_{P} \Vert\nabla y\Vert_{H},\quad\forall \,(y,y_{\Gamma})\in \widetilde{\mathcal{V}}_{(0)}^1.
   \label{Po4}
  \end{align}
\end{lemma}
\noindent
Moreover, we have the following interpolation inequality (see \cite[Lemma 2.3]{KLLM}):
\begin{lemma}
For any $\gamma>0$, there exists a constant $C_\gamma$ depending only on $L\in [0,\infty)$ and $\Omega$ such that
\begin{align}
\|y\|_H+\|y_\Gamma\|_{H_\Gamma}\leq \gamma \|\nabla y\|_{H}+ C_\gamma \|(y,y_\Gamma)\|_{0,*},\quad \forall\, (y,y_\Gamma)\in \widetilde{\mathcal{V}}_{(0)}^{1}.
\label{int-0}
\end{align}
\end{lemma}

Throughout the paper, the symbol $C$ stands for generic positive constants that may depend on $\Omega$, the final time $T$ and
on the coefficients and the norms of functions involved in the assumptions of either our statements or our approximation.
Specific dependence will be pointed out if necessary. Besides, we use different symbols to denote precise constants that we would refer to.

\subsection{Problem setting}
In what follows, we assume that $\Omega\subset \mathbb{R}^d$ ($d\in\{2,3\}$) is a bounded domain of class $C^2$, $T\in (0,\infty)$
is an arbitrary but fixed final time, and we set $Q:=\Omega\times (0,T)$, $\Sigma:=\Gamma\times (0,T)$.

For $L\in [0,\infty]$, $\delta \in [0,\infty)$, using \eqref{trace1}, we can rewrite our target problem \eqref{1.1}--\eqref{1.7} as follows:
\begin{align}
&\partial_{t}\varphi=\Delta\mu,&\text{in }Q,
 \label{2.1a}\\
&\mu=-\Delta\varphi+\xi+\pi(\varphi)-f,\quad\xi\in\beta(\varphi),&\text{in }Q,
 \label{2.1b}\\
&\partial_{t}\psi=\Delta_{\Gamma}\theta-\partial_{\mathbf{n}}\mu,&\text{on }\Sigma,
 \label{2.1c}\\
&\theta=\partial_{\mathbf{n}}\varphi-\delta\Delta_{\Gamma}\psi+\xi_{\Gamma}+\pi_{\Gamma}(\psi)-f_\Gamma,\quad\xi_{\Gamma}\in\beta_{\Gamma}(\psi),&\text{on }\Sigma,\label{2.1d} \\
&\varphi=\psi,&\text{on }\Sigma,
 \label{2.1e}\\
&\begin{cases}
L\partial_{\mathbf{n}}\mu=\theta-\mu,\qquad \text{if}\ L\in[0,\infty),\\
\partial_{\mathbf{n}}\mu=0,\qquad\qquad\ \   \text{if}\ L=\infty,
\end{cases}
&\text{on }\Sigma,
\label{2.1f} \\
&\varphi |_{t=0}= \varphi_{0}, &\text{in }\Omega,
\label{2.1g} \\
&\psi |_{t=0}= \psi_{0}, &\text{on }\Gamma,
\label{2.1h}
\end{align}
where $f:Q\rightarrow\mathbb{R}$, $f_\Gamma:\Sigma\rightarrow\mathbb{R}$, $\varphi_{0}:\Omega\rightarrow\mathbb{R}$, $\psi_{0}:\Gamma\rightarrow\mathbb{R}$
are given data. For convenience, we denote \eqref{2.1a}--\eqref{2.1h} by problem $(S_{L,\delta})$.

Hereafter, we consider the nonlinear bulk and boundary potentials $F$ and $G$ that can be decomposed as
$$
F = \widehat{\beta} + \widehat{\pi},\qquad G = \widehat{\beta}_\Gamma + \widehat{\pi}_\Gamma.
$$
Besides, we make the following assumptions:
\begin{description}
\item[$\mathbf{(A1)}$] $\widehat{\beta}, \widehat{\beta}_\Gamma: \mathbb{R} \to [0,\infty]$ are lower semicontinuous and convex functions with $\widehat{\beta}(0)=\widehat{\beta}_\Gamma(0) = 0$. Their subdifferentials
      $$
      \beta:=\partial \widehat{\beta},\quad \beta_\Gamma:=\partial \widehat{\beta}_\Gamma,
      $$
      are maximal monotone graphs in $\mathbb{R} \times \mathbb{R}$, with effective domains $D(\beta)$
and $D(\beta_\Gamma)$, respectively. Since $0$ is a minimum point of both $\widehat{\beta}, \widehat{\beta}_\Gamma$,
it follows that $0 \in  \beta(0)$ and $0\in  \beta_\Gamma(0)$.
\item[$\mathbf{(A2)}$] $ D(\beta_\Gamma)\subseteq D(\beta)$ and there exist positive constants $\varrho$, $c_{0}>0$ such that
  \begin{align}
|\beta^{\circ}(r)|\leq \varrho|\beta_{\Gamma}^{\circ}(r)|+c_{0},\quad\forall\,r\in D(\beta_\Gamma),\label{eq2.1}
  \end{align}
  where $\beta^{\circ}$ denotes the minimal section of the graph $\beta$, i.e., $\beta^{\circ}(r):=\{r^{*}\in\beta(r):\,|r^{*}|=\inf_{s\in\beta(r)}|s|\}$ and the same definition applies to $\beta_{\Gamma}^{\circ}$.
\item[$\mathbf{(A3)}$] $\widehat{\pi}, \widehat{\pi}_\Gamma \in C^1(\mathbb{R})$ and their derivatives $\pi:=\widehat{\pi}'$, $\pi_\Gamma:=\widehat{\pi}_\Gamma'$ are globally Lipschitz continuous with Lipschitz constants denoted by $K$ and $K_{\Gamma}$, respectively.
\item[$\mathbf{(A4)}$] $(f,f_\Gamma)\in L^{2}(0,T;\mathcal{V}^{1})$.
\item[$\mathbf{(A5)}$] $(\varphi_0,\psi_0)\in \mathcal{V}^1$ satisfying $ \widehat{\beta}(\varphi_0)\in L^1(\Omega)$, $\widehat{\beta}_\Gamma(\psi_0)\in L^1(\Gamma)$ as well as $\langle \varphi_0\rangle_\Omega, \langle \psi_0\rangle_\Gamma \in \mathrm{int}D(\beta_\Gamma)$.
\end{description}
\begin{remark}
All the typical examples of potentials given in \eqref{eq1.14}--\eqref{eq1.13} fulfill the assumptions $\mathbf{(A1)}$--$\mathbf{(A3)}$, provided that the boundary potential dominates the one in the bulk as demanded in \eqref{eq2.1}. The condition \eqref{eq2.1} follows \cite{CC,CF-CH,CF15,CFW,CGS14,LW,S}.
An alternative choice is that the bulk potential dominates the boundary one (see, e.g., \cite{GMS09,GMS10}).
These compatibility conditions are useful when we deal with the bulk-surface interaction and derive necessary a priori estimates.
Here, we take the classical choice of a dominating boundary potential in order to avoid further technicalities.
From $\mathbf{(A2)}$ and $\mathbf{(A5)}$, we find $\overline{m}(\varphi_0,\psi_0)\in \mathrm{int}D(\beta_\Gamma)$.
\end{remark}

For convenience, below we shall use the bold notations
$$
\bm{\varphi}=(\varphi, \psi), \quad \bm{\mu}=(\mu, \theta),\quad \bm{\xi}=(\xi,\xi_\Gamma),
\quad \bm{\pi}=(\pi, \pi_\Gamma), \quad \boldsymbol{f}=(f,f_\Gamma),\quad \bm{\varphi}_0=(\varphi_0, \psi_0),
$$
and also for generic elements $\boldsymbol{y}=(y,y_\Gamma)$ in the product spaces $\mathcal{L}^{2}$, $\mathcal{H}^1$, $(\mathcal{H}^1)'$ etc.
Owing to the solvability of the elliptic problem \eqref{2.2} for $L\in (0,\infty)$, inspired by \cite{CF-CH,CF15,Kubo},
we can formally write our problem $(S_{L,\delta})$ into a suitable abstract formulation.
To this end, we define the projection operator
$$
\mathbf{P}:\mathcal{L}^{2}\rightarrow\mathcal{L}_{(0)}^{2},\quad (y,y_{\Gamma})\mapsto(y-\overline{m},y_{\Gamma}-\overline{m}),
$$
where $\overline{m}$ is the generalized mean given by \eqref{gmean}.
Then, for any functions $\boldsymbol{z}\in \mathcal{L}^{2}$ and $\boldsymbol{y}\in \mathcal{L}_{(0)}^{2}$, it holds
\begin{align}
\big(\mathbf{P}\boldsymbol{z},\boldsymbol{y}\big)_{\mathcal{L}^{2}}=\int_{\Omega}\big(z-\overline{m}(\boldsymbol{z})\big)y\,\mathrm{d}x
+\int_{\Gamma}\big(z_{\Gamma}-\overline{m}(\boldsymbol{z})\big)y_{\Gamma}\,\mathrm{d}S
=\big(\boldsymbol{z},\boldsymbol{y}\big)_{\mathcal{L}^{2}}.
 \label{eq2.2}
\end{align}
Concerning the equations \eqref{2.1a}, \eqref{2.1c}, \eqref{2.1f}, using \eqref{2.2}, we can (formally) write $$\mathbf{P}\bm{\mu}=\mathfrak{S}^{L}(-\partial_{t}\bm{\varphi}).$$
On the other hand, consider the lower semicontinuous and convex functional $\Phi_{\delta}:\mathcal{L}_{(0)}^{2}\rightarrow[0,\infty]$
given by
\begin{align}
\Phi_{\delta}(\boldsymbol{z}):=
\left\{
\begin{array}{ll}
\dfrac{1}{2}\displaystyle{\int_{\Omega}}|\nabla z|^{2}\,\mathrm{d}x
+\dfrac{\delta}{2}\displaystyle{\int_{\Gamma}}|\nabla_{\Gamma}z_{\Gamma}|^{2}\,\mathrm{d}S,
&\quad \text{if }\boldsymbol{z}\in \mathcal{V}^{1}_{(0)},\\
+\infty,&\quad \text{otherwise}.
\end{array}\right.
 \notag
\end{align}
For any $\boldsymbol{z}\in \mathcal{V}^{1}_{(0)}$, we have $2\Phi_{\delta}(\boldsymbol{z})=b_\delta(\bm{z},\bm{z})$.
Thanks to \cite[Lemma C]{CF15}, the subdifferential $\partial\Phi_{\delta}$ on $\mathcal{L}_{(0)}^{2}$ fulfills
\begin{align}
\partial\Phi_{\delta}(\boldsymbol{z})=(-\Delta z,\partial_{\mathbf{n}}z-\delta\Delta_{\Gamma}z_{\Gamma}),\quad \forall\,\boldsymbol{z}\in D(\partial\Phi_{\delta})=\mathcal{H}^2\cap\mathcal{V}^{1}_{(0)}.
\label{Phi}
\end{align}
Set
$$\boldsymbol{\omega}=(\omega,\omega_\Gamma):=\bm{\varphi}-\overline{m}_{0}\boldsymbol{1}\quad
\text{with}\quad
\overline{m}_{0}=\overline{m}(\bm{\varphi}_0)=\frac{|\Omega|\langle \varphi_{0}\rangle_{\Omega}+|\Gamma|\langle \psi_{0}\rangle_{\Gamma}}{|\Omega|+|\Gamma|}.
$$
Then our target problem $(S_{L,\delta})$ is equivalent to the Cauchy problem for a suitable evolution equation:
\begin{align}
\left\{
\begin{array}{ll}
\mathfrak{S}^{L}\big(\boldsymbol{\omega}'(t)\big)
+\mathbf{P}\bm{\mu}(t) =\mathbf{0}\quad\text{in }\mathcal{H}_{(0)}^{1},
&\text{for a.a. }t\in(0,T),\\
\bm{\mu}(t)=\partial\Phi_{\delta}\big(\boldsymbol{\omega}(t)\big)+\boldsymbol{\xi}(t) +\boldsymbol{\pi}\big(\boldsymbol{\omega}(t)+\overline{m}_{0}\boldsymbol{1}\big)-\boldsymbol{f}(t)
\quad\text{in }\mathcal{L}^{2},&\text{for a.a. }t\in(0,T),
\\
\boldsymbol{\xi}(t)\in\boldsymbol{\beta}\big(\boldsymbol{\omega}(t)+\overline{m}_{0}\boldsymbol{1}\big)
\quad\text{in }\mathcal{L}^{2},&\text{for a.a. }t\in(0,T),\\
\boldsymbol{\omega}(0)=\boldsymbol{\omega}_{0}=(\omega_0,\omega_{\Gamma0})\quad\text{in }\mathcal{L}^{2}_{(0)},
\end{array}\right.
\label{2.4}
\end{align}
where $\boldsymbol{\beta}(\boldsymbol{z}):=\big(\beta(z),\beta_{\Gamma}(z_{\Gamma})\big)$ and $\boldsymbol{\omega}_0:= \bm{\varphi}_0-\overline{m}_{0}\boldsymbol{1}$.

\subsection{Statement of results}
In this subsection, we summarize the main results of this paper.

\subsubsection{Well-posedness for $L, \delta\in (0,\infty)$}

Let us start with the case $L, \delta\in (0,\infty)$. First, we give the definition of weak solutions.
\begin{definition}
\label{weakdefn}
Let $L, \delta\in (0,\infty)$ be fixed parameters.
Suppose that the assumptions $\mathbf{(A1)}$--$\mathbf{(A5)}$ are satisfied.
The triplet $(\boldsymbol{\varphi},\boldsymbol{\mu},\boldsymbol{\xi})$ is called a weak solution of problem $(S_{L,\delta})$ on $[0,T]$, if the following conditions are fulfilled:

(1) The functions $\boldsymbol{\varphi},\boldsymbol{\mu},\boldsymbol{\xi}$ have the regularity properties
\begin{align}
&\boldsymbol{\omega}=\bm{\varphi}-\overline{m}_{0}\boldsymbol{1} \in H^{1}\big(0,T;\mathcal{H}^{-1}_{(0)}\big)\cap L^{\infty}(0,T;\mathcal{V}^{1}_{(0)})
\cap L^{2}(0,T;\mathcal{V}^2),
\notag\\
&\boldsymbol{\mu}\in L^{2}(0,T;\mathcal{H}^{1}),\quad \boldsymbol{\xi}\in L^{2}(0,T;\mathcal{L}^{2}).
\notag
\end{align}

(2) The variational formulations
\begin{align}
&\big\langle\partial_t\boldsymbol{\varphi}(t),\boldsymbol{y}\big\rangle_{(\mathcal{H}^{1})',\mathcal{H}^{1}} +a_{L}\big(\boldsymbol{\mu}(t),\boldsymbol{y}\big)=0,
&&\forall\,\boldsymbol{y}\in\mathcal{H}^{1},
\label{eq3.2}\\
&\big(\boldsymbol{\mu}(t),\boldsymbol{z}\big)_{\mathcal{L}^{2}}=b_{\delta}\big(\boldsymbol{\varphi}(t),\boldsymbol{z}\big) +\left(\boldsymbol{\xi}(t)+\boldsymbol{\pi}\big(\boldsymbol{\varphi}(t)\big)-\boldsymbol{f}(t),\boldsymbol{z}\right)_{\mathcal{L}^{2}}, &&\forall\,\boldsymbol{z}\in\mathcal{V}^{1},
\label{eq3.3}
\end{align}
hold for a.a. $t\in(0,T)$, with
$$
\xi\in\beta(\varphi)\ \text{ a.e. in }Q,\quad\xi_{\Gamma}\in\beta_{\Gamma}(\psi)\ \text{ a.e. on }\Sigma.
$$

(3) The initial conditions are satisfied
\begin{align}
 \varphi|_{t=0}=\varphi_{0}\ \ \text{a.e. in}\ \Omega, \quad  \psi|_{t=0}=\psi_{0}\ \ \text{a.e. on}\ \Gamma.\notag
\end{align}
\end{definition}
\begin{remark}
The regularity of $\bm{\varphi}$ implies that $\bm{\varphi}\in C_w([0,T];\mathcal{V}^1)\cap C([0,T];\mathcal{L}^2)$.
Since $\bm{\varphi}\in L^{2}(0,T;\mathcal{V}^2)$, we infer from \eqref{eq3.3} that
\begin{align}
&\mu=-\Delta\varphi+\xi+\pi(\varphi)-f,&&\text{a.e. in }Q,
\label{eq3.5}\\
&\theta=\partial_{\mathbf{n}}\varphi-\delta\Delta_{\Gamma}\psi+\xi_{\Gamma}+\pi_{\Gamma}(\psi)-f_\Gamma,&&\text{a.e. on }\Sigma.
\label{eq3.6}
\end{align}
The variational equality \eqref{eq3.2} provides a representation of the time derivative $\partial_{t}\boldsymbol{\varphi}$ as an element of the dual space $L^{2}(0,T;(\mathcal{H}^{1})')$.
\end{remark}

\begin{theorem}[Existence of weak solution for $L,\delta\in (0,\infty)$]
\label{weakexist}
Suppose that the assumptions $\mathbf{(A1)}$--$\mathbf{(A5)}$ are satisfied.
For any given $L, \delta\in (0,\infty)$, problem $(S_{L,\delta})$ admits a weak solution $(\boldsymbol{\varphi},\boldsymbol{\mu},\boldsymbol{\xi})$ on $[0,T]$ in the sense of Definition \ref{weakdefn}.
\end{theorem}

Uniqueness of the phase function $\boldsymbol{\varphi}$ associated with problem $(S_{L,\delta})$ is guaranteed by the following continuous dependence estimate:
\begin{theorem}[Continuous dependence for $L,\delta\in (0,\infty)$]
\label{contidepen}
Suppose that the assumptions of Theorem \ref{weakexist} are satisfied. Let $(\boldsymbol{\varphi}_{i},\boldsymbol{\mu}_{i},\boldsymbol{\xi}_{i})$, $i=1,2$, be two weak solutions of problem $(S_{L,\delta})$ corresponding to the data $(\boldsymbol{f}_{i}$, $\boldsymbol{\varphi}_{0,i})$ with $\overline{m}(\boldsymbol{\varphi}_{0,1})=\overline{m}(\boldsymbol{\varphi}_{0,2})=\overline{m}_0$.
Then there exists a constant $C>0$, depending only on $K$, $K_{\Gamma}$, $\Omega$ and $T$, such that
\begin{align}
&\Vert\boldsymbol{\varphi}_{1}(t)-\boldsymbol{\varphi}_{2}(t)\Vert_{0,*}^{2}
+\int_{0}^{t}\Vert\boldsymbol{\varphi}_{1}(s)-\boldsymbol{\varphi}_{2}(s)\Vert_{\mathcal{V}_{(0)}^{1}}^{2}\,\mathrm{d}s
 \notag\\
&\quad \leq C\Big(\Vert\boldsymbol{\varphi}_{0,1}-\boldsymbol{\varphi}_{0,2}\Vert_{0,*}^{2} +\int_{0}^{t}\Vert\boldsymbol{f}_{1}(s)-\boldsymbol{f}_{2}(s)\Vert_{(\mathcal{V}^{1})'}^{2}\,\mathrm{d}s\Big),
 \quad\forall\,t\in[0,T].
 \label{3.19}
\end{align}
\end{theorem}

In order to prove the existence of strong solutions, some additional regularity assumptions on $\boldsymbol{\varphi}_{0}$ and $\boldsymbol{f}$ are needed:
\begin{description}
\item[$\mathbf{(A6)}$] $\boldsymbol{f}=(f,f_\Gamma)\in H^{1}(0,T;\mathcal{L}^{2})$.
\item[$\mathbf{(A7)}$] $\boldsymbol{\varphi}_{0}\in\mathcal{V}^2$ and the family $\big\{\partial\Phi_{\delta}(\boldsymbol{\varphi}_{0}-\overline{m}_{0}\boldsymbol{1})
    +\boldsymbol{\beta}_{\varepsilon}(\boldsymbol{\varphi}_{0}) +\boldsymbol{\pi}(\boldsymbol{\varphi}_{0})-\boldsymbol{f}(0):\varepsilon\in(0,\varepsilon_{0}]\big\}$ is bounded in $\mathcal{H}^{1}$ for some $\varepsilon_{0}\in(0,1)$.
    Here, $\boldsymbol{\beta}_{\varepsilon}=(\beta_{\varepsilon}, \beta_{\Gamma,\varepsilon})$ denotes the Yosida regularization for maximal monotone operators $\boldsymbol{\beta}=(\beta, \beta_\Gamma)$, see \eqref{Y1}--\eqref{Y2} in Section 3.
\end{description}

\begin{definition}
\label{strongdefn}
Let $L, \delta\in (0,\infty)$ be fixed parameters.
Suppose that the assumptions $\mathbf{(A1)}$--$\mathbf{(A7)}$ are satisfied.
The triplet $(\boldsymbol{\varphi},\boldsymbol{\mu},\boldsymbol{\xi})$ is called a strong solution of problem $(S_{L,\delta})$ on $[0,T]$,
if the following additional regularity properties
\begin{align}
&\boldsymbol{\varphi}\in W^{1,\infty}(0,T;(\mathcal{H}^{1})')\cap H^{1}(0,T;\mathcal{V}^{1})\cap L^{\infty}(0,T;\mathcal{H}^2),\notag\\
&\boldsymbol{\mu}\in L^{\infty}(0,T;\mathcal{H}^{1})\cap L^{2}(0,T;\mathcal{H}^{2}),\notag\\
&\boldsymbol{\xi}\in L^{\infty}(0,T;\mathcal{L}^{2}),\notag
\end{align}
are satisfied, and it holds
\begin{align*}
&\partial_{t}\varphi=\Delta\mu,&& \text{a.e. in }Q, \\
&\mu=-\Delta u+\xi+\pi(\varphi)-f,\quad \xi\in\beta(\varphi),&&\text{a.e. in }Q, \\
&\psi=\varphi|_{\Gamma},\quad L\partial_{\mathbf{n}}\mu=\theta-\mu|_{\Gamma},&&\text{a.e. on }\Sigma, \\
&\partial_{t}\psi=\Delta_{\Gamma}\theta-\partial_{\mathbf{n}}\mu,&&\text{a.e. on }\Sigma, \\
&\theta=\partial_{\mathbf{n}}\varphi-\delta\Delta_{\Gamma}\psi+\xi_{\Gamma}+\pi_{\Gamma}(\psi)-f_\Gamma,\quad \xi_{\Gamma}\in\beta_{\Gamma}(\psi),&&\text{a.e. on }\Sigma, \\
&\varphi(0)=\varphi_{0},&&\text{a.e. in }\Omega,\\
&\psi(0)=\psi_{0},&&\text{a.e. on }\Gamma.
\end{align*}
\end{definition}

\begin{theorem}[Existence and uniqueness of strong solution for $L,\delta\in (0,\infty)$]
\label{strongexist}
Suppose that the assumptions $\mathbf{(A1)}$--$\mathbf{(A7)}$ are satisfied.
For any given $L, \delta\in (0,\infty)$, problem $(S_{L,\delta})$ admits a unique strong solution in the sense of Definition \ref{strongdefn}.
\end{theorem}
\begin{remark}
With minor modifications in the proof (cf. \cite{CF15,CFS}), we can still obtain the conclusions of Theorems \ref{weakexist}--\ref{strongexist}, if $\mathbf{(A4)}$ is replaced by the following alternative assumption:
\begin{description}
\item[$\mathbf{(A4)}'$] $\boldsymbol{f}:=(f,f_\Gamma)\in W^{1,1}(0,T;\mathcal{L}^{2})$.
\end{description}
Note that $\mathbf{(A6)}$ implies $\mathbf{(A4)}'$.
\end{remark}

\subsubsection{Asymptotic limit: $L\in (0,\infty)$ fixed, $\delta\to 0$}

In the case of $L\in (0,\infty)$ and $\delta=0$ (i.e., without surface diffusion), well-posedness of problem $(S_{L,0})$ can be established via the asymptotic limit of problem $(S_{L,\delta})$ as $\delta\rightarrow 0$. To this end, we first introduce the notion of weak solutions to problem $(S_{L,0})$.
\begin{definition}
\label{weakdefinition}
Let $L\in (0,\infty)$ be a fixed parameter, $\delta=0$.
Suppose that the assumptions $\mathbf{(A1)}$--$\mathbf{(A4)}$ are satisfied.
Besides, we assume
\begin{description}
\item[$\mathbf{(A5)}'$] $(\varphi_0,\psi_0)\in \widetilde{\mathcal{V}}^1$ satisfying $ \widehat{\beta}(\varphi_0)\in L^1(\Omega)$, $\widehat{\beta}_\Gamma(\psi_0)\in L^1(\Gamma)$ as well as $\langle \varphi_0\rangle_\Omega, \langle \psi_0\rangle_\Gamma \in \mathrm{int}D(\beta_\Gamma)$.
\end{description}
The triplet $(\boldsymbol{\varphi},\boldsymbol{\mu},\boldsymbol{\xi})$ is called a weak solution of problem $(S_{L,0})$ on $[0,T]$, if the following conditions are fulfilled:

(1) The functions $\boldsymbol{\varphi},\boldsymbol{\mu},\boldsymbol{\xi}$ have the regularity properties
\begin{align*}
& \boldsymbol{\omega}=\bm{\varphi}-\overline{m}_{0}\boldsymbol{1}
\in H^{1}(0,T;\mathcal{H}^{-1}_{(0)})\cap L^{\infty}\big(0,T;\widetilde{\mathcal{V}}_{(0)}^1\big)\cap C_w\big([0,T];\widetilde{\mathcal{V}}^1\big),\\
& \Delta\varphi\in L^{2}(0,T;H),\quad \boldsymbol{\mu}\in L^{2}(0,T;\mathcal{H}^{1}),\\
&\boldsymbol{\xi}\in L^{2}\big(0,T;H\times (H^{1/2}(\Gamma))'\big),
\end{align*}
such that $\boldsymbol{\varphi}(0)=\boldsymbol{\varphi}_{0}$ in $\widetilde{\mathcal{V}}^1$ (by weak continuity).

(2) The variational formulations are satisfied
\begin{align}
&\big\langle \partial_{t}\bm{\varphi}(t),\bm{y}\big\rangle_{(\mathcal{H}^{1})',\mathcal{H}^{1}} +a_{L}\big(\boldsymbol{\mu}(t),\boldsymbol{y}\big)=0,
\quad\forall\,\boldsymbol{y}\in\mathcal{H}^{1},
\label{modelu}
\end{align}
and
\begin{align}
\big(\boldsymbol{\mu}(t),\boldsymbol{z}\big)_{\mathcal{L}^{2}}
&=\int_{\Omega}\nabla \varphi(t)\cdot\nabla z\,\mathrm{d}x
+(\xi(t),z)_{H}+\big\langle\xi_{\Gamma}(t),z_{\Gamma}\big\rangle_{(H^{1/2}(\Gamma))',H^{1/2}(\Gamma)}
\notag\\
&\quad
+\left(\boldsymbol{\pi}\big(\boldsymbol{\varphi}(t)\big)-\boldsymbol{f}(t),\boldsymbol{z}\right)_{\mathcal{L}^{2}},
\quad\forall\,\boldsymbol{z}\in\widetilde{\mathcal{V}}^1,
\label{modeltheta}
\end{align}
almost everywhere in $(0,T)$, and
\begin{align}
&\xi\in\beta(\varphi),\quad\text{a.e. in }Q,\label{xi}\\
&\int_{\Sigma}\widehat{\beta}_{\Gamma}(\psi) \,\mathrm{d}S\mathrm{d}t +\int_{0}^{T}\langle\xi_{\Gamma},\zeta_{\Gamma}-\psi\rangle_{(H^{1/2}(\Gamma))',H^{1/2}(\Gamma)}\,\mathrm{d}S\mathrm{d}t
\leq\int_{\Sigma}\widehat{\beta}_{\Gamma}(\zeta_{\Gamma})\,\mathrm{d}S\mathrm{d}t,
 \label{zeta}
\end{align}
for all $\zeta_{\Gamma}\in L^{2}(0,T;H^{1/2}(\Gamma))$, where the last integral is intended to be $+\infty$ whenever $\widehat{\beta}_{\Gamma}(z_{\Gamma})\notin L^{1}(\Sigma)$.
\end{definition}

\begin{remark}
We can deduce that
\begin{align}
&\mu=-\Delta\varphi+\xi+\pi(\varphi)-f\qquad\text{a.e. in }Q,
\label{mupoint}\\
&\theta=\partial_{\mathbf{n}}\varphi+\xi_\Gamma +\pi_{\Gamma}(\psi)-f_\Gamma
\quad\ \text{in }(H^{1/2}(\Gamma))'\ \text{ a.e. in }(0,T).
\label{thetapoint}
\end{align}
Indeed, \eqref{mupoint} follows from the fact $\Delta\varphi\in L^{2}(0,T;H)$ and the weak formulation \eqref{modeltheta}.
Concerning \eqref{thetapoint}, since for almost all $t\in(0,T)$ it holds that $\varphi(t)\in V$ and $\Delta \varphi(t)\in H$, we deduce from by \cite[Theorem 2.27]{BG} that $\partial_{\mathbf{n}}\varphi(t)$ is well-defined in $(H^{1/2}(\Gamma))'$. Hence, \eqref{thetapoint} can be deduced from \eqref{modeltheta} using \eqref{mupoint}.
\end{remark}

\begin{theorem}[Asymptotic limit: $L\in (0,\infty)$ fixed, $\delta\to 0$]
\label{existthm}
Let $L\in (0,\infty)$ be given. Suppose that the assumptions $\mathbf{(A1)}$--$\mathbf{(A4)}$ together with
$\mathbf{(A5)}'$ are satisfied.
We consider a family of data $\{\bm{\varphi}_0^\delta,\bm{f}^\delta\}_{\delta\in (0,1)}$ that satisfy the assumptions
$\mathbf{(A4)}$ and $\mathbf{(A5)}$. Assume in addition, there exists a constant $M>0$ such that for all $\delta\in (0,1)$,
\begin{align}
&\delta \|\nabla_\Gamma \psi_0^\delta\|_{H_\Gamma}^2 + \|\widehat{\beta}(\varphi_0^\delta)\|_{L^1(\Omega)}
+ \|\widehat{\beta}_\Gamma(\psi_0^\delta)\|_{L^1(\Gamma)}\leq M,
 \label{del-1}\\
& \|\bm{f}^\delta\|_{L^2(0,T;\mathcal{H}^1)}\leq M,
 \label{del-2}
\end{align}
and as $\delta\to 0$
\begin{align}
\bm{\varphi}^\delta_0\to \bm{\varphi}_0\quad \text{weakly in}\ \widetilde{\mathcal{V}}^1,\qquad \bm{f}^\delta\to \bm{f}\quad \text{weakly in}\  L^2(0,T;\mathcal{L}^2).
\label{del-3}
\end{align}
Let $\{(\boldsymbol{\varphi}^\delta,\boldsymbol{\mu}^\delta,\boldsymbol{\xi}^\delta)\}_{\delta\in (0,1)}$ be weak solutions of problem $(S_{L,\delta})$ corresponding to $\{\bm{\varphi}_0^\delta,\bm{f}^\delta\}_{\delta\in (0,1)}$ that are determined by Theorem \ref{weakexist}.
Then problem $(S_{L,0})$ admits a weak solution $(\boldsymbol{\varphi},\boldsymbol{\mu},\boldsymbol{\xi})$ in the sense of Definition \ref{weakdefinition} such that as $\delta\to 0$ (in the sense of a subsequence),
\begin{align*}
&\boldsymbol{\varphi}^\delta \to \boldsymbol{\varphi} && \text{strongly in}\ C([0,T];\mathcal{L}^2),\\
&\boldsymbol{\varphi}^\delta \to \boldsymbol{\varphi} && \text{weakly star in}\ H^1(0,T;(\mathcal{H}^1)')\cap L^\infty(0,T;\widetilde{\mathcal{V}}^1),\\
&\Delta \varphi^\delta \to \Delta \varphi && \text{weakly in}\ L^2(0,T;H),\\
& \bm{\mu}^\delta\to \bm{\mu} && \text{weakly in}\ L^2(0,T;\mathcal{H}^1),\\
& \bm{\xi}^\delta\to \bm{\xi} && \text{weakly in}\ L^2(0,T; H\times (H^{1}(\Gamma))'),\\
&-\delta\Delta_\Gamma\psi^\delta +\xi_\Gamma^\delta \to \xi_\Gamma && \text{weakly in}\ L^2(0,T;(H^{1/2}(\Gamma))'),\\
&\delta \bm{\varphi}^\delta \to \bm{0} && \text{strongly in}\ L^\infty(0,T; \mathcal{V}^1).
\end{align*}
\end{theorem}
\begin{remark}
For any initial data $\bm{\varphi}_0$ satisfying $\mathbf{(A5)}'$, the existence of an approximating sequence $\{\bm{\varphi}_0^\delta\}_{\delta\in (0,1)}$ as in Theorem \ref{existthm} is guaranteed by \cite[Proposition A.2]{CFS}.
\end{remark}

Uniqueness of the phase function $\bm{\varphi}$ associated with problem $(S_{L,0})$ is ensured by the following continuous dependence estimate.
\begin{theorem}[Continuous dependence, $L\in (0,\infty)$, $\delta=0$]
\label{continuousdepen}
Let $L\in (0,\infty)$ be given.
Suppose that the assumptions $\mathbf{(A1)}$--$\mathbf{(A3)}$ are satisfied.
Let $(\boldsymbol{\varphi}_i,\boldsymbol{\mu}_i$, $\boldsymbol{\xi}_i)$, $i=1,2$, be two weak solutions of problem $(S_{L,0})$ corresponding to the data $\{(\boldsymbol{\varphi}_{0,i},\boldsymbol{f}_i)\}_{i=1,2}$ satisfying $\mathbf{(A4)}$ and $\mathbf{(A5)}'$ with $\overline{m}(\boldsymbol{\varphi}_{0,1})=\overline{m}(\boldsymbol{\varphi}_{0,2}) =\overline{m}_0$. Then there exists a constant $C>0$, depending only on $K$, $K_{\Gamma}$, $\Omega$ and $T$, such that
\begin{align}
&\Vert\boldsymbol{\varphi}_{1}(t)-\boldsymbol{\varphi}_{2}(t)\Vert_{0,*}^{2}
+\int_{0}^{t}\Vert\boldsymbol{\varphi}_{1}(s)-\boldsymbol{\varphi}_{2}(s)\Vert_{\widetilde{\mathcal{V}}_{(0)}^{1}}^{2}\,\mathrm{d}s
 \notag\\
&\quad \leq C\Big(\Vert\boldsymbol{\varphi}_{0,1}-\boldsymbol{\varphi}_{0,2}\Vert_{0,*}^{2} +\int_{0}^{t}\Vert\boldsymbol{f}_{1}(s)-\boldsymbol{f}_{2}(s)\Vert_{(\widetilde{\mathcal{V}}^{1})'}^{2}\,\mathrm{d}s\Big),
 \quad\forall\,t\in[0,T].
 \label{conti-d0}
\end{align}
\end{theorem}
\begin{remark}
Thanks to Theorem \ref{continuousdepen}, every convergent subsequence $\{\bm{\varphi}^{\delta_k}\}$ in Theorem \ref{existthm} converges to the same limit $\bm{\varphi}$. In the case when the two graphs $\beta$, $\beta_\Gamma$ exhibit the same growth, we can obtain further results on refined convergence and  error estimate, see Corollary \ref{Cor-d}.
\end{remark}

\subsubsection{Asymptotic limit: $\delta\in(0,\infty)$ fixed, $L\to 0$ or $L\to \infty$}

Problem $(S_{L,\delta})$ with $L\in (0,\infty)$ can be viewed as an interpolation of the GMS model (with $L=0$) and the LW model (with $L=\infty$).
Assuming $\delta>0$, below we provide a rigorous investigation of the asymptotic limits as $L\rightarrow 0$ and $L\rightarrow \infty$, respectively.

First, we introduce the definition of weak solutions to the GMS model (cf. \cite[Definition 2.1]{CF15}).
\begin{definition}
\label{GMS}
Let $\delta\in (0,\infty)$ be a fixed parameter, $L=0$.
Suppose that $\mathbf{(A1)}$--$\mathbf{(A5)}$ are satisfied.
The triplet $(\boldsymbol{\varphi},\boldsymbol{\mu},\boldsymbol{\xi})$ is called a weak solution of problem $(S_{0,\delta})$ on $[0,T]$, if the following conditions are fulfilled:

(1) The functions $\boldsymbol{\varphi},\boldsymbol{\mu},\boldsymbol{\xi}$ have the regularity properties
\begin{align}
&\bm{\omega}= \boldsymbol{\varphi}- \overline{m}_{0}\boldsymbol{1} \in H^{1}\big(0,T;\mathcal{H}^{-1}_{(0)}\big)\cap L^{\infty}(0,T;\mathcal{V}^{1}_{(0)})\cap L^{2}(0,T;\mathcal{V}^2),
\notag\\
&\boldsymbol{\mu}\in L^{2}(0,T;\mathcal{V}^{1}),
\quad
\boldsymbol{\xi}\in L^{2}(0,T;\mathcal{L}^{2}).
 \notag
\end{align}

(2) The variational formulations
\begin{align}
&\big\langle\partial_t\boldsymbol{\varphi}(t),\boldsymbol{y}\big\rangle_{(\mathcal{H}^{1})',\mathcal{H}^{1}} +a_{0}\big(\boldsymbol{\mu}(t),\boldsymbol{y}\big)=0,
\quad\forall\,\boldsymbol{y}\in\mathcal{V}^{1},
\label{eq3.2-0}\\
&\big(\boldsymbol{\mu}(t),\boldsymbol{z}\big)_{\mathcal{L}^{2}}=b_{\delta}\big(\boldsymbol{\varphi}(t),\boldsymbol{z}\big) +\left(\boldsymbol{\xi}(t)+\boldsymbol{\pi}\big(\boldsymbol{\varphi}(t)\big)-\boldsymbol{f}(t),\boldsymbol{z}\right)_{\mathcal{L}^{2}}, \quad\forall\,\boldsymbol{z}\in\mathcal{V}^{1},\label{eq3.3-0}
\end{align}
hold for a.a. $t\in(0,T)$, with
$$
\xi\in\beta(\varphi)\ \text{ a.e. in }Q,\quad\xi_{\Gamma}\in\beta_{\Gamma}(\psi)\ \text{ a.e. on }\Sigma.
$$

(3) The initial conditions are satisfied
\begin{align}
 \varphi|_{t=0}=\varphi_{0}\ \ \text{a.e. in}\ \Omega, \quad  \psi|_{t=0}=\psi_{0}\ \ \text{a.e. on}\ \Gamma.\notag
\end{align}
\end{definition}

\begin{theorem}[Asymptotic limit: $\delta\in(0,\infty)$ fixed, $L\to 0$]
\label{asymptotic0}
Let $\delta\in (0,\infty)$ be given. Suppose that the assumptions $\mathbf{(A1)}$--$\mathbf{(A5)}$ are satisfied.
For every $L\in (0,1)$, let $(\boldsymbol{\varphi}^{L},\boldsymbol{\mu}^{L},\boldsymbol{\xi}^{L})$ be a weak solution of problem $(S_{L,\delta})$ corresponding to the data $(\boldsymbol{\varphi}_{0}, \bm{f})$. Then there exist a triplet $(\boldsymbol{\varphi}^{0},\boldsymbol{\mu}^{0},\boldsymbol{\xi}^{0})$ such that as $L\rightarrow0$ (in the sense of a subsequence),
\begin{align}
\boldsymbol{\varphi}^{L}&\to\boldsymbol{\varphi}^{0}
&&\text{weakly in }H^{1}(0,T;(\mathcal{V}^{1})'),\label{7.20}\\
& &&\text{weakly star in } L^{\infty}(0,T;\mathcal{V}^{1}),\label{7.21}\\
& &&\text{strongly in }C([0,T];\mathcal{L}^{2}),\label{7.22}\\
\boldsymbol{\mu}^{L}&\to\boldsymbol{\mu}^{0}&&\text{weakly in }L^{2}(0,T;\mathcal{H}^{1}),\label{7.25}\\
\boldsymbol{\xi}^{L}&\to\boldsymbol{\xi}^{0}&&\text{weakly in }L^{2}(0,T;\mathcal{L}^{2}),\label{7.28}
\end{align}
with
\begin{align}
\Vert\theta^{L}-\mu^{L}|_\Gamma\Vert_{L^{2}(0,T;H_{\Gamma})}\leq C\sqrt{L}.
\label{7.30}
\end{align}
The limit triplet $(\boldsymbol{\varphi}^{0},\boldsymbol{\mu}^{0},\boldsymbol{\xi}^{0})$ is a weak solution of the GMS model $(S_{0,\delta})$ corresponding to the data $(\boldsymbol{\varphi}_{0}, \bm{f})$ in the sense of Definition \ref{GMS}.
\end{theorem}
\begin{remark}
Thanks to \cite[Theorem 2.1]{CF15}, the limit function $\boldsymbol{\varphi}^{0}$, as a solution to problem $(S_{0,\delta})$, is unique.
Hence, every convergent subsequence $\{\bm{\varphi}^{L_k}\}$ in Theorem \ref{asymptotic0} converges to the same limit $\bm{\varphi}^0$.
\end{remark}

Next, we introduce the definition of weak solutions to the LW model (cf. \cite[Definition 2.1]{CFW}).

\begin{definition}
\label{LW}
Let $\delta\in (0,\infty)$ be a fixed parameter, $L=\infty$.
Suppose that the assumptions $\mathbf{(A1)}$--$\mathbf{(A5)}$ are satisfied.
The triplet $(\boldsymbol{\varphi},\boldsymbol{\mu},\boldsymbol{\xi})$ is called a weak solution of problem $(S_{\infty,\delta})$ on $[0,T]$,
if the functions $\boldsymbol{\varphi},\boldsymbol{\mu},\boldsymbol{\xi}$ have the regularity properties
\begin{align}
&\varphi\in H^{1}(0,T;V')\cap L^{\infty}(0,T;V)\cap L^{2}(0,T;H^{2}(\Omega)),\notag\\
&\psi\in H^{1}(0,T;V_{\Gamma}')\cap L^{\infty}(0,T;V_{\Gamma})\cap L^{2}(0,T;H^{2}(\Gamma)),\notag\\
&\mu\in L^{2}(0,T;V),\quad\xi\in L^{2}(0,T;H),\notag\\
&\theta\in L^{2}(0,T;V_{\Gamma}),\quad\xi_{\Gamma}\in L^{2}(0,T;H_{\Gamma}),\notag
\end{align}
and they satisfy
\begin{align}
&\big\langle\partial_{t}\varphi,z\big\rangle_{V',V}+\int_{\Omega}\nabla\mu\cdot\nabla z\,\mathrm{d}x=0,
\quad\forall\, z\in V, &&\text{a.e. in }(0,T),
\notag\\
&\mu=-\Delta\varphi+\xi+\pi(\varphi)-f,\quad\xi\in\beta(\varphi),
&&\text{a.e. in }Q,
\notag\\
&\varphi|_{\Gamma}=\psi,
&&\text{a.e. on }\Sigma,
\notag\\
&\big\langle\partial_{t}\psi,z_{\Gamma}\big\rangle_{V_{\Gamma}',V_{\Gamma}}
+\int_{\Gamma}\nabla_{\Gamma}\theta\cdot\nabla_{\Gamma}z_{\Gamma}\,\mathrm{d}S=0,
\quad\forall\,z_{\Gamma}\in V_{\Gamma},
&&\text{a.e. in }(0,T),
\notag\\
&\theta=\partial_{\mathbf{n}}\varphi-\delta \Delta_{\Gamma}\psi
+\xi_{\Gamma}+\pi_{\Gamma}(\psi)-f_\Gamma,\quad\xi_{\Gamma}\in\beta_{\Gamma}(\psi),
&&\text{a.e. on }\Sigma,
\notag\\
&\varphi(0)=\varphi_{0}\quad\text{a.e. in }\Omega,\qquad\psi(0)=\psi_{0}\quad\text{a.e. on }\Gamma.
\notag
\end{align}
\end{definition}

\begin{theorem}[Asymptotic limit: $\delta\in(0,\infty)$ fixed, $L\to \infty$]
\label{asymptoticinfinity}
Let $\delta\in (0,\infty)$ be given.
Suppose that the assumptions $\mathbf{(A1)}$--$\mathbf{(A5)}$ are satisfied with $\pi, \pi_\Gamma\in W^{1,\infty}(\mathbb{R})$.
For every $L\in (1,\infty)$, let $(\boldsymbol{\varphi}^{L},\boldsymbol{\mu}^{L},\boldsymbol{\xi}^{L})$ be a weak solution of problem $(S_{L,\delta})$ corresponding to the data $(\boldsymbol{\varphi}_{0},\bm{f})$. Then there exist a triplet $(\boldsymbol{\varphi}^{\infty},\boldsymbol{\mu}^{\infty},\boldsymbol{\xi}^{\infty})$ such that as $L\to \infty$ (in the sense of a subsequence),
\begin{align*}
\boldsymbol{\varphi}^{L}&\to \boldsymbol{\varphi}^{\infty}
&&\text{weakly in }H^{1}(0,T;(\mathcal{V}^{1})'),\\
& &&\text{weakly star in }L^{\infty}(0,T;\mathcal{V}^{1}),\\
& &&\text{strongly in }C([0,T];\mathcal{L}^{2}),\\
\boldsymbol{\mu}^{L}&\to\boldsymbol{\mu}^{\infty}&&\text{weakly in }L^{2}(0,T;\mathcal{H}^{1}),\\
\boldsymbol{\xi}^{L}&\to\boldsymbol{\xi}^{\infty}&&\text{weakly in }L^{2}(0,T;\mathcal{L}^{2}),
\end{align*}
and
\begin{align*}
\frac{1}{L}\|\theta^{L}-\mu^{L}|_{\Gamma}\|_{L^{2}(0,T;H_{\Gamma})} \leq \frac{C}{\sqrt{L}}.
\end{align*}
The limit triplet $(\boldsymbol{\varphi}^{\infty},\boldsymbol{\mu}^{\infty},\boldsymbol{\xi}^{\infty})$ is a weak solution of the LW model corresponding to the data $(\boldsymbol{\varphi}_{0},\bm{f})$ in the sense of Definition \ref{LW}.
\end{theorem}
\begin{remark}
Thanks to \cite[Theorem 2.4]{CFW}, the limit function $\boldsymbol{\varphi}^{\infty}$, as a solution to problem $(S_{\infty,\delta})$, is unique.
Hence, every convergent subsequence $\{\bm{\varphi}^{L_k}\}$ in Theorem \ref{asymptoticinfinity} converges to the same limit $\bm{\varphi}^\infty$.
\end{remark}

\section{Well-posedness in the presence of surface diffusion}
\label{sec3}
\setcounter{equation}{0}
In this section, we focus on the case when the surface diffusion on $\Gamma$ is present, that is, $\delta>0$. Then we prove Theorem \ref{weakexist}--Theorem \ref{strongexist} on the well-poseness of problem $(S_{L,\delta})$ for any given $L, \delta\in (0,\infty)$.

\subsection{Approximation via Yosida regularization}
To prove the existence of weak solutions for problem $(S_{L,\delta})$ with general potentials, we approximate the maximal
monotone operators $\beta$ and $\beta_\Gamma$ by means of suitable Yosida regularizations, see, e.g., \cite{CF15,CFW} for the case of Cahn-Hilliard equations with dynamic boundary conditions.
For each $\varepsilon\in(0,1)$, we define $\beta_{\varepsilon}$, $\beta_{\Gamma,\varepsilon}:\mathbb{R}\rightarrow\mathbb{R}$, along with the associated resolvent operators $J_{\varepsilon}$, $J_{\Gamma,\varepsilon}:\mathbb{R}\rightarrow\mathbb{R}$ by
\begin{align}
&\beta_{\varepsilon}(r):=\frac{1}{\varepsilon}\big(r-J_{\varepsilon}(r)\big):=\frac{1}{\varepsilon}\big(r-(I+\varepsilon\beta)^{-1}(r)\big),\label{Y1}\\
&\beta_{\Gamma,\varepsilon}(r):=\frac{1}{\varepsilon\varrho}\big(r-J_{\Gamma,\varepsilon}(r)\big):=\frac{1}{\varepsilon\varrho}\big(r-(I+\varepsilon\varrho\beta_{\Gamma})^{-1}(r)\big),\label{Y2}
\end{align}
for all $r\in\mathbb{R}$, where $\varrho>0$ is the constant given in the condition \eqref{eq2.1}.
The related Moreau-Yosida regularizations $\widehat{\beta}_{\varepsilon}$, $\widehat{\beta}_{\Gamma,\varepsilon}$ of $\widehat{\beta}$, $\widehat{\beta}_{\Gamma}:\mathbb{R}\rightarrow\mathbb{R}$ are given by (cf., e.g., \cite{R.E.S})
\begin{align}
&\widehat{\beta}_{\varepsilon}(r) :=\inf_{s\in\mathbb{R}}\left\{\frac{1}{2\varepsilon}|r-s|^{2} +\widehat{\beta}(s)\right\}=\frac{1}{2\varepsilon}|r-J_{\varepsilon}(r)|^{2} +\widehat{\beta}\big(J_{\varepsilon}(r)\big) =\int_{0}^{r}\beta_{\varepsilon}(s)\,\mathrm{d}s,
\notag\\
&\widehat{\beta}_{\Gamma,\varepsilon}(r) :=\inf_{s\in\mathbb{R}}\left\{\frac{1}{2\varepsilon\varrho}|r-s|^{2} +\widehat{\beta}_{\Gamma}(s)\right\}=\int_{0}^{r}\beta_{\Gamma,\varepsilon}(s)\,\mathrm{d}s \quad \forall\;r\in\mathbb{R}.
\notag
\end{align}
From $\mathbf{(A1)}$, $\mathbf{(A2)}$, we find $\beta_{\varepsilon}(0)=\beta_{\Gamma,\varepsilon}(0)=0$. Besides,
$\beta_{\varepsilon}$, $\beta_{\Gamma,\varepsilon}$ are Lipschitz continuous with Lipschitz constants $1/\varepsilon$ and $1/\varepsilon\varrho$, respectively (see \cite[Propositions 2.6 and 2.7]{Br}).
Thus, it follows that $\widehat{\beta}_{\varepsilon}$ and $\widehat{\beta}_{\Gamma,\varepsilon}$ are nonnegative convex
functions with (at most) quadratic growth.
Moreover, it holds (cf., e.g., \cite{B,R.E.S})
\begin{align}
&|\beta_{\varepsilon}(r)|\leq|\beta^{\circ}(r)|\ \ \text{for}\ r\in D(\beta),\quad|\beta_{\Gamma,\varepsilon}(r)|\leq|\beta_{\Gamma}^{\circ}(r)| \ \ \text{for}\ r\in D(\beta_\Gamma),
\label{2.6}\\
&0\leq\widehat{\beta}_{\varepsilon}(r)\leq\widehat{\beta}(r),\quad 0\leq\widehat{\beta}_{\Gamma,\varepsilon}(r)\leq\widehat{\beta}_{\Gamma}(r),
\quad\forall\,r\in\mathbb{R}.
\label{2.7}
\end{align}
Thanks to \cite[Lemma 4.4]{CC}, we keep the compatibility condition
\begin{align}
|\beta_{\varepsilon}(r)|\leq\varrho|\beta_{\Gamma,\varepsilon}(r)|+c_{0},\quad\forall\;r\in\mathbb{R},\label{2.8}
\end{align}
with the same constants $\varrho$ and $c_{0}$ as in \eqref{eq2.1}.
Additionally, the following inequalities hold for $\varepsilon\in (0,1)$ (see \cite[Section 5]{GMS09}):
for any given $r_0\in \mathrm{int}D(\beta_\Gamma)$,
\begin{align}
\beta_{\varepsilon}(r)(r-r_{0})\geq\delta_{0}|\beta_{\varepsilon}(r)|-c_{1},
\quad\beta_{\Gamma,\varepsilon}(r)(r-r_{0})\geq\delta_{0}|\beta_{\Gamma,\varepsilon}(r)|-c_{1},\quad\forall\;r\in\mathbb{R},
\label{eq2.8}
\end{align}
where $\delta_{0}$, $c_{1}$ are positive constants depending on $r_0$, but independent of $\varepsilon$.

Recalling the abstract formulation \eqref{2.4} for problem $(S_{L,\delta})$, it is straightforward to check that the first and second equations in \eqref{2.4} yield the single evolution equation:
\begin{align}
\mathfrak{S}^{L}\big(\boldsymbol{\omega}'(t)\big)+\partial\Phi_{\delta}\big(\boldsymbol{\omega}(t)\big)
=\mathbf{P}\left(-\boldsymbol{\xi}(t)-\boldsymbol{\pi}\big(\boldsymbol{\omega}(t)+\overline{m}_{0}\boldsymbol{1}\big)+\boldsymbol{f}(t)\right)
\;\mathrm{in}\;\mathcal{L}^{2}_{(0)},\quad\text{for a.a. }t\in(0,T),
\label{2.5}
\end{align}
where
$$
\boldsymbol{\xi}(t)\in\boldsymbol{\beta}\big(\boldsymbol{\omega}(t)+\overline{m}_{0}\boldsymbol{1}\big)\;\mathrm{in\;}\mathcal{L}^{2},
\quad\text{for a.a. }t\in(0,T).
$$
For each $\varepsilon\in(0,1)$, let us consider the following approximating problem for \eqref{2.5}: find $\boldsymbol{\omega}_{\varepsilon}:=(\omega_{\varepsilon},\omega_{\Gamma,\varepsilon})$ satisfying
\begin{align}
&\varepsilon\boldsymbol{\omega}_{\varepsilon}'(t) +\mathfrak{S}^{L}\big(\boldsymbol{\omega}_{\varepsilon}'(t)\big) +\partial\Phi_{\delta}\big(\boldsymbol{\omega}_{\varepsilon}(t)\big)
\notag\\
&\quad=\mathbf{P}\left(-\boldsymbol{\beta}_{\varepsilon}\big(\boldsymbol{\omega}_{\varepsilon}(t)+\overline{m}_{0}\boldsymbol{1}\big) -\boldsymbol{\pi}\big(\boldsymbol{\omega}_{\varepsilon}(t)+\overline{m}_{0}\boldsymbol{1}\big)+\boldsymbol{f}(t)\right)\;\mathrm{in}\;\mathcal{L}^{2}_{(0)},
\quad\text{for a.a. }t\in(0,T),
\label{2.9}\\
&\boldsymbol{\omega}_{\varepsilon}(0)=\boldsymbol{\omega}_{0}
:=\boldsymbol{\varphi}_{0}-\overline{m}_{0}\boldsymbol{1}\quad\text{in }\mathcal{L}_{(0)}^{2}.
\label{2.10}
\end{align}

\begin{proposition}
\label{approexist}
Assume that the assumptions of Theorem \ref{weakexist} are satisfied.
For each $\varepsilon\in(0,1)$, problem \eqref{2.9}--\eqref{2.10} admits a unique solution
$$
\boldsymbol{\omega}_{\varepsilon}\in H^{1}(0,T;\mathcal{L}_{(0)}^{2})\cap C([0,T];\mathcal{V}^{1}_{(0)})\cap L^{2}(0,T;\mathcal{V}_{(0)}^2).
$$
\end{proposition}
\begin{proof}
The proof mainly follows the argument for \cite[Proposition 4.1]{CF15}, which is based on the abstract theory of doubly nonlinear evolution inclusions \cite{CV} combined with the contraction mapping principle.

By the definition of $\mathfrak{S}^{L}$, it is straightforward to check that
\begin{align}
\big((\varepsilon I+\mathfrak{S}^{L})\boldsymbol{z},\boldsymbol{z}\big)_{\mathcal{L}^{2}}
&=\varepsilon\Vert\boldsymbol{z}\Vert_{\mathcal{L}^{2}}^{2}
+\big(\mathfrak{S}^{L}\boldsymbol{z},\boldsymbol{z}\big)_{\mathcal{L}^{2}}
 =\varepsilon\Vert\boldsymbol{z}\Vert_{\mathcal{L}^{2}}^{2}
+\big(\mathfrak{S}^{L}\boldsymbol{z},\mathfrak{S}^{L}\boldsymbol{z}\big)_{\mathcal{H}^1_{L,0}}
 \geq\varepsilon\Vert\boldsymbol{z}\Vert_{\mathcal{L}^{2}}^{2},
\quad \forall\, \boldsymbol{z}\in \mathcal{L}_{(0)}^{2}.
\notag
\end{align}
On the other hand, thanks to \eqref{2.2a}, we find
\begin{align}
\Vert(\varepsilon I+\mathfrak{S}^{L})\boldsymbol{z}\Vert_{\mathcal{L}^{2}}
&\leq \varepsilon \Vert\boldsymbol{z}\Vert_{\mathcal{L}^{2}}
+ \Vert\mathfrak{S}^{L}\boldsymbol{z}\Vert_{\mathcal{L}^{2}}
\leq  (\varepsilon+C) \Vert\boldsymbol{z}\Vert_{\mathcal{L}^{2}},
\quad \forall\, \boldsymbol{z}\in\mathcal{L}_{(0)}^{2}.
\notag
\end{align}
As a consequence, for each $\varepsilon\in (0,1)$, the operator $\varepsilon I + \mathfrak{S}^{L}$ is coercive and with linear growth in $\mathcal{L}_{(0)}^{2}$.
These facts enable us to apply the abstract theory \cite[Theorem 2.1]{CV} with the particular choices $A=\varepsilon I+\mathfrak{S}^{L}$, $B=\partial\Phi_{\delta}$, $H=\mathcal{L}_{(0)}^{2}$ and $V=\mathcal{V}^{1}_{(0)}$ therein.
Hence, for any given function $\widetilde{\boldsymbol{\omega}}\in C([0,T];\mathcal{L}_{(0)}^{2})$, we can conclude that there exists a unique
$$
\boldsymbol{\omega}\in H^{1}(0,T;\mathcal{L}_{(0)}^{2})\cap L^{\infty}(0,T;\mathcal{V}^{1}_{(0)})\subset C([0,T];\mathcal{L}_{(0)}^{2})
$$
satisfying $\partial\Phi_{\delta}(\boldsymbol{\omega})\in L^{2}(0,T;\mathcal{L}_{(0)}^{2})$ and
\begin{align}
&(\varepsilon I+\mathfrak{S}^{L})\boldsymbol{\omega}'(t)+\partial\Phi_{\delta}\big(\boldsymbol{\omega}(t)\big)
\notag\\
&\quad \ni\mathbf{P}\left(-\boldsymbol{\beta}_{\varepsilon}\big(\widetilde{\boldsymbol{\omega}}(t)+\overline{m}_{0}\boldsymbol{1}\big) -\boldsymbol{\pi}\big(\widetilde{\boldsymbol{\omega}}(t)+\overline{m}_{0}\boldsymbol{1}\big)+\boldsymbol{f}(t)\right)
\;\mathrm{in}\;\mathcal{L}^{2}_{(0)},\quad\text{for a.a. }t\in(0,T),
\notag\\
&\boldsymbol{\omega}(0)=\boldsymbol{\omega}_{0}\quad\text{in }\mathcal{L}_{(0)}^{2}.
\notag
\end{align}
In this way, we can define the map $\Psi:\widetilde{\boldsymbol{\omega}}\mapsto\boldsymbol{\omega}$ from $C([0,T];\mathcal{L}_{(0)}^{2})$ into itself.
Next, for given $\widetilde{\boldsymbol{\omega}}^{(i)}\in C([0,T];\mathcal{L}_{(0)}^{2})$, $i=1,2$, we set  $\boldsymbol{\omega}^{(i)}:=\Psi(\widetilde {\boldsymbol{\omega}}^{(i)})$. Consider the equation of their differences
\begin{align}
&\varepsilon\Big((\boldsymbol{\omega}^{(1)})'-(\boldsymbol{\omega}^{(2)})'\Big)
+\mathfrak{S}^{L}\Big((\boldsymbol{\omega}^{(1)})'-(\boldsymbol{\omega}^{(2)})'\Big)
+\partial\Phi_{\delta}\Big(\boldsymbol{\omega}^{(1)}-\boldsymbol{\omega}^{(2)}\Big)
\notag\\
&\quad=\mathbf{P}\Big(-\boldsymbol{\beta}_{\varepsilon}\big(\widetilde{\boldsymbol{\omega}}^{(1)}+\overline{m}_{0}\boldsymbol{1}\big)
+\boldsymbol{\beta}_{\varepsilon}\big(\widetilde{\boldsymbol{\omega}}^{(2)}+\overline{m}_{0}\boldsymbol{1}\big)
-\boldsymbol{\pi}\big(\widetilde{\boldsymbol{\omega}}^{(1)}+\overline{m}_{0}\boldsymbol{1}\big)
+\boldsymbol{\pi}\big(\widetilde{\boldsymbol{\omega}}^{(2)}+\overline{m}_{0}\boldsymbol{1}\big)\Big).
\notag
\end{align}
Taking the $\mathcal{L}^{2}$ inner product with $\boldsymbol{\omega}^{(1)}-\boldsymbol{\omega}^{(2)}$, we obtain
\begin{align}
&\frac{\varepsilon}{2}\frac{\mathrm{d}}{\mathrm{d}t}\Vert\boldsymbol{\omega}^{(1)}
-\boldsymbol{\omega}^{(2)}\Vert_{\mathcal{L}_{(0)}^{2}}^{2}
+\frac{1}{2}\frac{\mathrm{d}}{\mathrm{d}t}\Vert\boldsymbol{\omega}^{(1)}-\boldsymbol{\omega}^{(2)}\Vert_{0,*}^{2}
+\Big(\partial\Phi_{\delta}(\boldsymbol{\omega}^{(1)}-\boldsymbol{\omega}^{(2)}),\boldsymbol{\omega}^{(1)}-\boldsymbol{\omega}^{(2)}\Big)_{\mathcal{L}^{2}}
\notag\\
&\quad=\Big(-\boldsymbol{\beta}_{\varepsilon}(\widetilde{\boldsymbol{\omega}}^{(1)}+\overline{m}_{0}\boldsymbol{1})
+\boldsymbol{\beta}_{\varepsilon}(\widetilde{\boldsymbol{\omega}}^{(2)}+\overline{m}_{0}\boldsymbol{1})
-\boldsymbol{\pi}(\widetilde{\boldsymbol{\omega}}^{(1)}+\overline{m}_{0}\boldsymbol{1})
+\boldsymbol{\pi}(\widetilde{\boldsymbol{\omega}}^{(2)}+\overline{m}_{0}\boldsymbol{1}),\boldsymbol{\omega}^{(1)}
-\boldsymbol{\omega}^{(2)}\Big)_{\mathcal{L}^{2}}
\notag\\
&\quad\leq C_{\varepsilon}\Vert\widetilde{\boldsymbol{\omega}}^{(1)}
-\widetilde{\boldsymbol{\omega}}^{(2)}\Vert_{\mathcal{L}_{(0)}^{2}}\Vert\boldsymbol{\omega}^{(1)}
-\boldsymbol{\omega}^{(2)}\Vert_{\mathcal{L}_{(0)}^{2}},
\label{es-diff}
\end{align}
where the last inequality follows from the fact that $\beta_\varepsilon$, $\beta_{\Gamma,\varepsilon}$, $\pi$, $\pi_\Gamma$ are Lipschitz continuous, and $C_\varepsilon>0$ depends on $K$, $K_\Gamma$, $\varepsilon$.
Using the monotonicity of $\partial\Phi_{\delta}$, Cauchy-Schwarz inequality and Gronwall's lemma, we deduce from \eqref{es-diff} that
\begin{align}
\Vert\boldsymbol{\omega}^{(1)}(t)-\boldsymbol{\omega}^{(2)}(t)\Vert_{\mathcal{L}_{(0)}^{2}}^{2}
\leq C_{\varepsilon}\int_{0}^{t}\Vert\widetilde{\boldsymbol{\omega}}^{(1)}(s)-\widetilde{\boldsymbol{\omega}}^{(2)}(s)\Vert_{\mathcal{L}_{(0)}^{2}}^{2}\,\mathrm{d}s, \quad\forall\;t\in[0,T],
\notag
\end{align}
which implies
\begin{align}
\Vert\boldsymbol{\omega}^{(1)}(t)-\boldsymbol{\omega}^{(2)}(t)\Vert_{\mathcal{L}_{(0)}^{2}}^{2}
\leq C_{\varepsilon} t \Vert\widetilde{\boldsymbol{\omega}}^{(1)}-\widetilde{\boldsymbol{\omega}}^{(2)}\Vert_{C([0,T];\mathcal{L}_{(0)}^{2})}^{2}, \quad\forall\;t\in[0,T].
\notag
\end{align}
For $k\in \mathbb{Z}^+$, we set $\boldsymbol{\omega}_k^{(i)}:=\Psi^k(\widetilde {\boldsymbol{\omega}}^{(i)})$.
By iteration, we infer from the above estimates that
$$
\Vert\boldsymbol{\omega}^{(1)}_{k}(t)-\boldsymbol{\omega}^{(2)}_{k}(t)\Vert_{\mathcal{L}_{(0)}^{2}}^{2}
\leq C_{\varepsilon}\left(\frac{t^{k}}{k!}\right)\Vert\widetilde{\boldsymbol{\omega}}^{(1)}-\widetilde{\boldsymbol{\omega}}^{(2)}\Vert_{C([0,T];\mathcal{L}_{(0)}^{2})}^2, \quad\forall\;t\in[0,T].
$$
Hence, there exists some $k_0\in\mathbb{Z}^+$ large enough such that
$$
\Vert\boldsymbol{\omega}^{(1)}_{k_0}-\boldsymbol{\omega}^{(2)}_{k_0}\Vert_{C([0,T];\mathcal{L}_{(0)}^{2})}^{2}
\leq \frac12 \Vert\widetilde{\boldsymbol{\omega}}^{(1)}-\widetilde{\boldsymbol{\omega}}^{(2)}\Vert_{C([0,T];\mathcal{L}_{(0)}^{2})}^2.
$$
This yields that $\Psi^{k_0}$ is a contraction mapping from $C([0,T];\mathcal{L}_{(0)}^{2})$ into itself. Thanks to the contraction mapping principle, $\Psi^{k_0}$ admits a unique fixed point $\boldsymbol{\omega}^*\in C([0,T];\mathcal{L}_{(0)}^{2})$. It follows that  $\Psi(\boldsymbol{\omega}^*)=\Psi(\Psi^{k_0}(\boldsymbol{\omega}^*))= \Psi^{k_0}(\Psi(\boldsymbol{\omega}^*))$.
Hence, we get $\Psi(\boldsymbol{\omega}^*)=\boldsymbol{\omega}^*$ thanks to the uniqueness of $\boldsymbol{\omega}^*$. By the definition of $\Psi$, $\boldsymbol{\omega}^*$ is indeed a solution to problem \eqref{2.9}--\eqref{2.10}. Uniqueness easily follows from an estimate similar to \eqref{es-diff} and Gronwall's lemma.

For every $\varepsilon\in(0,1)$, we have shown that problem \eqref{2.9}--\eqref{2.10} admits a unique solution $\boldsymbol{\omega}_{\varepsilon}\in H^{1}(0,T;\mathcal{L}_{(0)}^{2})\cap L^{\infty}(0,T;\mathcal{V}^{1}_{(0)})$ with $\partial\Phi_{\delta}(\boldsymbol{\omega}_{\varepsilon})\in L^{2}(0,T;\mathcal{L}_{(0)}^{2})$. In view of \eqref{Phi}, we can apply the elliptic estimate \cite[Theorem 3.3]{KL} to get  $\boldsymbol{\omega}_{\varepsilon}\in L^{2}(0,T;\mathcal{V}^2_{(0)})$. Finally, observing that $ H^{1}(0,T;\mathcal{L}^{2}_{(0)})\cap L^{2}(0,T;\mathcal{V}^2_{(0)})\subset C([0,T];\mathcal{V}^{1}_{(0)})$, we arrive at the conclusion of Proposition \ref{approexist}.
\end{proof}

\subsection{Uniform estimates}
For every $\varepsilon\in(0,1)$, in view of Proposition \ref{approexist}, we set $\boldsymbol{\varphi}_{\varepsilon}:=\boldsymbol{\omega}_{\varepsilon}+\overline{m}_{0}\boldsymbol{1}$ and
\begin{align}
\boldsymbol{\mu}_{\varepsilon}(t):=\varepsilon\boldsymbol{\omega}_{\varepsilon}'(t)
+\partial\Phi_{\delta}\big(\boldsymbol{\omega}_{\varepsilon}(t)\big)
+\boldsymbol{\beta}_{\varepsilon}\big(\boldsymbol{\varphi}_{\varepsilon}(t)\big)
+\boldsymbol{\pi}\big(\boldsymbol{\varphi}_{\varepsilon}(t)\big)-\boldsymbol{f}(t),
\quad\text{for a.a. }t\in(0,T).
\notag
\end{align}
Then the evolution equation \eqref{2.9} can be written as
\begin{align}
\mathfrak{S}^{L}\big(\boldsymbol{\omega}_{\varepsilon}'(t)\big)
+\boldsymbol{\mu}_{\varepsilon}(t)-m_{\varepsilon}(t)\boldsymbol{1}
=\boldsymbol{0}\quad\mathrm{in\;}\mathcal{L}_{(0)}^{2},
\quad\text{for a.a. }t\in(0,T),
\label{2.13}
\end{align}
with
\begin{align}
m_{\varepsilon}(t):=\overline{m}\left(\boldsymbol{\beta}_{\varepsilon}\big(\boldsymbol{\varphi}_{\varepsilon}(t)\big)
+\boldsymbol{\pi}\big(\boldsymbol{\varphi}_{\varepsilon}(t)\big)-\boldsymbol{f}(t)\right),
\quad\text{for a.a. }t\in(0,T).
\label{me}
\end{align}
It follows that $\mathbf{P}\boldsymbol{\mu}_{\varepsilon}=\boldsymbol{\mu}_{\varepsilon}-m_{\varepsilon}\boldsymbol{1} \in L^{2}(0,T;\mathcal{H}_{(0)}^{1})$ and $m_{\varepsilon}\in L^{2}(0,T)$.
As a consequence, $\boldsymbol{\mu}_{\varepsilon}=(\mu_\varepsilon,\theta_\varepsilon)\in L^{2}(0,T;\mathcal{H}^{1})$ and  the following weak formulations hold:
\begin{align}
&\int_{\Omega}\partial_{t}\omega_{\varepsilon}(t)z\,\mathrm{d}x
+\int_{\Gamma}\partial_{t}\omega_{\Gamma,\varepsilon}(t)z_{\Gamma}\,\mathrm{d}S
+\int_{\Omega}\nabla\mu_{\varepsilon}(t)\cdot\nabla z\,\mathrm{d}x
\notag\\
&\quad+\int_{\Gamma}\nabla_{\Gamma}\theta_{\varepsilon}(t)\cdot\nabla_{\Gamma} z_{\Gamma}\,\mathrm{d}S
+\frac{1}{L}\int_{\Gamma}\big(\mu_{\varepsilon}(t)-\theta_{\varepsilon}(t)\big)\big(z-z_{\Gamma}\big)\,\mathrm{d}S=0,
\quad\forall\,\boldsymbol{z}\in\mathcal{H}_{(0)}^{1},
\label{2.14}
\end{align}
and
\begin{align}
&\int_{\Omega}\mu_{\varepsilon}(t)z\,\mathrm{d}x
+\int_{\Gamma}\theta_{\varepsilon}(t)z_{\Gamma}\,\mathrm{d}S
\notag\\
&\quad=\varepsilon\int_{\Omega}\partial_{t}\omega_{\varepsilon}(t)z\,\mathrm{d}x
+\varepsilon\int_{\Gamma}\partial_{t}\omega_{\Gamma,\varepsilon}(t)z_{\Gamma}\,\mathrm{d}S
+\int_{\Omega}\nabla\omega_{\varepsilon}(t)\cdot\nabla z\,\mathrm{d}x
+\delta\int_{\Gamma}\nabla_{\Gamma}\omega_{\Gamma,\varepsilon}(t)\cdot\nabla_{\Gamma}z_{\Gamma}\,\mathrm{d}S
\notag\\
&\qquad+\int_{\Omega}\left[\beta_{\varepsilon}\big(\varphi_{\varepsilon}(t)\big)
+\pi\big(\varphi_{\varepsilon}(t)\big)-f(t)\right]z\,\mathrm{d}x
\notag\\
&\qquad+\int_{\Gamma}\left[\beta_{\Gamma,\varepsilon}\big(\psi_{\varepsilon}(t)\big)
+\pi_{\Gamma}\big(\psi_{\varepsilon}(t)\big)-f_\Gamma(t)\right]z_{\Gamma}\,\mathrm{d}S,
\quad\forall\,\boldsymbol{z}\in\mathcal{V}^{1},
\label{2.15}
\end{align}
for a.a. $t\in(0,T)$.
In particular, it follows from \eqref{2.15} that
\begin{align}
&\mu_{\varepsilon}=\varepsilon\partial_{t}\omega_{\varepsilon}-\Delta\omega_{\varepsilon} +\beta_{\varepsilon}(\varphi_{\varepsilon})+\pi(\varphi_{\varepsilon})-f,
&&\text{a.e. in }Q,
\label{2.16}\\
&\theta_{\varepsilon}=\varepsilon\partial_{t}\omega_{\Gamma,\varepsilon}
+\partial_{\mathbf{n}}\omega_{\varepsilon}
-\delta\Delta_{\Gamma}\omega_{\Gamma,\varepsilon}
+\beta_{\Gamma,\varepsilon}(\psi_{\varepsilon})+\pi_{\Gamma}(\psi_{\varepsilon})-f_\Gamma,
&&\text{a.e. on }\Sigma.
\label{2.17}
\end{align}
Besides, since $\bm{\omega}'_{\varepsilon} \in L^{2}(0,T;\mathcal{L}^{2}_{(0)})$, we infer from \eqref{2.13} and the elliptic estimate \eqref{2.2b} that  $\boldsymbol{\mu}_{\varepsilon}\in L^{2}(0,T;\mathcal{H}^2)$. This allows us to deduce from \eqref{2.14} the following pointwise relations:
\begin{align}
&\partial_{t}\omega_{\varepsilon}=\Delta\mu_{\varepsilon},&&\text{a.e. in }Q,
\label{2.18}\\
&\partial_{t}\omega_{\Gamma,\varepsilon}=\Delta_{\Gamma}\theta_{\varepsilon}-\partial_{\mathbf{n}}\mu_{\varepsilon},&&\text{a.e. on }\Sigma,
\label{2.19}\\
&L\partial_{\mathbf{n}}\mu_{\varepsilon}=\theta_{\varepsilon}-\mu_{\varepsilon},&&\text{a.e. on }\Sigma.
\label{appromutheta}
\end{align}

Now we proceed to derive uniform estimates with respect to the approximating parameter $\varepsilon\in (0,1)$.
\begin{lemma}
\label{uni0}
The mass is conserved in time, that is,
\begin{align}
\overline{m}(\bm{\varphi}_\varepsilon(t))= \overline{m}(\bm{\varphi}_0)=\overline{m}_0,\quad \forall\,  t\in [0,T].
\label{con-mass}
\end{align}
\end{lemma}
\begin{proof}
The conclusion \eqref{con-mass} easily follows from the definition of $\bm{\varphi}_\varepsilon$ and the fact that $\bm{\omega}_\varepsilon \in C([0,T];\mathcal{L}_{(0)}^{2})$.
\end{proof}

\begin{lemma}
There exists a positive constant $M_{1}$, independent of $\varepsilon\in(0,1)$ and $\delta\in (0,\infty)$, such that
\begin{align}
&\varepsilon^{1/2}\Vert\boldsymbol{\omega}_{\varepsilon}\Vert_{L^{\infty}(0,T;\mathcal{L}_{(0)}^{2})}
+\Vert\boldsymbol{\omega}_{\varepsilon}\Vert_{L^{\infty}(0,T;\mathcal{H}_{(0)}^{-1})}
+\Vert\boldsymbol{\omega}_{\varepsilon}\Vert_{L^{2}(0,T;\mathcal{V}_{(0)}^{1})}
\notag\\
&\quad+\Vert\beta_{\varepsilon}(\varphi_{\varepsilon})\Vert_{L^{1}(0,T;L^{1}(\Omega))}
+\Vert\beta_{\Gamma,\varepsilon}(\psi_{\varepsilon})\Vert_{L^{1}(0,T;L^{1}(\Gamma))}
\leq M_{1},
\label{eq3.1}
\end{align}
where the norm $\|\cdot\|_{\mathcal{V}_{(0)}^1}$ is defined as in \eqref{nv01}.
\end{lemma}
\begin{proof}
We adopt the same strategy as in \cite{CC,CF15,CF-CH}.
Testing \eqref{2.9} at time $s\in(0,T)$ by $\boldsymbol{\omega}_{\varepsilon}$, using \eqref{eq2.2} and the definition of the subdifferential $\partial\Phi_\delta$ with the fact $\Phi_\delta(\boldsymbol{0})=0$, we obtain
\begin{align}
&\varepsilon\big(\boldsymbol{\omega}_{\varepsilon}'(s),\boldsymbol{\omega}_{\varepsilon}(s)\big)_{\mathcal{L}^{2}}
+\big(\mathfrak{S}^{L}(\boldsymbol{\omega}_{\varepsilon}'(s)),\boldsymbol{\omega}_{\varepsilon}(s)\big)_{\mathcal{L}^{2}}
+\Phi_\delta\big(\boldsymbol{\omega}_{\varepsilon}(s)\big)\notag\\
&\quad+\big(\boldsymbol{\beta}_{\varepsilon}(\boldsymbol{\varphi}_{\varepsilon}(s)),\boldsymbol{\omega}_{\varepsilon}(s)\big)_{\mathcal{L}^{2}}\leq\big(\boldsymbol{f}(s)-\boldsymbol{\pi}(\boldsymbol{\varphi}_{\varepsilon}(s)),\boldsymbol{\omega}_{\varepsilon}(s)\big)_{\mathcal{L}^{2}}\quad\text{for a.a. }s\in(0,T).
\label{3.1}
\end{align}
Recalling the assumptions $\mathbf{(A2)}$ and $\mathbf{(A5)}$, we find $\overline{m}_{0}\in \mathrm{int}D(\beta_{\Gamma})$.
Then we can apply the inequalities in \eqref{eq2.8} with $r_0=\overline{m}_{0}$ to get
\begin{align}
\big(\boldsymbol{\beta}_{\varepsilon}(\boldsymbol{\varphi}_{\varepsilon}(s)),\boldsymbol{\omega}_{\varepsilon}(s)\big)_{\mathcal{L}^{2}}
&=\int_{\Omega}\beta_{\varepsilon}\big(\varphi_{\varepsilon}(s)\big)\big(\varphi_{\varepsilon}(s)- \overline{m}_{0}\big)\,\mathrm{d}x
+\int_{\Gamma}\beta_{\Gamma,\varepsilon}\big(\psi_{\varepsilon}(s)\big)\big(\psi_{\varepsilon}(s)-\overline{m}_{0}\big)\,\mathrm{d}S
\notag\\
&\quad\geq\delta_{0}\int_{\Omega}|\beta_{\varepsilon}\big(\varphi_{\varepsilon}(s)\big)|\,\mathrm{d}x-c_{1}|\Omega|
+\delta_{0}\int_{\Gamma}|\beta_{\Gamma,\varepsilon}\big(\psi_{\varepsilon}(s)\big)|\,\mathrm{d}S-c_{1}|\Gamma|,
\label{3.2}
\end{align}
for a.a.$\,s\in(0,T)$.
 Furthermore, owing to the assumption $\mathbf{(A3)}$ and the interpolation inequality \eqref{int-0}, there exists a positive constant $\widetilde{M}_{1}$, depending only on $K$, $K_{\Gamma}$, $\pi(\overline{m}_{0})$, $\pi_{\Gamma}(\overline{m}_{0})$, $|\Omega|$ and $|\Gamma|$, such that
\begin{align}
&(\boldsymbol{f}(s)-\boldsymbol{\pi}(\boldsymbol{\varphi}_{\varepsilon}(s)),\boldsymbol{\omega}_{\varepsilon}(s))_{\mathcal{L}^{2}}
\notag\\
&\quad\leq    \frac{1}{2}\int_{\Omega}|f(s)|^{2}\,\mathrm{d}x
+\frac{1}{2}\int_{\Omega}|\omega_{\varepsilon}(s)|^{2}\,\mathrm{d}x
+K\int_{\Omega}|\omega_{\varepsilon}(s)|^{2}\,\mathrm{d}x
+|\pi(\overline{m}_{0})|\int_{\Omega}|\omega_{\varepsilon}(s)|\,\mathrm{d}x
\notag\\
&\qquad+\frac{1}{2}\int_{\Gamma}|f_\Gamma(s)|^{2}\,\mathrm{d}S
+\frac{1}{2}\int_{\Gamma}|\omega_{\Gamma,\varepsilon}(s)|^{2}\,\mathrm{d}S
+K_{\Gamma}\int_{\Gamma}|\omega_{\Gamma,\varepsilon}(s)|^{2}\,\mathrm{d}S
+|\pi_{\Gamma}(\overline{m}_{0})|\int_{\Gamma}|\omega_{\Gamma,\varepsilon}(s)|\,\mathrm{d}S
\notag\\
&\quad\leq\widetilde{M}_{1}\left(1+\Vert\boldsymbol{f}(s)\Vert^{2}_{\mathcal{L}^{2}}
+\Vert\boldsymbol{\omega}_{\varepsilon}(s)\Vert_{0,*}^{2}\right)
+\frac{1}{4}\Vert \nabla \omega_{\varepsilon}(s)\Vert^{2}_{H}.
\label{3.3}
\end{align}
From \eqref{3.1}--\eqref{3.3}, we obtain
\begin{align}
&\varepsilon\frac{\mathrm{d}}{\mathrm{d}s}\Vert\boldsymbol{\omega}_{\varepsilon}(s)\Vert_{\mathcal{L}^{2}}^{2}
+\frac{\mathrm{d}}{\mathrm{d}s}\Vert\boldsymbol{\omega}_{\varepsilon}(s)\Vert_{0,*}^{2}
+\frac{1}{2}\Vert \nabla \omega_{\varepsilon}(s)\Vert^{2}_{H} + \delta  \Vert\nabla_{\Gamma}\omega_{\Gamma,\varepsilon}\Vert_{H_{\Gamma}}^{2}
\notag\\
&\qquad +2\delta_{0}\int_{\Omega}|\beta_{\varepsilon}(\varphi_{\varepsilon}(s))|\,\mathrm{d}x
+2\delta_{0}\int_{\Gamma}|\beta_{\Gamma,\varepsilon}(\psi_{\varepsilon}(s))|\,\mathrm{d}S
\notag\\
&\quad\leq 2c_{1}(|\Omega|+|\Gamma|)+2\widetilde{M}_{1}\left(1+\Vert\boldsymbol{f}(s)\Vert^{2}_{\mathcal{L}^{2}}
+\Vert\boldsymbol{\omega}_{\varepsilon}(s)\Vert_{0,*}^{2}\right),
\quad\text{for a.a. }s\in(0,T),
\notag
\end{align}
which combined with Gronwall's inequality leads to the conclusion \eqref{eq3.1}.
\end{proof}

\begin{lemma}
\label{uni2}
There exists a positive constant $M_{2}$, independent of $\varepsilon,\delta\in(0,1)$, such that
\begin{align}
&\Vert\boldsymbol{\omega}_{\varepsilon}\Vert_{L^{\infty}(0,T;\mathcal{V}_{(0)}^{1})}
+\varepsilon^{1/2}\Vert\boldsymbol{\omega}_{\varepsilon}'\Vert_{L^{2}(0,T;\mathcal{L}_{(0)}^{2})}
+\Vert\boldsymbol{\omega}_{\varepsilon}'\Vert_{L^{2}(0,T;\mathcal{H}_{(0)}^{-1})}
\notag\\
&\quad+\Vert\widehat{\beta}_{\varepsilon}(\varphi_{\varepsilon})\Vert_{L^{\infty}(0,T;L^{1}(\Omega))}
+\Vert\widehat{\beta}_{\Gamma,\varepsilon}(\psi_{\varepsilon})\Vert_{L^{\infty}(0,T;L^{1}(\Gamma))}\leq M_{2},
\label{3.4}
\end{align}
where the norm $\|\cdot\|_{\mathcal{V}_{(0)}^1}$ is defined as in \eqref{nv01}.
\end{lemma}
\begin{proof}
Testing \eqref{2.9} at time $s\in(0,T)$ by $\boldsymbol{\omega}_{\varepsilon}'=\boldsymbol{\varphi}_{\varepsilon}' $, using \eqref{eq2.2} , we find
\begin{align}
&\frac{\mathrm{d}}{\mathrm{d}s}\left(\Phi_\delta(\boldsymbol{\omega}_{\varepsilon}(s))
+ \int_{\Omega}\widehat{\beta}_{\varepsilon}(\varphi_{\varepsilon}(s))\,\mathrm{d}x
+ \int_{\Gamma}\widehat{\beta}_{\Gamma,\varepsilon}(\psi_{\varepsilon}(s))\,\mathrm{d}S
+ \int_{\Omega}\widehat{\pi}(\varphi_{\varepsilon}(s))\,\mathrm{d}x
+ \int_{\Gamma}\widehat{\pi}_{\Gamma}(\psi_{\varepsilon}(s))\,\mathrm{d}S
\right)
 \notag\\
&\qquad +\varepsilon\Vert\boldsymbol{\omega}_{\varepsilon}'(s)\Vert_{\mathcal{L}^{2}}^{2}
+\Vert\boldsymbol{\omega}_{\varepsilon}'(s)\Vert_{0,*}^{2}
 \notag\\
& \quad =  \big(\mathbf{P}\boldsymbol{f}(s),\boldsymbol{\omega}_{\varepsilon}'(s)\big)_{\mathcal{L}^{2}},
\quad \text{for a.a.}\ s\in(0,T).
\label{equ1'}
\end{align}
Since $\mathfrak{S}^L \boldsymbol{\omega}_{\varepsilon}'(s) \in \mathcal{V}_{(0)}^1$ for a.a. $s\in (0,T)$, we observe that
\begin{align}
 \big(\mathbf{P}\boldsymbol{f}(s),\boldsymbol{\omega}_{\varepsilon}'(s)\big)_{\mathcal{L}^{2}}
  = \big(\mathbf{P}\boldsymbol{f}(s),\mathfrak{S}^L \boldsymbol{\omega}_{\varepsilon}'(s) \big)_{\mathcal{H}_{L,0}^1}
 \leq\frac{1}{2}\Vert\boldsymbol{\omega}_{\varepsilon}'(s)\Vert_{0,*}^{2}
+ \frac12 \|\boldsymbol{f}(s)\|_{\mathcal{V}^{1}}^2.
\label{equ1'-ho}
\end{align}
Integrating \eqref{equ1'} over $(0,t)$ with respect to $s$, we infer from  $\mathbf{(A3)}$ (i.e., $\widehat{\pi}$, $\widehat{\pi}_\Gamma$ are at most with quadratic growth), Lemma \ref{uni0}, \eqref{equ1'-ho} and \eqref{2.7} that
\begin{align}
&\frac{1}{2}\Vert\nabla\omega_{\varepsilon}(t)\Vert^{2}_{H}+ \frac{\delta}{2}\Vert\nabla\omega_{\Gamma,\varepsilon}(t)\Vert^{2}_{H_\Gamma}
+ \int_{\Omega}\widehat{\beta}_{\varepsilon}\big(\varphi_{\varepsilon}(t)\big)\,\mathrm{d}x
+ \int_{\Gamma}\widehat{\beta}_{\Gamma,\varepsilon}\big(\psi_{\varepsilon}(t)\big)\,\mathrm{d}S\notag\\
&\qquad
+ \varepsilon\int_{0}^{t}\Vert\boldsymbol{\omega}_{\varepsilon}'(s)\Vert_{\mathcal{L}^{2}}^{2}\,\mathrm{d}s
+\frac{1}{2}\int_{0}^{t}\Vert\boldsymbol{\omega}_{\varepsilon}'(s)\Vert_{0,*}^{2}\,\mathrm{d}s
\notag\\
&\quad\leq
\frac{1}{2}\Vert\boldsymbol{\omega}_{0}\Vert^{2}_{\mathcal{V}_{(0)}^{1}}
+\int_{\Omega}\widehat{\beta}(\varphi_{0})\,\mathrm{d}x
+\int_{\Gamma}\widehat{\beta}_{\Gamma}(\psi_{0})\,\mathrm{d}S
+ \int_{\Omega}\widehat{\pi}(\varphi_{0})\,\mathrm{d}x
- \int_{\Omega}\widehat{\pi}(\varphi_{\varepsilon}(s))\,\mathrm{d}x
\notag\\
&\qquad +\int_{\Gamma}\widehat{\pi}_{\Gamma}(\psi_0)\,\mathrm{d}S
- \int_{\Gamma}\widehat{\pi}_{\Gamma}(\psi_{\varepsilon}(s))\,\mathrm{d}S
+ \frac12 \int_0^t \|\boldsymbol{f}(s)\|_{\mathcal{V}^{1}}^2\,\mathrm{d}s
\notag\\
&\quad \leq \frac{1}{2}\Vert\boldsymbol{\omega}_{0}\Vert^{2}_{\mathcal{V}_{(0)}^{1}}
+\int_{\Omega}\widehat{\beta}(\varphi_{0})\,\mathrm{d}x
+\int_{\Gamma}\widehat{\beta}_{\Gamma}(\psi_{0})\,\mathrm{d}S + C(1+\|\bm{\varphi}_0\|_{\mathcal{L}^2}^2
+ \|\bm{\omega}_\varepsilon(t)+\overline{m}_{0}\boldsymbol{1}\|_{\mathcal{L}^2}^2)
\notag\\
&\qquad + \frac12 \int_0^t \|\boldsymbol{f}(s)\|_{\mathcal{V}^{1}}^2\,\mathrm{d}s,
\quad \forall\, t\in [0,T].
\label{equ2'a}
\end{align}
The interpolation inequality \eqref{int-0} implies
\begin{align}
\|\bm{\omega}_\varepsilon(t)\|_{\mathcal{L}^2}^2
 \leq \gamma\|\nabla \omega_\varepsilon(t)\|_{H}^2+C_\gamma\Vert\boldsymbol{\omega}_{\varepsilon}(t)\Vert_{0,*}^{2},
\quad \forall\, \gamma>0.
\notag
\end{align}
Then taking $\gamma$ sufficiently small (independent of $\delta$), we infer from \eqref{eq3.1}, \eqref{equ2'a} and $\mathbf{(A5)}$ that
\begin{align}
&\frac{1}{4}\Vert\nabla\omega_{\varepsilon}(t)\Vert^{2}_{H}+ \frac{\delta}{2}\Vert\nabla\omega_{\Gamma,\varepsilon}(t)\Vert^{2}_{H_\Gamma}
+ \int_{\Omega}\widehat{\beta}_{\varepsilon}\big(\varphi_{\varepsilon}(t)\big)\,\mathrm{d}x
+ \int_{\Gamma}\widehat{\beta}_{\Gamma,\varepsilon}\big(\psi_{\varepsilon}(t)\big)\,\mathrm{d}S
\notag\\
&\qquad
+ \varepsilon\int_{0}^{t}\Vert\boldsymbol{\omega}_{\varepsilon}'(s)\Vert_{\mathcal{L}^{2}}^{2}\,\mathrm{d}s
+\frac{1}{2}\int_{0}^{t}\Vert\boldsymbol{\omega}_{\varepsilon}'(s)\Vert_{0,*}^{2}\,\mathrm{d}s
\notag\\
&\quad \leq C + \frac12 \int_0^t \|\boldsymbol{f}(s)\|_{\mathcal{V}^{1}}^2\,\mathrm{d}s,
\quad \forall\, t\in [0,T],
\label{equ2'b}
\end{align}
which gives the estimate \eqref{3.4}.
\end{proof}

\begin{lemma}
\label{uni3}
There exist positive constants $M_{3}$ and $M_{4}$, independent of $\varepsilon,\delta\in(0,1)$, such that
\begin{align}
&\Vert\beta_{\varepsilon}(\varphi_{\varepsilon})\Vert_{L^{2}(0,T;L^{1}(\Omega))}+\Vert\beta_{\Gamma,\varepsilon}(\psi_{\varepsilon})\Vert_{L^{2}(0,T;L^{1}(\Gamma))}\leq M_{3},\label{3.5}\\
&\Vert m_{\varepsilon}\Vert_{L^{2}(0,T)}+\Vert\boldsymbol{\mu}_{\varepsilon}\Vert_{L^{2}(0,T;\mathcal{H}^{1})}\leq M_{4}.\label{3.6}
\end{align}
\end{lemma}
\begin{proof}
We can deduce from \eqref{3.1}, \eqref{3.2} that
\begin{align}
&\delta_{0}\int_{\Omega}|\beta_{\varepsilon}\big(\varphi_{\varepsilon}(s)\big)|\,\mathrm{d}x
+\delta_{0}\int_{\Gamma}|\beta_{\Gamma,\varepsilon}\big(\psi_{\varepsilon}(s)\big)|\mathrm{d}S
\notag\\
&\quad\leq c_{1}(|\Omega|+|\Gamma|)+\Big(\boldsymbol{f}(s)-\boldsymbol{\pi}\big(\boldsymbol{\varphi}_{\varepsilon}(s)\big)
-\varepsilon\boldsymbol{\omega}_{\varepsilon}'(s),\boldsymbol{\omega}_{\varepsilon}(s)\Big)_{\mathcal{L}^{2}}
-\big(\mathfrak{S}^{L}(\boldsymbol{\omega}_{\varepsilon}'(s)),\boldsymbol{\omega}_{\varepsilon}(s)\big)_{\mathcal{L}^{2}}
\notag\\
&\quad \leq c_{1}\big(|\Omega|+|\Gamma|\big)
+ \left(\Vert\boldsymbol{f}(s)\Vert_{\mathcal{L}^{2}}
+\Vert\boldsymbol{\pi}\big(\boldsymbol{\varphi}_{\varepsilon}(s)\big)\Vert_{\mathcal{L}^{2}}
+\varepsilon\Vert\boldsymbol{\omega}_{\varepsilon}'(s)\Vert_{\mathcal{L}_{(0)}^{2}}\right)
\Vert\boldsymbol{\omega}_{\varepsilon}(s)\Vert_{\mathcal{L}_{(0)}^{2}}\notag\\
&\qquad +\Vert\boldsymbol{\omega}_{\varepsilon}'(s)\Vert_{0,*}
\Vert\boldsymbol{\omega}_{\varepsilon}(s)\Vert_{0,*},
\quad \text{for a.a.}\ s\in(0,T).
\label{3.7'}
\end{align}
Due to Lemma \ref{uni2} and $\mathbf{(A4)}$, the right-hand side of \eqref{3.7'} is uniformly bounded in $L^2(0,T)$. This yields the estimate \eqref{3.5}.
Next, by the definition of $m_{\varepsilon}(t)$ (recall \eqref{me}), we have
\begin{align}
|m_{\varepsilon}(t)|^{2}
&\leq\frac{6}{(|\Omega|+|\Gamma|)^{2}}\Big(\Vert\beta_{\varepsilon}\big(\varphi_{\varepsilon}(t)\big)\Vert_{L^{1}(\Omega)}^{2}
+\Vert\pi\big(\varphi_{\varepsilon}(t)\big)\Vert_{L^{1}(\Omega)}^{2}
+|\Omega|\Vert f(t)\Vert_{H}^{2}
\notag\\
&\quad+\Vert\beta_{\Gamma,\varepsilon}\big(\psi_{\varepsilon}(t)\big)\Vert_{L^{1}(\Gamma)}^{2}
+\Vert\pi_{\Gamma}\big(\psi_{\varepsilon}(t)\big)\Vert_{L^{1}(\Gamma)}^{2}
+|\Gamma|\Vert f_\Gamma(t)\Vert_{H_{\Gamma}}^{2}\Big).
\notag
\end{align}
Integrating over $(0,T)$, using Lemma \ref{uni2}, \eqref{3.5} and $\mathbf{(A3)}$, $\mathbf{(A4)}$,
we obtain the first estimate in \eqref{3.6}.
On the other hand, it follows from \eqref{2.13} that
\begin{align}
\Vert\mathbf{P}\boldsymbol{\mu}_{\varepsilon}(s)\Vert_{\mathcal{H}_{L,0}^{1}}  = \Vert\boldsymbol{\omega}_{\varepsilon}'(s)\Vert_{0,*}.
\label{pmue}
\end{align}
This fact combined with the Poincar\'{e} inequality \eqref{Po3} yields
\begin{align}
\Vert\boldsymbol{\mu}_{\varepsilon}(s)\Vert_{\mathcal{H}^{1}}
&\leq\Vert\mathbf{P}\boldsymbol{\mu}_{\varepsilon}(s)\Vert_{\mathcal{H}^{1}}
+\Vert m_{\varepsilon}(s)\boldsymbol{1}\Vert_{\mathcal{H}^{1}}
\leq C\Vert\mathbf{P}\boldsymbol{\mu}_{\varepsilon}(s)\Vert_{\mathcal{H}_{L,0}^{1}}
+\Vert m_{\varepsilon}(s)\boldsymbol{1}\Vert_{\mathcal{L}^{2}}
\notag\\
&\leq C\Vert\boldsymbol{\omega}_{\varepsilon}'(s)\Vert_{0,*}+(|\Omega|+|\Gamma|)^{1/2}|m_{\varepsilon}(s)|.
\notag
\end{align}
Recalling \eqref{3.4}, we obtain the second estimate in \eqref{3.6}.
\end{proof}

\begin{lemma}
\label{uni4}
There exist positive constants $M_{5}$, $M_{6}$ and $M_{7}$, independent of $\varepsilon\in(0,1)$, such that
\begin{align}
&\Vert\beta_{\varepsilon}(\varphi_{\varepsilon})\Vert_{L^{2}(0,T;H)}+\Vert\beta_{\varepsilon}(\psi_{\varepsilon})\Vert_{L^{2}(0,T;H_{\Gamma})}\leq M_{5},
\label{3.7}\\
&\Vert\beta_{\Gamma,\varepsilon}\big(\psi_{\varepsilon}\big)\Vert_{L^{2}(0,T;H_{\Gamma})}\leq M_{6},
\label{3.8}\\
&\Vert\boldsymbol{\omega}_{\varepsilon}\Vert_{L^{2}(0,T;\mathcal{H}^2)}\leq M_{7}.
\label{3.9}
\end{align}
\end{lemma}
\begin{proof}
The proof of \eqref{3.7} follows the argument for \cite[Lemma 4.4]{CF15}, we sketch it here for completeness.
Testing \eqref{2.16}  by $\beta_{\varepsilon}(\varphi_{\varepsilon})\in L^{2}(0,T;V)$, using \eqref{2.17} and noting that $\big(\beta_{\varepsilon}(\varphi_{\varepsilon})\big)|_{\Gamma}=\beta_{\varepsilon}(\psi_{\varepsilon})\in L^{2}(0,T;V_{\Gamma})$, we find
\begin{align}
&\int_{\Omega}\beta_{\varepsilon}'\big(\varphi_{\varepsilon}(s)\big)|\nabla\omega_{\varepsilon}(s)|^{2}\,\mathrm{d}x
+\Vert\beta_{\varepsilon}\big(\varphi_{\varepsilon}(s)\big)\Vert_{H}^{2}
\notag\\
&\qquad+ \delta \int_{\Gamma}\beta_{\varepsilon}'\big(\psi_{\varepsilon}(s)\big)|\nabla_{\Gamma}\omega_{\Gamma,\varepsilon}(s)|^{2}\,\mathrm{d}S
+\int_{\Gamma}\beta_{\Gamma,\varepsilon}\big(\psi_{\varepsilon}(s)\big)\beta_{\varepsilon}\big(\psi_{\varepsilon}(s)\big)\,\mathrm{d}S
\notag\\
&\quad=\left(f(s)+\mu_{\varepsilon}(s)-\varepsilon\partial_{t}\omega_{\varepsilon}(s)-\pi\big(\varphi_{\varepsilon}(s)\big)
,\beta_{\varepsilon}\big(\varphi_{\varepsilon}(s)\big)\right)_{H}
\notag\\
&\qquad+\left(f_\Gamma(s)+\theta_{\varepsilon}(s)-\varepsilon\partial_{t}\omega_{\Gamma,\varepsilon}(s)
-\pi_{\Gamma}\big(\psi_{\varepsilon}(s)\big),\beta_{\varepsilon}\big(\psi_{\varepsilon}(s)\big)\right)_{H_{\Gamma}},
\label{equ2'}
\end{align}
for a.a. $s\in(0,T)$. Recalling the compatibility condition $(\ref{2.8})$, and noticing that $\beta_{\varepsilon}(r)$, $\beta_{\Gamma,\varepsilon}(r)$ have the same sign for all $r\in\mathbb{R}$, we get
\begin{align}
\int_{\Gamma}\beta_{\Gamma,\varepsilon}\big(\psi_{\varepsilon}(s)\big)\beta_{\varepsilon}\big(\psi_{\varepsilon}(s)\big)\,\mathrm{d}S
&=\int_{\Gamma}|\beta_{\Gamma,\varepsilon}\big(\psi_{\varepsilon}(s)\big)| |\beta_{\varepsilon}\big(\psi_{\varepsilon}(s)\big)|\,\mathrm{d}S
\notag\\
&\geq\frac{1}{\varrho}\int_{\Gamma}|\beta_{\varepsilon}\big(\psi_{\varepsilon}(s)\big)|^{2}\,\mathrm{d}S
-\frac{c_{0}}{\varrho}\int_{\Gamma}|\beta_{\varepsilon}\big(\psi_{\varepsilon}(s)\big)|\,\mathrm{d}S
\notag\\
&\geq\frac{1}{2\varrho}\int_{\Gamma}|\beta_{\varepsilon}\big(\psi_{\varepsilon}(s)\big)|^{2}\,\mathrm{d}S
-\frac{c_{0}^{2}}{2\varrho}|\Gamma|.
\notag
\end{align}
On the other hand, we observe that
\begin{align}
\int_{\Omega}\beta_{\varepsilon}'\big(\varphi_{\varepsilon}(s)\big)|\nabla\omega_{\varepsilon}(s)|^{2}\,\mathrm{d}x\geq0,
\quad\int_{\Gamma}\beta_{\Gamma,\varepsilon}'\big(\psi_{\varepsilon}(s)\big)|\nabla_{\Gamma}\omega_{\Gamma,\varepsilon}(s)|^{2}\,\mathrm{d}S\geq0.
\notag
\end{align}
By the Young's inequality and the Lipschitz continuity of $\pi$ and $\pi_{\Gamma}$, we can find a positive constant $\widetilde{M}_{5}$, independent of $\varepsilon\in(0,1)$, such that
\begin{align}
&\left(f(s)+\mu_{\varepsilon}(s)-\varepsilon\partial_{t}\omega_{\varepsilon}(s)-\pi\big(\varphi_{\varepsilon}(s)\big),
\beta_{\varepsilon}\big(\varphi_{\varepsilon}(s)\big)\right)_{H}
\notag\\
&\quad\leq\frac{1}{2}\Vert\beta_{\varepsilon}(\varphi_{\varepsilon}(s))\Vert_{H}^{2}
+\widetilde{M}_{5}\left(1+\Vert f(s)\Vert_{H}^{2}+\Vert\mu_{\varepsilon}(s)\Vert_{H}^{2}
+\Vert\omega_{\varepsilon}(s)\Vert_{H}^{2}
+\varepsilon\Vert\partial_{t}\omega_{\varepsilon}(s)\Vert_{H}^{2}\right),
\notag\\
&\left(f_\Gamma(s)+\theta_{\varepsilon}(s)-\varepsilon\partial_{t}\omega_{\Gamma,\varepsilon}(s)-\pi_{\Gamma}\big(\psi_{\varepsilon}(s)\big),
\beta_{\varepsilon}\big(\psi_{\varepsilon}(s)\big)\right)_{H_{\Gamma}}
\notag\\
&\quad \leq \frac{1}{4\varrho}\Vert\beta_{\varepsilon}(\psi_{\varepsilon}(s))\Vert_{H_{\Gamma}}^{2}
+\varrho\widetilde{M}_{5}\left(1+\Vert f_\Gamma(s)\Vert _{H_{\Gamma}}^{2}+\Vert\theta_{\varepsilon}(s)\Vert_{H_{\Gamma}}^{2}
+\Vert\omega_{\Gamma,\varepsilon}(s)\Vert_{H_{\Gamma}}^{2}
+\varepsilon\Vert\partial_{t}\omega_{\Gamma,\varepsilon}(s)\Vert_{H_{\Gamma}}^{2}\right).
\notag
\end{align}
Combining the above inequalities, integrating \eqref{equ2'} in $(0,T)$ with respect to $s$, we infer from Lemmas \ref{uni2}, \ref{uni3} that
\begin{align}
\frac{1}{2}\int_{0}^{T}\Vert\beta_{\varepsilon}\big(\varphi_{\varepsilon}(s)\big)\Vert_{H}^{2}\,\mathrm{d}s
+\frac{1}{4\varrho}\int_{0}^{T}\Vert\beta_{\varepsilon}\big(\psi_{\varepsilon}(s)\big)\Vert_{H_{\Gamma}}^{2}\,\mathrm{d}s
\leq C,
\label{be-be}
\end{align}
where $C>0$ is independent of $\varepsilon\in(0,1)$. This yields the estimate \eqref{3.7}.

By comparison in \eqref{2.16} and the estimates just proved, we have
\begin{align}
\|\Delta \omega_\varepsilon\|_{L^2(0,T;H)}\leq C,
\label{es-no1}
\end{align}
where $C>0$ is independent of $\varepsilon\in(0,1)$.
Next, testing \eqref{2.17} by $\beta_{\Gamma,\varepsilon}(\psi_{\varepsilon})\in L^{2}(0,T;V_\Gamma)$, we can deduce that
\begin{align}
&\delta \int_{\Gamma}\beta_{\Gamma,\varepsilon}'\big(\psi_{\varepsilon}(s)\big)|\nabla_{\Gamma}\omega_{\Gamma,\varepsilon}(s)|^{2}\,\mathrm{d}S
+\Vert\beta_{\Gamma,\varepsilon}\big(\psi_{\varepsilon}(s)\big)\Vert_{H_{\Gamma}}^{2}
\notag\\
&\quad=\left(f_\Gamma(s)+\theta_{\varepsilon}(s)-\varepsilon\partial_{t}\omega_{\Gamma,\varepsilon}(s)-\partial_{\mathbf{n}}\omega_{\varepsilon}(s)
-\pi_{\Gamma}\big(\psi_{\varepsilon}(s)\big),\beta_{\Gamma,\varepsilon}\big(\psi_{\varepsilon}(s)\big)\right)_{H_{\Gamma}}
\notag\\
&\quad\leq\frac{1}{2}\Vert\beta_{\Gamma,\varepsilon}\big(\psi_{\varepsilon}(s)\big)\Vert_{H_{\Gamma}}^{2}
+3\Big(\Vert f_\Gamma(s)\Vert _{H_{\Gamma}}^{2}+\Vert\theta_{\varepsilon}(s)\Vert_{H_{\Gamma}}^{2}
\notag\\
&\qquad+\varepsilon\Vert\partial_{t}\omega_{\Gamma,\varepsilon}(s)\Vert_{H_{\Gamma}}^{2}
+\Vert\partial_{\mathbf{n}}\omega_{\varepsilon}(s)\Vert_{H_{\Gamma}}^{2}+\Vert\pi_{\Gamma}\big(\psi_{\varepsilon}(s)\big)\Vert_{H_{\Gamma}}^{2}\Big)
\label{betaG}
\end{align}
for a.a. $s\in(0,T)$.
Integrating \eqref{betaG} over $(0,T)$, we infer from Lemmas \ref{uni2}, \ref{uni3} that
\begin{align}
\Vert\beta_{\Gamma,\varepsilon}\big(\psi_{\varepsilon}\big)\Vert_{L^{2}(0,T;H_{\Gamma})}
\leq C(1+\Vert\partial_{\mathbf{n}}\omega_{\varepsilon}\Vert_{L^2(0,T;H_{\Gamma})}).
\label{es-no4}
\end{align}
By comparing the terms in \eqref{2.17}, we infer that
\begin{align}
\Vert\partial_{\mathbf{n}}\omega_{\varepsilon}
-\Delta_{\Gamma}\omega_{\Gamma,\varepsilon}\Vert_{L^{2}(0,T;H_{\Gamma})}
\leq C\left(1+\frac{1}{\delta}\right)(1+\Vert\partial_{\mathbf{n}}\omega_{\varepsilon}\Vert_{L^2(0,T;H_{\Gamma})}),
\label{es-no2}
\end{align}
where the constant $C>0$ is independent of $\varepsilon\in(0,1)$.
Combining \eqref{es-no1} and \eqref{es-no2}, we can apply the elliptic estimate \eqref{2.2b} (with $L=0$, $k=0$) to get
\begin{align}
\|\bm{\omega}_\varepsilon\|_{L^2(0,T;\mathcal{H}^2)}\leq  C(1+\Vert\partial_{\mathbf{n}}\omega_{\varepsilon}\Vert_{L^2(0,T;H_{\Gamma})}),
\label{es-no3}
\end{align}
where the constant $C>0$ is independent of $\varepsilon$, but depends on $\delta$.
Thanks to the trace theorem and the Ehrling lemma (see Lemma \ref{Ehrling}), we have
\begin{align}
\Vert\partial_{\mathbf{n}}\omega_{\varepsilon}\Vert_{L^2(0,T;H_{\Gamma})}
&\leq C\Vert\omega_{\varepsilon}\Vert_{L^2(0,T;H^{2-r}(\Omega))}\quad \text{for some}\ r\in (0,1/2)\notag\\
&\leq \gamma \Vert\omega_{\varepsilon}\Vert_{L^2(0,T;H^2(\Omega))}+ C_\gamma \Vert\omega_{\varepsilon}\Vert_{L^2(0,T;V)}.
\notag
\end{align}
Hence, taking $\gamma>0$ sufficiently small, we can conclude \eqref{3.9} from \eqref{es-no3}, with some constant $M_7>0$ independent of $\varepsilon$ (but depends on $\delta$). Finally, this estimate together with \eqref{es-no4} easily yields the estimate \eqref{3.8}.
\end{proof}

\subsection{Existence and uniqueness of weak solutions}
\textbf{Proof of Theorem \ref{weakexist}.} We are in a position to prove the existence of weak solutions to problem $(S_{L,\delta})$.
This can be done by passing to the limit in the approximating problem as $\varepsilon\rightarrow 0$, following a standard compactness argument as in \cite{CF15}. Owing to the uniform estimates derived in Lemmas \ref{uni0}--\ref{uni4}, there exist a subsequence of $\varepsilon$ (not relabeled) and some limit functions $\boldsymbol{\omega}=(\omega, \omega_\Gamma)$, $\boldsymbol{\mu}=(\mu, \theta)$, $\bm{\xi}=(\xi,\xi_{\Gamma})$ and $\widetilde{m}$ such that
\begin{align}
&\boldsymbol{\omega}_{\varepsilon}\to \boldsymbol{\omega}\quad\text{weakly star in }
 L^{\infty}(0,T;\mathcal{V}_{(0)}^{1}),
\label{3.11}\\
&\boldsymbol{\omega}_{\varepsilon}\to \boldsymbol{\omega}\quad\text{weakly in }
H^{1}(0,T;\mathcal{H}_{(0)}^{-1})\cap L^{2}(0,T;\mathcal{V}^2),
\label{3.11a}\\
&\varepsilon\boldsymbol{\omega}_{\varepsilon}\to\boldsymbol{0}\quad\text{strongly in }H^{1}(0,T;\mathcal{L}^{2}),
\label{3.12}\\
&\boldsymbol{\mu}_{\varepsilon}\to\boldsymbol{\mu}\quad\text{weakly in }L^{2}(0,T;\mathcal{H}^{1}),
\label{3.13}\\
&\beta_{\varepsilon}(\varphi_{\varepsilon})\to\xi\quad\text{weakly in }L^{2}(0,T;H),
\label{3.14}\\
&\beta_{\Gamma,\varepsilon}(\psi_{\varepsilon})\to\xi_{\Gamma}\quad\text{weakly in }L^{2}(0,T;H_{\Gamma}),
\label{3.15}\\
&m_{\varepsilon}\to \widetilde{m}\quad\text{weakly in }L^{2}(0,T).\label{3.16}
\end{align}
From \eqref{3.11a}, due to the Aubin-Lions-Simon lemma (see Lemma \ref{ALS}), we find
\begin{align}
\boldsymbol{\omega}_{\varepsilon}\to\boldsymbol{\omega}
\quad\text{strongly in }C([0,T];\mathcal{L}_{(0)}^{2})\cap L^{2}(0,T;\mathcal{V}_{(0)}^{1}),
\label{3.17}
\end{align}
which implies
\begin{align}
\boldsymbol{\varphi}_{\varepsilon}\to\boldsymbol{\varphi} =\boldsymbol{\omega}+\overline{m}_{0}\boldsymbol{1}
\quad\text{strongly in }C([0,T];\mathcal{L}^{2})\cap L^{2}(0,T;\mathcal{V}^{1}).
\label{3.18}
\end{align}
The above facts also imply that
$$
\varphi(0)=\varphi_{0}\;\;\text{a.e. in }\Omega,\quad \psi(0)=\psi_{0} \;\;\text{a.e. on }\Gamma.
$$
Moreover, \eqref{3.18} and the Lipschitz continuity of $\pi$, $\pi_{\Gamma}$ ensure that
$$
\boldsymbol{\pi}(\boldsymbol{\varphi}_{\varepsilon})\to\boldsymbol{\pi}(\boldsymbol{\varphi})
\quad\text{strongly in }C([0,T];\mathcal{L}^{2}),
$$
and $\widetilde{m}=\overline{m}\big(\boldsymbol{\xi}+\boldsymbol{\pi}(\boldsymbol{\varphi})-\boldsymbol{f}\big)$ with $\boldsymbol{\xi}=(\xi,\xi_{\Gamma})$.
Due to the maximal monotonicity of $\beta$ and $\beta_{\Gamma}$,
by applying \cite[Proposition 2.2]{B} and \eqref{3.14}, \eqref{3.15}, \eqref{3.18}, we obtain
\begin{align}
\xi\in\beta(\varphi)\;\;\text{a.e. in }Q,\quad\xi_{\Gamma}\in\beta_{\Gamma}(\psi)\;\;\text{a.e. on }\Sigma.
\label{xiOG}
\end{align}
Now we are able to pass to the limit in the weak formulations \eqref{2.14}, \eqref{2.15} to recover \eqref{eq3.2}, \eqref{eq3.3}, respectively.
Hence, the limit triplet $(\boldsymbol{\varphi},\boldsymbol{\mu},\boldsymbol{\xi})$ is indeed a weak solution of problem $(S_{L,\delta})$.
\hfill $\square$
\bigskip

\textbf{Proof of Theorem \ref{contidepen}.}
In what follows, we prove the continuous dependence on the initial data and the source terms. This can be done by the standard energy method.
For $i=1,2$, let $(\boldsymbol{\varphi}_{i},\boldsymbol{\mu}_{i},\boldsymbol{\xi}_{i})$ be a weak solution to problem $(S_{L,\delta})$ corresponding to the data $(\boldsymbol{f}_{i},\boldsymbol{\omega}_{0,i})$. Set $\boldsymbol{\omega}_{i}:=\boldsymbol{\varphi}_{i}-\overline{m}_{0}\boldsymbol{1}$.
We consider the difference between \eqref{2.5}, at the time $s\in(0,T)$, for $\boldsymbol{\omega}_{1}(s)=\big(\omega_{1}(s),\omega_{\Gamma,1}(s)\big)$ and $\boldsymbol{\omega}_{2}(s)=\big(\omega_{2}(s),\omega_{\Gamma,2}(s)\big)$, that is,
\begin{align}
&\mathfrak{S}^{L}\big(\boldsymbol{\omega}_{1}'(s)-\boldsymbol{\omega}_{2}'(s)\big)
+\partial\Phi_\delta \big(\boldsymbol{\omega}_{1}(s)-\boldsymbol{\omega}_{2}(s)\big)
\notag\\
&\quad=\mathbf{P}\left(-\boldsymbol{\xi}_{1}(s)+\boldsymbol{\xi}_{2}(s)\right)
+\mathbf{P}\left(-\boldsymbol{\pi}\big(\boldsymbol{\omega}_{1}(s)+\overline{m}_{0}\boldsymbol{1}\big) +\boldsymbol{\pi}\big(\boldsymbol{\omega}_{2}(s)+\overline{m}_{0}\boldsymbol{1}\big)\right)
+\mathbf{P}\left(\boldsymbol{f}_{1}(s)-\boldsymbol{f}_{2}(s)\right).
\notag
\end{align}
Taking $\mathcal{L}^{2}$ inner product with $\boldsymbol{\omega}_{1}(s)-\boldsymbol{\omega}_{2}(s)$,
by the Lipschitz continuity of $\pi$, $\pi_{\Gamma}$, we obtain
\begin{align}
&\frac{1}{2}\frac{\mathrm{d}}{\mathrm{d}s}\Vert\boldsymbol{\omega}_{1}(s)-\boldsymbol{\omega}_{2}(s)\Vert_{0,*}^{2}
+\Vert\boldsymbol{\omega}_{1}(s)-\boldsymbol{\omega}_{2}(s)\Vert_{\mathcal{V}_{(0)}^{1}}^{2}
+ \left(\boldsymbol{\xi}_{1}(s)-\boldsymbol{\xi}_{2}(s),\boldsymbol{\omega}_{1}(s)-\boldsymbol{\omega}_{2}(s)\right)_{\mathcal{L}^{2}}
\notag\\
&\quad =
 -\left(\boldsymbol{\pi}\big(\boldsymbol{\omega}_{1}(s)+\overline{m}_{0}\boldsymbol{1}\big)
-\boldsymbol{\pi}\big(\boldsymbol{\omega}_{2}(s)+\overline{m}_{0}\boldsymbol{1}\big),
\boldsymbol{\omega}_{1}(s)-\boldsymbol{\omega}_{2}(s)\right)_{\mathcal{L}^{2}}
\notag\\
&\qquad+\left(\boldsymbol{f}_{1}(s)-\boldsymbol{f}_{2}(s),\boldsymbol{\omega}_{1}(s)-\boldsymbol{\omega}_{2}(s)\right)_{\mathcal{L}^{2}}
\notag\\
&\quad\leq(K+K_{\Gamma})\Vert\boldsymbol{\omega}_{1}(s)-\boldsymbol{\omega}_{2}(s)\Vert_{\mathcal{L}^{2}}^{2}
+\Vert\boldsymbol{f}_{1}(s)-\boldsymbol{f}_{2}(s)\Vert_{(\mathcal{V}^{1})'}
\Vert\boldsymbol{\omega}_{1}(s)-\boldsymbol{\omega}_{2}(s)\Vert_{\mathcal{V}^{1}}.
\label{3.19a}
\end{align}
Since $\boldsymbol{\omega}_{1}(s)-\boldsymbol{\omega}_{2}(s)\in \mathcal{V}_{(0)}^1$ for a.a. $s\in (0,T)$, by the Ehrling lemma and the Poincar\'e inequality \eqref{Po3}, we get
\begin{align}
\Vert\boldsymbol{\omega}_{1}(s)-\boldsymbol{\omega}_{2}(s)\Vert_{\mathcal{L}^{2}}^{2}
\leq \gamma \Vert\boldsymbol{\omega}_{1}(s)-\boldsymbol{\omega}_{2}(s)\Vert_{\mathcal{V}_{(0)}^{1}}^{2}
+C_\gamma \Vert\boldsymbol{\omega}_{1}(s)-\boldsymbol{\omega}_{2}(s)\Vert_{0,*}^{2},
\notag
\end{align}
for any $\gamma>0$.
Besides, it follows from \eqref{xiOG} and the monotonicity of $\beta$, $\beta_{\Gamma}$ that
\begin{align}
\left(\boldsymbol{\xi}_{1}(s)-\boldsymbol{\xi}_{2}(s),\boldsymbol{\omega}_{1}(s)-\boldsymbol{\omega}_{2}(s)\right)_{\mathcal{L}^{2}}
=\left(\boldsymbol{\xi}_{1}(s)-\boldsymbol{\xi}_{2}(s),\boldsymbol{\varphi}_{1}(s)-\boldsymbol{\varphi}_{2}(s)\right)_{\mathcal{L}^{2}}\geq 0,
\quad \text{for a.a.}\ s\in (0,T).
\notag
\end{align}
Inserting the above inequalities into \eqref{3.19a}, taking $\gamma$ sufficiently small and using Young's inequality, we obtain
\begin{align}
&\frac{1}{2}\frac{\mathrm{d}}{\mathrm{d}s}\Vert\boldsymbol{\omega}_{1}(s)-\boldsymbol{\omega}_{2}(s)\Vert_{0,*}^{2}
+\Vert\boldsymbol{\omega}_{1}(s)-\boldsymbol{\omega}_{2}(s)\Vert_{\mathcal{V}_{(0)}^{1}}^{2}
\notag\\
&\quad\leq\frac{1}{2}\Vert\boldsymbol{\omega}_{1}(s)-\boldsymbol{\omega}_{2}(s)\Vert_{\mathcal{V}_{(0)}^{1}}^{2}
+C\Vert\boldsymbol{\omega}_{1}(s)-\boldsymbol{\omega}_{2}(s)\Vert_{0,*}^{2}
+C\Vert\boldsymbol{f}_{1}(s)-\boldsymbol{f}_{2}(s)\Vert_{(\mathcal{V}^{1})'}^{2}.
\notag
\end{align}
Then by Gronwall's lemma, we arrive at the conclusion \eqref{3.19}.
\hfill $\square$
\bigskip

Thanks to the uniqueness of the phase function $\bm{\varphi}$ associated with problem $(S_{L,\delta})$, we are able to provide an estimate on the convergence rate for $\boldsymbol{\varphi}_{\varepsilon}\to \bm{\varphi}$ in terms of the parameter $\varepsilon$ of the Yosida approximation.
\begin{proposition}[Convergence rate]
\label{rate}
Suppose that the assumptions of Theorem \ref{weakexist} are satisfied.
Let $\boldsymbol{\varphi}_{\varepsilon}$ be the unique solution to the approximating problem \eqref{2.9}--\eqref{2.10} and $\boldsymbol{\varphi}$ be the unique weak solution to problem $(S_{L,\delta})$. Then we have
$$
\Vert\boldsymbol{\varphi}_{\varepsilon}-\boldsymbol{\varphi}\Vert_{L^{\infty}(0,T;(\mathcal{H}^{1})')}
+\Vert\boldsymbol{\varphi}_{\varepsilon}-\boldsymbol{\varphi}\Vert_{L^{2}(0,T;\mathcal{V}^{1})}
\leq C\sqrt{\varepsilon},\quad \forall\, \varepsilon\in (0,1),
$$
where the constant $C>0$ depends on $\widehat{\pi}$, $\widehat{\pi}_{\Gamma}$, $\widehat{\beta}$, $\widehat{\beta}_{\Gamma}$ and the $\mathcal{V}^{1}$-norm of the initial data $\bm{\varphi}_0$, but not on $\varepsilon$.
\end{proposition}
\begin{proof}
For any $\varepsilon_{1}, \varepsilon_{2}\in (0,1)$,
let $\boldsymbol{\omega}_{\varepsilon_{1}}$ and $\boldsymbol{\omega}_{\varepsilon_{2}}$ be the corresponding solutions to the approximating problem \eqref{2.9}--\eqref{2.10}, respectively. Recalling the relation $\boldsymbol{\varphi}_{\varepsilon_i}=\boldsymbol{\omega}_{\varepsilon_i}+\overline{m}_{0}\boldsymbol{1}$, $i=1,2$, we consider their difference $\boldsymbol{\varphi}_{*}=\boldsymbol{\varphi}_{\varepsilon_{1}}-\boldsymbol{\varphi}_{\varepsilon_{2}}$. It follows from  \eqref{2.9} that
\begin{align}
\mathfrak{S}^{L}(\boldsymbol{\varphi}_{*}')
+\partial\Phi_\delta(\boldsymbol{\varphi}_{*})
=-\varepsilon_{1}\boldsymbol{\varphi}_{\varepsilon_{1}}'
+\varepsilon_{2}\boldsymbol{\varphi}_{\varepsilon_{2}}'
+\mathbf{P}\left(-\boldsymbol{\beta}_{\varepsilon_{1}}(\boldsymbol{\varphi}_{\varepsilon_{1}})
+\boldsymbol{\beta}_{\varepsilon_{2}}(\boldsymbol{\varphi}_{\varepsilon_{2}})
-\boldsymbol{\pi}(\boldsymbol{\varphi}_{\varepsilon_{1}})+\boldsymbol{\pi}(\boldsymbol{\varphi}_{\varepsilon_{2}})\right)
\quad\text{in }\mathcal{L}_{(0)}^{2}.
\label{3.30}
\end{align}
Testing \eqref{3.30} by $\boldsymbol{\varphi}_{*}$, and integrating in $[0,t]\subset [0,T]$, we obtain
\begin{align}
\frac{1}{2}\Vert\boldsymbol{\varphi}_{*}(t)\Vert_{0,*}^{2} + \int_{0}^{t}\Vert\boldsymbol{\varphi}_{*}(s)\Vert_{\mathcal{V}_{(0)}^{1}}^{2}\,\mathrm{d}s &=\int_{0}^{t}\left(-\varepsilon_{1}\boldsymbol{\varphi}_{\varepsilon_{1}}'(s)
+\varepsilon_{2}\boldsymbol{\varphi}_{\varepsilon_{2}}'(s),\boldsymbol{\varphi}_{*}(s)\right)_{\mathcal{L}^{2}}\,\mathrm{d}s
\notag\\
&\quad +\int_{0}^{t}\left(-\boldsymbol{\beta}_{\varepsilon_{1}}\big(\boldsymbol{\varphi}_{\varepsilon_{1}}(s)\big)
+\boldsymbol{\beta}_{\varepsilon_{2}} \big(\boldsymbol{\varphi}_{\varepsilon_{2}}(s)\big),\boldsymbol{\varphi}_{*}(s)\right)_{\mathcal{L}^{2}}\,\mathrm{d}s
\notag\\
&\quad +\int_{0}^{t}\left(-\boldsymbol{\pi}\big(\boldsymbol{\varphi}_{\varepsilon_{1}}(s)\big)
+\boldsymbol{\pi}\big(\boldsymbol{\varphi}_{\varepsilon_{2}}(s)\big),\boldsymbol{\varphi}_{*}(s)\right)_{\mathcal{L}^{2}}\,\mathrm{d}s
\notag\\
&=: J_1+J_2+J_3.
\label{3.32}
\end{align}
 Using Lemma \ref{uni2}, we obtain that
\begin{align}
J_1
&=\int_{0}^{t}\left(-\varepsilon_{1}\mathfrak{S}^L\boldsymbol{\varphi}_{\varepsilon_{1}}'(s)
+\varepsilon_{2}\mathfrak{S}^L\boldsymbol{\varphi}_{\varepsilon_{2}}'(s),\boldsymbol{\varphi}_{*}(s)\right)_{\mathcal{H}_{L,0}^{1}}\,\mathrm{d}s
\notag\\
&\leq C\int_{0}^{t}\left(\varepsilon_{1} \Vert \boldsymbol{\varphi}_{\varepsilon_{1}}'(s)\|_{0,*} +\varepsilon_{2}\|\boldsymbol{\varphi}_{\varepsilon_{2}}'(s)\Vert_{0,*}\right)
\Vert\boldsymbol{\varphi}_{*}(s)\Vert_{\mathcal{V}^1}\,\mathrm{d}s
\notag\\
&\leq C(\varepsilon_{1}+\varepsilon_{2}).
\label{3.34}
\end{align}
Concerning $J_2$, we apply the argument in \cite{CG00}.
It follows from the definition of the Yosida approximation that $J_{\varepsilon}\varphi_{\varepsilon}=\varphi_{\varepsilon}-\varepsilon\beta_{\varepsilon}(\varphi_{\varepsilon})$.
From this we find
\begin{align}
&-\int_{0}^{t}\int_{\Omega}\big(\beta_{\varepsilon_{1}}(\varphi_{\varepsilon_{1}})
-\beta_{\varepsilon_{2}}(\varphi_{\varepsilon_{2}})\big)\varphi_{*}\,\mathrm{d}x\mathrm{d}s
\notag\\
&\quad=-\int_{0}^{t}\int_{\Omega}\big(\beta_{\varepsilon_{1}}(\varphi_{\varepsilon_{1}})-\beta_{\varepsilon_{2}}(\varphi_{\varepsilon_{2}})\big)
\big(J_{\varepsilon_{1}}(\varphi_{\varepsilon_{1}})-J_{\varepsilon_{2}}(\varphi_{\varepsilon_{2}})\big)\,\mathrm{d}x\mathrm{d}s
\notag \\
&\qquad-\int_{0}^{t}\int_{\Omega}\big(\beta_{\varepsilon_{1}}(\varphi_{\varepsilon_{1}})-\beta_{\varepsilon_{2}}
(\varphi_{\varepsilon_{2}})\big)\big(\varepsilon_{1}\beta_{\varepsilon_{1}}(\varphi_{\varepsilon_{1}})
-\varepsilon_{2}\beta_{\varepsilon_{2}}(\varphi_{\varepsilon_{2}})\big)\,\mathrm{d}x\mathrm{d}s.
\label{equ3'}
\end{align}
The first term on the right-hand side of \eqref{equ3'} is nonpositive by the monotonicity of $\beta$ and the fact that $\beta_{\varepsilon}(\varphi_{\varepsilon})\in\beta(J_{\varepsilon}(\varphi_{\varepsilon}))$ (cf. \cite[(2.7)]{GST}).
On the other hand, the second term on the right-hand side of \eqref{equ3'} can be estimated by H\"older's inequality and \eqref{3.7}, that is,
\begin{align}
&-\int_{0}^{t}\int_{\Omega}\big(\beta_{\varepsilon_{1}}(\varphi_{\varepsilon_{1}})
-\beta_{\varepsilon_{2}}(\varphi_{\varepsilon_{2}})\big)\big(\varepsilon_{1}\beta_{\varepsilon_{1}}
(\varphi_{\varepsilon_{1}})-\varepsilon_{2}\beta_{\varepsilon_{2}}(\varphi_{\varepsilon_{2}})\big)\,\mathrm{d}x\mathrm{d}s
\notag\\
&\quad= -\varepsilon_{1}\Vert\beta_{\varepsilon_{1}}(\varphi_{\varepsilon_{1}})\Vert_{L^{2}(0,T;H)}^{2}
-\varepsilon_{2}\Vert\beta_{\varepsilon_{2}}(\varphi_{\varepsilon_{2}})\Vert_{L^{2}(0,T;H)}^{2}
+(\varepsilon_{1}+\varepsilon_{2})\int_{0}^{t}\int_{\Omega}\beta_{\varepsilon_{1}}(\varphi_{\varepsilon_{1}})
\beta_{\varepsilon_{2}}(\varphi_{\varepsilon_{2}})\,\mathrm{d}x\mathrm{d}s
\notag\\
&\quad\leq(\varepsilon_{1}+\varepsilon_{2})\Vert\beta_{\varepsilon_{1}}(\varphi_{\varepsilon_{1}})\Vert_{L^{2}(0,T;H)} \Vert\beta_{\varepsilon_{2}}(\varphi_{\varepsilon_{2}})\Vert_{L^{2}(0,T;H)}
\notag\\
&\quad \leq C (\varepsilon_{1}+\varepsilon_{2}).
\notag
\end{align}
Similarly, by the monotonicity of $\beta_\Gamma$ and \eqref{3.8}, we can obtain the following estimate for boundary potentials
\begin{align}
-\int_{0}^{t}\int_{\Gamma}\big(\beta_{\Gamma,\varepsilon_{1}}(\psi_{\varepsilon_{1}})-\beta_{\Gamma,\varepsilon_{2}}(\psi_{\varepsilon_{2}})\big)\psi_{*} \,\mathrm{d}S \mathrm{d}s
\leq C(\varepsilon_{1}+\varepsilon_{2}).
\notag
\end{align}
The above estimates provide the control of $J_2$:
\begin{align}
J_2\leq C(\varepsilon_{1}+\varepsilon_{2}).
\label{3.37}
\end{align}
To handle $J_3$, using $\mathbf{(A3)}$ and the Ehrling lemma, we find
\begin{align}
J_3
&\leq \max\{K,K_{\Gamma}\}\int_0^t \Vert\boldsymbol{\varphi}_{*}(s)\Vert_{\mathcal{L}^{2}}^2\,\mathrm{d}s
\notag\\
&\leq \max\{K,K_{\Gamma}\}\gamma \int_0^t \Vert\boldsymbol{\varphi}_{*}(s)\Vert_{\mathcal{V}_{(0)}^{1}}^2\,\mathrm{d}s
+ \max\{K,K_{\Gamma}\}C_\gamma  \int_0^t \Vert\boldsymbol{\varphi}_{*}(s)\Vert_{0,*}^2\,\mathrm{d}s,
\label{3.38}
\end{align}
Combining \eqref{3.32}, \eqref{3.34}, \eqref{3.37}, \eqref{3.38}, and taking $\gamma$ sufficiently small, we obtain
\begin{align}
\Vert\boldsymbol{\varphi}_{*}(t)\Vert_{0,*}^{2}+\int_0^t \Vert\boldsymbol{\varphi}_{*}(s)\Vert_{\mathcal{V}_{(0)}^{1}}^2\,\mathrm{d}s
\leq C(\varepsilon_{1}+\varepsilon_{2})
 +C\int_0^t \Vert\boldsymbol{\varphi}_{*}(s)\Vert_{0,*}^2\,\mathrm{d}s,
 \quad \forall\,t\in[0,T].
 \notag
\end{align}
Then we deduce from Gronwall's lemma that
\begin{align}
\Vert\boldsymbol{\varphi}_{*}\Vert_{L^{\infty}(0,T;\mathcal{H}_{(0)}^{-1})}^{2}
+\Vert\boldsymbol{\varphi}_{*}\Vert_{L^{2}(0,T;\mathcal{V}_{(0)}^{1})}^{2}
\leq C(\varepsilon_{1}+\varepsilon_{2}),
\label{error}
\end{align}
where the constant $C>0$ is independent of $\varepsilon_1$ and $\varepsilon_2$.

The estimate \eqref{error} implies that $\{\bm{\varphi}_\varepsilon\}$ is a Cauchy sequence in $L^{\infty}(0,T;\mathcal{H}_{(0)}^{-1})\cap L^{2}(0,T;\mathcal{V}^{1})$ as $\varepsilon\to 0$. Denote its limit by $\bm{\varphi}$. Recalling Theorems \ref{weakexist}, \ref{contidepen}, we find $\boldsymbol{\varphi}_{\varepsilon}\to\boldsymbol{\varphi}$ strongly in $L^{2}(0,T;\mathcal{V}^{1})\cap C([0,T];\mathcal{L}^{2})$ and $\bm{\varphi}$ is the unique weak solution to problem $(S_{L,\delta})$.
Hence, we can pass to the limit $\varepsilon_{2}\rightarrow0$ in \eqref{error} and obtain
$$\Vert\boldsymbol{\varphi}_{\varepsilon_{1}}-\boldsymbol{\varphi}\Vert^{2}_{L^{\infty}(0,T;(\mathcal{H}^{1})')}+\Vert\boldsymbol{\varphi}_{\varepsilon_{1}}-\boldsymbol{\varphi}\Vert^{2}_{L^{2}(0,T;\mathcal{V}^{1})}\leq C\varepsilon_{1},$$
which completes the proof.
\end{proof}

\begin{remark}
In the recent work \cite{GST}, the authors studied the nonlocal Cahn-Hillard equation:
\begin{align}
&\partial_{t}\varphi=\Delta\mu,&&\text{in }Q,\notag\\
&\mu=B(\varphi)+\beta(\varphi)+\pi(\varphi),&&\text{in }Q,\notag\\
&\partial_{\mathbf{n}}\mu=0,&&\text{on }\Sigma,\notag
\end{align}
with $B\varphi(x):=\int_{\Omega}J(x-y)\big(\varphi(x)-\varphi(y)\big)\,\mathrm{d}y$ and $J\in L_{\text{loc}}^{2}(\mathbb{R}^{d})$.
Assuming some additional assumptions on the nonlinearity $\pi$ and applying the theory of Hilbert-Schmidt operator for the nonlocal term,
they proved the convergence rate $\|\varphi_\varepsilon-\varphi\|_{L^2(0,T;H)}\leq C\sqrt{\varepsilon}$, where $\varepsilon$ is the parameter of the Yosida approximation.
\end{remark}

\subsection{Existence and uniqueness of strong solutions}
\textbf{Proof of Theorem \ref{strongexist}.} Now we consider strong solutions to problem $(S_{L,\delta})$.
The uniqueness is guaranteed by Theorem \ref{contidepen}. As in \cite{CF15}, we can prove the existence
by deriving some uniform estimates for higher order norms of the approximating solutions.
Consider the difference between \eqref{2.9}, at time $s$ and $s+h$, that is,
\begin{align}
&\varepsilon\big(\boldsymbol{\omega}_{\varepsilon}'(s+h)-\boldsymbol{\omega}_{\varepsilon}'(s)\big)
+\mathfrak{S}^{L}\big(\boldsymbol{\omega}_{\varepsilon}'(s+h)
-\boldsymbol{\omega}_{\varepsilon}'(s)\big)
+\partial\Phi_\delta\big(\boldsymbol{\omega}_{\varepsilon}(s+h)-\boldsymbol{\omega}_{\varepsilon}(s)\big)
\notag\\
&=\mathbf{P}\left(-\boldsymbol{\beta}_{\varepsilon}\big(\boldsymbol{\omega}_{\varepsilon}(s+h)+\overline{m}_{0}\boldsymbol{1}\big)
+\boldsymbol{\beta}_{\varepsilon}\big(\boldsymbol{\omega}_{\varepsilon}(s)+\overline{m}_{0}\boldsymbol{1}\big)\right)
+\mathbf{P}\left(-\boldsymbol{\pi}\big(\boldsymbol{\omega}_{\varepsilon}(s+h)+\overline{m}_{0}\boldsymbol{1}\big)
+\boldsymbol{\pi}\big(\boldsymbol{\omega}_{\varepsilon}(s)+\overline{m}_{0}\boldsymbol{1}\big)\right)
\notag\\
&\quad +\mathbf{P}\left(\boldsymbol{f}(s+h)-\boldsymbol{f}(s)\right),
\notag
\end{align}
for a.a. $s\in(0,T-h]$.
Taking $\mathcal{L}^{2}$ inner product with $\boldsymbol{\omega}_{\varepsilon}(s+h)-\boldsymbol{\omega}_{\varepsilon}(s)$, then integrating from $0$ to $t$ with respect to $s$, and dividing the resultant by $h^{2}$, we obtain
\begin{align}
&\frac{\varepsilon}{2}\Big\Vert\frac{\boldsymbol{\omega}_{\varepsilon}(s+h)-\boldsymbol{\omega}_{\varepsilon}(s)}{h}\Big\Vert_{\mathcal{L}^{2}}^{2}
+\frac{1}{2}\Big\Vert\frac{\boldsymbol{\omega}_{\varepsilon}(s+h)-\boldsymbol{\omega}_{\varepsilon}(s)}{h}\Big\Vert_{0,*}^{2}
+\frac{1}{2}\int_{0}^{t}\Big\Vert\frac{\boldsymbol{\omega}_{\varepsilon}(s+h)-\boldsymbol{\omega}_{\varepsilon}(s)}{h}\Big\Vert_{\mathcal{V}_{(0)}^{1}}^{2}\,\mathrm{d}s \notag\\
&\quad\leq\frac{\varepsilon}{2}\Big\Vert\frac{\boldsymbol{\omega}_{\varepsilon}(h)-\boldsymbol{\omega}_{0}}{h}\Big\Vert_{\mathcal{L}^{2}}^{2}
+\frac{1}{2}\Big\Vert\frac{\boldsymbol{\omega}_{\varepsilon}(h)-\boldsymbol{\omega}_{0}}{h}\Big\Vert_{0,*}^{2}
+\Big(K+K_{\Gamma}+\frac{1}{2}\Big)\int_{0}^{t}\Big\Vert\frac{\boldsymbol{\omega}_{\varepsilon}(s+h) -\boldsymbol{\omega}_{\varepsilon}(s)}{h}\Big\Vert_{\mathcal{L}^{2}}^{2}\,\mathrm{d}s
\notag\\
&\qquad+\frac{1}{2}\int_{0}^{t}\Big\Vert\frac{\boldsymbol{f}(s+h)-\boldsymbol{f}(s)}{h}\Big\Vert_{\mathcal{L}^{2}}^{2}\,\mathrm{d}s
\notag\\
&\quad\leq\frac{\varepsilon}{2}\Big\Vert\frac{\boldsymbol{\omega}_{\varepsilon}(h)-\boldsymbol{\omega}_{0}}{h} \Big\Vert_{\mathcal{L}^{2}}^{2}+\frac{1}{2}\Big\Vert\frac{\boldsymbol{\omega}_{\varepsilon}(h)-\boldsymbol{\omega}_{0}}{h} \Big\Vert_{0,*}^{2}
+\frac{1}{4}\int_{0}^{t}\Big\Vert\frac{\boldsymbol{\omega}_{\varepsilon}(s+h) -\boldsymbol{\omega}_{\varepsilon}(s)}{h}\Big\Vert_{\mathcal{V}_{(0)}^{1}}^{2}\,\mathrm{d}s
\notag\\
&\qquad+C\int_{0}^{t}\Big\Vert\frac{\boldsymbol{\omega}_{\varepsilon}(s+h)-\boldsymbol{\omega}_{\varepsilon}(s)}{h} \Big\Vert_{0,*}^{2}\,\mathrm{d}s
+\frac{1}{2}\int_{0}^{t}\Big\Vert\frac{\boldsymbol{f}(s+h)-\boldsymbol{f}(s)}{h}\Big\Vert_{\mathcal{L}^{2}}^{2}\,\mathrm{d}s.
\notag
\end{align}
The first two terms on the right-hand side can be controlled as in \cite[(4.35)]{CF15}. Integrating \eqref{2.9}  from $0$ to $h$, taking the inner product with $(\boldsymbol{\omega}_{\varepsilon}(h)-\boldsymbol{\omega}_{0})/h^{2}$, we obtain
\begin{align}
&\frac{\varepsilon}{2}\Big\Vert\frac{\boldsymbol{\omega}_{\varepsilon}(h)
-\boldsymbol{\omega}_{0}}{h}\Big\Vert_{\mathcal{L}^{2}}^{2}
+\frac{1}{2}\Big\Vert\frac{\boldsymbol{\omega}_{\varepsilon}(h)-\boldsymbol{\omega}_{0}}{h}\Big\Vert_{0,*}^{2}
\notag\\
&\quad\leq-\Big(\frac{\boldsymbol{\omega}_{\varepsilon}(h)
-\boldsymbol{\omega}_{0}}{h},\frac{1}{h}\int_{0}^{h}\mathbf{P}\Big(\partial\Phi_\delta\big(\boldsymbol{\varphi}_{\varepsilon}(s)\big)
+\boldsymbol{\beta}_{\varepsilon}\big(\boldsymbol{\varphi}_{\varepsilon}(s)\big)
+\boldsymbol{\pi}\big(\boldsymbol{\varphi}_{\varepsilon}(s)\big)-\boldsymbol{f}(s)\Big)ds\Big)_{\mathcal{L}^{2}}
\notag\\
&\quad=-\Big(\mathfrak{S}^L\frac{\boldsymbol{\omega}_{\varepsilon}(h)
-\boldsymbol{\omega}_{0}}{h},\frac{1}{h}\int_{0}^{h}\mathbf{P}\Big(\partial\Phi_\delta\big(\boldsymbol{\varphi}_{\varepsilon}(s)\big)
+\boldsymbol{\beta}_{\varepsilon}\big(\boldsymbol{\varphi}_{\varepsilon}(s)\big)
+\boldsymbol{\pi}\big(\boldsymbol{\varphi}_{\varepsilon}(s)\big)-\boldsymbol{f}(s)\Big)ds\Big)_{\mathcal{H}_{L,0}^{1}}
\notag\\
&\quad\leq\frac{1}{4}\Big\Vert\frac{\boldsymbol{\omega}_{\varepsilon}(h)
-\boldsymbol{\omega}_{0}}{h}\Big\Vert_{0,*}^{2}
+C\Big\Vert\frac{1}{h}\int_{0}^{h}\mathbf{P}\Big(\partial\Phi_{\delta}\big(\boldsymbol{\varphi}_{\varepsilon}(s)\big)
+\boldsymbol{\beta}_{\varepsilon}\big(\boldsymbol{\varphi}_{\varepsilon}(s)\big)
+\boldsymbol{\pi}\big(\boldsymbol{\varphi}_{\varepsilon}(s)\big)-\boldsymbol{f}(s)\Big)\,\mathrm{d}s
\Big\Vert_{\mathcal{H}^{1}}^{2}.
\label{equ4'}
\end{align}
Thanks to the additional assumptions $\mathbf{(A6)}$, $\mathbf{(A7)}$, the last term on the right-hand side of \eqref{equ4'} is bounded for all $\varepsilon\in (0,\varepsilon_0]$.
Then, by Gronwall's inequality, we obtain that
$$
s\mapsto\frac{\boldsymbol{\omega}_{\varepsilon}(s+h)-\boldsymbol{\omega}_{\varepsilon}(s)}{h}\quad
\text{are bounded in}\ L^{\infty}(0,T-h;\mathcal{H}_{(0)}^{-1})\cap L^{2}(0,T-h;\mathcal{V}_{(0)}^{1}),
$$
and
$$
s\mapsto\varepsilon^{1/2}\frac{\boldsymbol{\omega}_{\varepsilon}(s+h)-\boldsymbol{\omega}_{\varepsilon}(s)}{h}
\quad \text{are bounded in}\ L^{\infty}(0,T-h;\mathcal{L}^{2}),
$$
uniformly with respect to $\varepsilon\in(0,\varepsilon_{0}]$.
Passing to the limit as $h\rightarrow0$, we obtain
\begin{align}
\varepsilon^{1/2}\Vert\boldsymbol{\omega}_{\varepsilon}'\Vert_{L^{\infty}(0,T;\mathcal{L}^{2})}
+\Vert\boldsymbol{\omega}_{\varepsilon}'\Vert_{L^{\infty}(0,T;\mathcal{H}_{(0)}^{-1})}
+\Vert\boldsymbol{\omega}_{\varepsilon}'\Vert_{L^{2}(0,T;\mathcal{V}^{1})}\leq M_8,
\label{3.25}
\end{align}
where the constant $M_8>0$ is independent of $\varepsilon\in(0,\varepsilon_{0}]$.

Keeping the improved estimate \eqref{3.25} and  the assumption $\mathbf{(A7)}$ in mind, arguing as in Lemmas \ref{uni3}--\ref{uni4}, without integration in time over $(0,T)$ but taking the $L^\infty$-norm on $(0,T)$, we obtain
\begin{align}
&\Vert m_{\varepsilon}\Vert_{L^{\infty}(0,T)}+\Vert\boldsymbol{\mu}_{\varepsilon}\Vert_{L^{\infty}(0,T;\mathcal{H}^{1})}
\leq M_{9},
\notag\\
&\Vert\beta_{\varepsilon}(\varphi_{\varepsilon})\Vert_{L^{\infty}(0,T;H)}
+\Vert\beta_{\Gamma,\varepsilon}(\psi_{\varepsilon})\Vert_{L^{\infty}(0,T;H_{\Gamma})}
\leq M_{9},
\notag\\
&\Vert\boldsymbol{\omega}_{\varepsilon}\Vert_{L^{\infty}(0,T;\mathcal{H}^2)}\leq M_{9},
\notag
\end{align}
where the constant $M_{9}>0$ is independent of $\varepsilon\in(0,\varepsilon_{0}]$.
Passing to the limit as $\varepsilon\rightarrow 0$, we obtain the following additional regularities for the solution $(\bm{\omega}, \bm{\mu},\bm{\xi})$
\begin{align}
&\boldsymbol{\omega}\in W^{1,\infty}(0,T;\mathcal{H}_{(0)}^{-1})\cap H^{1}(0,T;\mathcal{V}_{(0)}^{1})\cap L^{\infty}(0,T;\mathcal{H}^2),
\notag\\
&\boldsymbol{\mu}\in L^{\infty}(0,T;\mathcal{H}^{1}),\quad\boldsymbol{\xi}\in L^{\infty}(0,T;\mathcal{L}^{2}).
\notag
\end{align}
Since $\boldsymbol{\omega}'\in L^{2}(0,T;\mathcal{V}_{(0)}^{1})$ and $\mathbf{P}\boldsymbol{\mu}=\mathfrak{S}^{L}(-\boldsymbol{\omega}')$, we can infer from the elliptic estimate \eqref{2.2b} that  $\boldsymbol{\mu}\in L^{2}(0,T;\mathcal{H}^{2})$.

Applying the above regularity properties and integration by parts, the weak formulation \eqref{eq3.2} can be rewritten as
\begin{align}
&\int_{\Omega}\big(\partial_{t}\omega(t)-\Delta\mu(t)\big)y\,\mathrm{d}x
+\int_{\Gamma}\big(\partial_{t}\omega_{\Gamma}(t)-\Delta_{\Gamma}\theta(t)+\partial_{\mathbf{n}}\mu(t)\big)y_{\Gamma}\,\mathrm{d}S
\notag\\
&\qquad+\frac{1}{L}\int_{\Gamma}\big(\theta(t)-\mu(t)-L\partial_{\mathbf{n}}\mu(t)\big)\big(y_{\Gamma}-y\big)\,\mathrm{d}S=0,
\quad\forall\,\boldsymbol{y}\in\mathcal{H}^{1},\text{ a.a. }t\in(0,T).
\label{3.27}
\end{align}
Since $\boldsymbol{y}\in\mathcal{H}^{1}$ is arbitrary, it easily follows from \eqref{3.27} that
\begin{align*}
&\partial_{t}\omega=\Delta\mu,&&\text{a.e. in }Q,\\
&\partial_{t}\omega_{\Gamma}=\Delta_{\Gamma}\theta-\partial_{\mathbf{n}}\mu,&&\text{a.e. on }\Sigma,\\
&L\partial_{\mathbf{n}}\mu=\theta-\mu,&&\text{a.e. on }\Sigma.
\end{align*}
In summary, $(\bm{\varphi},\bm{\mu},\bm{\xi})$ is a strong solution of problem $(S_{L,\delta})$ in the sense of Definition \ref{strongdefn}.
\hfill $\square$

\section{Asymptotic limit as $\delta\to 0$ and well-posedness without surface diffusion}
\label{sec4}
\setcounter{equation}{0}
In this section, we investigate the asymptotic behavior of weak solutions to problem $(S_{L,\delta})$ as $\delta\rightarrow 0$,
with $L\in (0,\infty)$ being fixed.
Our goal is to show that the limit problem $(S_{L,0})$ without surface diffusion is still well-posed.

\subsection{Uniform estimates}
Let the assumptions of Theorem \ref{existthm} be satisfied.
For any $\delta,\varepsilon\in(0,1)$, we consider the approximating problem \eqref{2.9}--\eqref{2.10}
with data $\{\bm{\varphi}_0^\delta,\bm{f}^\delta\}_{\delta\in (0,1)}$. Let $\boldsymbol{\omega}_{\delta,\varepsilon}$
be the unique solution given by Proposition \ref{approexist}.
Besides, we set
\begin{align}
\boldsymbol{\varphi}_{\delta,\varepsilon}= \boldsymbol{\omega}_{\delta,\varepsilon} +\overline{m}_0^\delta\mathbf{1},\quad
\boldsymbol{\mu}_{\delta,\varepsilon}:=\varepsilon\boldsymbol{\omega}_{\delta,\varepsilon}'(t)
+\partial\Phi_{\delta}\big(\boldsymbol{\omega}_{\delta,\varepsilon}\big)
+\boldsymbol{\beta}_{\varepsilon}\big(\boldsymbol{\varphi}_{\delta,\varepsilon}\big)
+\boldsymbol{\pi}\big(\boldsymbol{\varphi}_{\delta,\varepsilon}\big)-\boldsymbol{f}^\delta,
\notag
\end{align}
where $\overline{m}_0^\delta=\overline{m}(\bm{\varphi}_0^\delta)$.
Then $\big(\boldsymbol{\varphi}_{\delta,\varepsilon},\boldsymbol{\mu}_{\delta,\varepsilon}\big)$
is also uniquely determined and satisfies \eqref{2.14}--\eqref{appromutheta} with $\{\bm{\varphi}_0^\delta,\bm{f}^\delta\}_{\delta\in (0,1)}$.

As in the proof of Theorem \ref{weakexist}, when $\varepsilon\to 0$,
$\big(\boldsymbol{\varphi}_{\delta,\varepsilon},\boldsymbol{\mu}_{\delta,\varepsilon}, \boldsymbol{\beta}_{\varepsilon}(\boldsymbol{\varphi}_{\delta,\varepsilon})\big)$
converges (up to a subsequence) to some limit triplet $\big(\boldsymbol{\varphi}_{\delta},\boldsymbol{\mu}_{\delta},\boldsymbol{\xi}_{\delta}\big)$, which is a weak solution to problem $(S_{L,\delta})$ with data $\{\bm{\varphi}_0^\delta,\bm{f}^\delta\}_{\delta\in (0,1)}$.
Since the uniqueness of $\boldsymbol{\mu}_{\delta}$, $\boldsymbol{\xi}_{\delta}$ is not clear (cf. Theorem \ref{contidepen}), the related convergence should always be understood in the sense of a suitable subsequence. Below we will not relabel the convergent subsequence for the sake of simplicity.

We now derive uniform estimates of the solutions $(\boldsymbol{\varphi}_{\delta},\boldsymbol{\mu}_{\delta}, \boldsymbol{\xi}_{\delta})$ with respect to the parameter $\delta\in(0,1)$.

\begin{lemma}
\label{Uniform0}
There exists $\widetilde{\delta}\in (0,1)$ such that
\begin{align}
\sup_{t\in[0,T]}|\overline{m}(\bm{\varphi}_\delta(t))|\leq M_{10}, \quad \forall\, \delta\in (0,\widetilde{\delta}].
\label{con-mass-d}
\end{align}
where $M_{10}>0$ is independent of $\delta\in (0,\widetilde{\delta}]$.
\end{lemma}
\begin{proof}
Recalling the mass conservation property \eqref{con-mass}, we have
$\overline{m}(\boldsymbol{\varphi}_{\delta,\varepsilon}(t))=\overline{m}(\bm{\varphi}_0^\delta)$ for $t\in [0,T]$.
This together with \eqref{3.18} implies
\begin{align}
\overline{m}(\boldsymbol{\varphi}_{\delta}(t))=\overline{m}(\bm{\varphi}_0^\delta),\quad \forall\, t\in [0,T].
\label{mass-d0}
\end{align}
On the other hand, it follows from \eqref{del-3} that $\lim_{\delta\to 0} \overline{m}(\bm{\varphi}_0^\delta)= \overline{m}(\bm{\varphi}_0)\in \mathrm{Int}D(\beta_\Gamma)$.
Then the conclusion \eqref{con-mass-d} easily follows.
\end{proof}

\begin{lemma}
There exists a positive constant $M_{11}$, independent of $\delta\in(0,\widetilde{\delta}]$, such that
\begin{align}
&\Vert\boldsymbol{\varphi}_{\delta}\Vert_{L^{\infty}(0,T;(\mathcal{H}^{1})')}
+\Vert \varphi_{\delta}\Vert_{L^{2}(0,T;V)}
+\delta^{1/2}\Vert\nabla_{\Gamma}\psi_{\delta}\Vert_{L^{2}(0,T;H_{\Gamma})}
   \notag\\
&\qquad+\Vert\xi_\delta\Vert_{L^{1}(0,T;L^{1}(\Omega))}
+\Vert\xi_{\Gamma,\delta}\Vert_{L^{1}(0,T;L^{1}(\Gamma))}
\leq M_{11}.
\label{uni1'}
\end{align}
\end{lemma}
\begin{proof}
Recalling \eqref{eq3.1}, we have
\begin{align}
&\Vert\boldsymbol{\omega}_{\delta,\varepsilon}\Vert_{L^{\infty}(0,T;\mathcal{H}_{(0)}^{-1})}
+\Vert\nabla \omega_{\delta,\varepsilon}\Vert_{L^{2}(0,T;H)} + \delta^{1/2}\Vert\nabla_\Gamma \omega_{\Gamma,\delta,\varepsilon}\Vert_{L^{2}(0,T;H_\Gamma)}
\notag\\
&\quad+\Vert\beta_{\varepsilon}(\varphi_{\delta,\varepsilon})\Vert_{L^{1}(0,T;L^{1}(\Omega))}
+\Vert\beta_{\Gamma,\varepsilon}(\psi_{\delta,\varepsilon})\Vert_{L^{1}(0,T;L^{1}(\Gamma))}
\leq C,
\notag
\end{align}
where $C>0$ is independent of $\delta, \varepsilon$, thanks to the assumptions \eqref{del-1}, \eqref{del-2}.
By the weak lower semicontinuity as $\varepsilon\to 0$, we find
\begin{align}
&\Vert\boldsymbol{\omega}_{\delta}\Vert_{L^{\infty}(0,T;\mathcal{H}_{(0)}^{-1})}
+\Vert\nabla \omega_{\delta}\Vert_{L^{2}(0,T;H)} + \delta^{1/2}\Vert\nabla_\Gamma \omega_{\Gamma,\delta}\Vert_{L^{2}(0,T;H_\Gamma)}
\notag\\
&\quad+\Vert\xi_\delta\Vert_{L^{1}(0,T;L^{1}(\Omega))}
+\Vert\xi_{\Gamma,\delta}\Vert_{L^{1}(0,T;L^{1}(\Gamma))}
\leq C,
\notag
\end{align}
where $ \boldsymbol{\omega}_{\delta}=\boldsymbol{\varphi}_{\delta} -\overline{m}_0^\delta\mathbf{1}$ and $C>0$ is independent of $\delta$.

The above estimate combined with \eqref{con-mass-d} and the Poincar\'{e} inequality \eqref{Po4} yields the conclusion \eqref{uni1'}.
\end{proof}

\begin{lemma}
\label{Uniform2}
There exists a positive constant $M_{12}$, independent of $\delta \in(0,\widetilde{\delta}]$, such that
\begin{align}
&\Vert \varphi_{\delta}\Vert_{L^{\infty}(0,T;V)}
+\delta^{1/2}\Vert\nabla_{\Gamma}\psi_{\delta}\Vert_{L^{\infty}(0,T;H_{\Gamma})}
+\Vert\widehat{\beta}(\varphi_{\delta})\Vert_{L^{\infty}(0,T;L^{1}(\Omega))}
+\Vert\widehat{\beta}_{\Gamma}(\psi_{\delta})\Vert_{L^{\infty}(0,T;L^{1}(\Gamma))}
\notag\\
&\qquad +\Vert\partial_{t}\boldsymbol{\varphi}_{\delta}\Vert_{L^{2}(0,T;\mathcal{H}_{(0)}^{-1})}
 + \Vert\mathbf{P}\boldsymbol{\mu}_{\delta}\Vert_{L^{2}(0,T;\mathcal{H}_{L,0}^{1})}
 \leq M_{12}.
 \label{uni2'}
\end{align}
\end{lemma}
\begin{proof}
Recalling \eqref{equ2'b} and using the assumptions \eqref{del-1}, \eqref{del-2}, we get
\begin{align}
&\frac{1}{4}\Vert\nabla\omega_{\delta,\varepsilon}(t)\Vert^{2}_{H}+ \frac{\delta}{2}\Vert\nabla\omega_{\Gamma,\delta,\varepsilon}(t)\Vert^{2}_{H_\Gamma}
+\frac{1}{2}\int_{0}^{t}\Vert\boldsymbol{\omega}_{\delta,\varepsilon}'(s)\Vert_{0,*}^{2}\,\mathrm{d}s
\notag\\
&\quad
+ \int_{\Omega}\widehat{\beta}_{\varepsilon}\big(\varphi_{\delta,\varepsilon}(t)\big)\,\mathrm{d}x
+ \int_{\Gamma}\widehat{\beta}_{\Gamma,\varepsilon}\big(\psi_{\delta,\varepsilon}(t)\big)\,\mathrm{d}S
\leq C,
\quad \forall\, t\in [0,T],
\notag
\end{align}
where $C>0$ is independent of $\delta, \varepsilon$. Using \eqref{pmue} and passing to the limit as $\varepsilon\to 0$, we deduce from the weak lower semicontinuity that
\begin{align}
& \Vert\nabla\varphi_{\delta}(t)\Vert^{2}_{H}+  \delta \Vert\nabla\psi_{\delta}(t)\Vert^{2}_{H_\Gamma}
+ \int_{0}^{t}\Vert\partial_t\boldsymbol{\varphi}_{\delta}(s)\Vert_{0,*}^{2}\,\mathrm{d}s
+ \int_{0}^{t}\Vert\mathbf{P}\boldsymbol{\mu}_{\delta}(s)\Vert_{H^1_{L,0}}^{2}\,\mathrm{d}s
\notag\\
&\qquad
+ \int_{\Omega}\widehat{\beta}\big(\varphi_{\delta}(t)\big)\,\mathrm{d}x
+ \int_{\Gamma}\widehat{\beta}_{\Gamma}\big(\psi_{\delta}(t)\big)\,\mathrm{d}S
 \leq C,
\quad \forall\, t\in [0,T].
\notag
\end{align}
The above estimate combined with \eqref{con-mass-d} and the Poincar\'{e} inequality \eqref{Po4} yields the conclusion \eqref{uni2'}.
\end{proof}
\begin{remark}
From \eqref{uni2'}, we easily find the uniform estimate
$$
\Vert\boldsymbol{\varphi}_{\delta}\Vert_{L^{\infty}(0,T;\widetilde{\mathcal{V}}^1)}
+\delta^{1/2}\Vert\boldsymbol{\varphi}_{\delta}\Vert_{L^{\infty}(0,T;\mathcal{V}^{1})}\leq C.
$$
\end{remark}

\begin{lemma}
There exists a positive constant $M_{13}$, independent of $\delta \in(0,\widetilde{\delta}]$  such that
\begin{align}
\Vert\xi_{\delta}\Vert_{L^{2}(0,T;L^{1}(\Omega))}
+\Vert\xi_{\Gamma,\delta}\Vert_{L^{2}(0,T;L^{1}(\Gamma))}
+\Vert\boldsymbol{\mu}_{\delta}\Vert_{L^{2}(0,T;\mathcal{H}^{1})}
\leq M_{13}.
\label{uni4'}
\end{align}
\end{lemma}
\begin{proof}
Using the estimates obtained in Lemmas \ref{Uniform0}--\ref{Uniform2}, we can get the conclusion \eqref{uni4'} by the same argument for Lemma \ref{uni3} and passing to the limit as $\varepsilon\to 0$.
\end{proof}

\begin{lemma}
\label{Uniform4}
There exists a positive constant $M_{14}$, independent of $\delta \in(0,\widetilde{\delta}]$, such that
\begin{align}
& \Vert\Delta \varphi_{\delta}\Vert_{L^{2}(0,T;H)}+ \Vert\partial_{\mathbf{n}}\varphi_{\delta}\Vert_{L^{2}(0,T;(H^{1/2}(\Gamma))')}
+ \delta^{1/2}\Vert\Delta_{\Gamma}\psi_{\delta}\Vert_{L^{2}(0,T;V_{\Gamma}')}
\notag\\
&\quad
+\Vert\xi_{\delta}\Vert_{L^{2}(0,T;H)}
+\Vert\xi_{\Gamma,\delta}\Vert_{L^{2}(0,T;V_{\Gamma}')}
\leq M_{14}.
\label{uni5'}
\end{align}
\end{lemma}
\begin{proof}
By the same argument for \eqref{be-be} and using the assumptions \eqref{del-1}, \eqref{del-2}, we can deduce that
\begin{align}
\Vert\beta_{\varepsilon}(\varphi_{\delta,\varepsilon})\Vert_{L^{2}(0,T;H)}
+ \varrho^{-1/2}\Vert\beta_{\varepsilon}(\psi_{\delta,\varepsilon})\Vert_{L^{2}(0,T;H_{\Gamma})}\leq C,
\notag
\end{align}
where $C>0$ is independent of $\delta, \varepsilon$. Passing to the limit as $\varepsilon\to 0$, we get
\begin{align}
\Vert\xi_{\delta}\Vert_{L^{2}(0,T;H)}
+ \varrho^{-1/2}\Vert\xi_{\delta}\Vert_{L^{2}(0,T;H_{\Gamma})}
\leq C.
\label{equ11}
\end{align}
By comparison in \eqref{eq3.5}, we infer from the above estimates and the assumption \eqref{del-2} that
\begin{align}
\Vert\Delta \varphi_{\delta}\Vert_{L^{2}(0,T;H)} \leq C,
\label{equ12}
\end{align}
where $C>0$ is independent of $\delta$.
Thanks to the trace theorem \cite[Theorem 2.27]{BG} and the elliptic regularity theorem \cite[Theorem 3.2]{BG}, we infer from the estimates \eqref{uni2'}, \eqref{equ12} that
\begin{align}
\Vert\partial_{\mathbf{n}}\varphi_{\delta}\Vert_{L^{2}(0,T;(H^{1/2}(\Gamma))')} \leq C.
\label{equ13}
\end{align}
By comparison in \eqref{eq3.6}, we infer from the estimates \eqref{uni2'}, \eqref{uni4'}, \eqref{equ13} and the assumption \eqref{del-2} that
\begin{align}
& \Vert-\delta\Delta_{\Gamma}\psi_{\delta}+\xi_{\Gamma,\delta}\Vert_{L^{2}(0,T;(H^{1/2}(\Gamma))')} \notag\\
&\quad \leq \|\theta_\delta\|_{L^{2}(0,T;H_\Gamma)} + \|\partial_{\mathbf{n}}\varphi_\delta\|_{L^{2}(0,T;(H^{1/2}(\Gamma))')} + \|\pi_{\Gamma}(\psi_\delta)\|_{L^{2}(0,T;H_\Gamma)}
+\|f_\Gamma^\delta\|_{L^{2}(0,T;H_\Gamma)}
\notag\\
&\quad \leq C.
\label{equ14}
\end{align}
Since $\delta^{1/2}\Delta_\Gamma \psi_\delta$  is uniformly bounded in $L^2(0,T;V_\Gamma')$ by the estimate \eqref{uni1'}, a direct comparison
in \eqref{equ14} yields
\begin{align}
\Vert\xi_{\Gamma,\delta}\Vert_{L^{2}(0,T;V_{\Gamma}')}
\leq C.
\label{equ15}
\end{align}
Collecting the estimates \eqref{equ11}, \eqref{equ12}, \eqref{equ13} and \eqref{equ15}, we arrive at the conclusion \eqref{uni5'}.
\end{proof}

\subsection{Passage to the limit as $\delta\to 0$}
\textbf{Proof of Theorem \ref{existthm}.}
From the uniform estimates obtained in Lemmas \ref{Uniform0}--\ref{Uniform4}, we find that there exist a triplet $\big(\boldsymbol{\varphi},\boldsymbol{\mu},\boldsymbol{\xi}\big)$ satisfying the regularity properties
\begin{align*}
&\boldsymbol{\varphi}\in H^{1}(0,T;(\mathcal{H}^{1})')\cap L^{\infty}(0,T;\widetilde{\mathcal{V}}^1),\quad\Delta \varphi\in L^{2}(0,T;H),\\
&\boldsymbol{\mu}\in L^{2}(0,T;\mathcal{H}^{1}),\\
&\boldsymbol{\xi}\in L^{2}(0,T;H\times (H^{1/2}(\Gamma))'),
\end{align*}
and the following convergence results hold  as $\delta\to 0$ (in the sense of a subsequence)
\begin{align}
\boldsymbol{\varphi}_{\delta}&\to\boldsymbol{\varphi}
&&\text{weakly star in }L^{\infty}(0,T;\widetilde{\mathcal{V}}^1),
\label{conver1}\\
\boldsymbol{\varphi}_{\delta}&\to\boldsymbol{\varphi}
&&\text{weakly in }H^{1}(0,T;(\mathcal{H}^{1})'),
\label{conver1a}\\
\Delta \varphi_{\delta}&\to\Delta \varphi
&&\text{weakly in }L^{2}(0,T;H),
\label{conver2}\\
\boldsymbol{\mu}_{\delta}&\to\boldsymbol{\mu}
&&\text{weakly in }L^{2}(0,T;\mathcal{H}^{1}),
\label{conver3}\\
\boldsymbol{\xi}_{\delta}&\to\boldsymbol{\xi}
&&\text{weakly in }L^{2}(0,T;H\times V_{\Gamma}'),
\label{conver4}\\
\delta\boldsymbol{\varphi}_{\delta}&\to\boldsymbol{0}
&&\text{strongly in }L^{\infty}(0,T;\mathcal{V}^{1}),
\label{conver5}\\
-\delta\Delta_{\Gamma}\psi_{\delta}+\xi_{\Gamma,\delta}&\to\xi_\Gamma
&&\text{weakly in }L^{2}(0,T;(H^{1/2}(\Gamma))').
\label{conver6}
\end{align}
Here, we note that \eqref{conver5} implies $\delta\Delta_\Gamma \psi_\delta\to 0$ in $L^2(0,T;V_\Gamma')$, thus \eqref{conver4} and \eqref{conver6} are consistent.
Next, by Lemma \ref{ALS}, we can conclude
\begin{align}
\boldsymbol{\varphi}_{\delta}\to\boldsymbol{\varphi}\quad\qquad\text{ strongly in }C([0,T];\mathcal{L}^{2}).
\label{compact1}
\end{align}
This implies $\boldsymbol{\varphi}(0)=\boldsymbol{\varphi}_{0}$ thanks to the assumption \eqref{del-3}.
Besides, it follows from \eqref{mass-d0} that
\begin{align}
\overline{m}(\boldsymbol{\varphi}(t))=\overline{m}(\bm{\varphi}_0),\quad \forall\, t\in [0,T].
\notag
\end{align}
The strong convergence \eqref{compact1} combined with the assumption $\mathbf{(A3)}$ also gives
\begin{align*}
\pi(\varphi_{\delta})\to\pi(\varphi)\quad\text{strongly in }C([0,T];H),
\quad\pi_{\Gamma}(\psi_{\delta})\to\pi_{\Gamma}(\psi)\quad\text{strongly in }C([0,T];H_{\Gamma}).
\end{align*}
Passing to the weak limit as $\delta\to 0$ in \eqref{eq3.2} yields the variational formulation \eqref{modelu}.
Next, we write \eqref{eq3.3} as
\begin{align}
\big(\boldsymbol{\mu}_\delta(t),\boldsymbol{z}\big)_{\mathcal{L}^{2}}
&= \big(\nabla  \varphi_\delta(t),\nabla z\big)_H + \big(\xi_\delta,z)_H
+ \big\langle -\delta\Delta_{\Gamma}\psi_{\delta} +\xi_{\Gamma,\delta}, z_\Gamma\big\rangle_{(H^{1/2}(\Gamma))',H^{1/2}(\Gamma)}
\notag\\
&\quad +\big(\boldsymbol{\pi}\big(\boldsymbol{\varphi}_\delta(t)\big)-\boldsymbol{f}^\delta(t),\boldsymbol{z}\big)_{\mathcal{L}^{2}}, \quad\forall\,\boldsymbol{z}\in\mathcal{V}^{1},
\notag
\end{align}
for a.a. $t\in (0,T)$.
Letting $\delta\to 0$, we easily recover \eqref{modeltheta}.
Thanks to the classical results in \cite{Br,B}, the weak convergence \eqref{conver4} and the strong convergence \eqref{compact1} combined with the demi-closedness of the maximal monotone operator $\beta$ yields \eqref{xi}, that is,  $\xi\in\beta(\varphi)$ almost everywhere in $Q$.
Moreover, by exactly the same argument as in \cite[Section 3.2]{CFS} we can justify the variational inequality \eqref{zeta} for $\xi_\Gamma$.
According to \eqref{conver6} and \cite[Remark 2.6]{CFS}, we have $\xi_\Gamma\in \partial J_\Sigma(\psi)$, where
$$
J_\Sigma:\,L^2(0,T;H^{1/2}(\Gamma))\to [0,\infty],\quad
J_\Sigma:=
\begin{cases}
\int_\Sigma \widehat{\beta}_\Gamma(\psi)\,\mathrm{d}S\mathrm{d}t, \quad \text{if}\ \widehat{\beta}_\Gamma(\psi)\in L^1(\Sigma),\\
\infty,\qquad\qquad \qquad\  \text{otherwise}.
\end{cases}
$$

In summary, the limit $(\boldsymbol{\varphi},\boldsymbol{\mu},\boldsymbol{\xi})$ is a weak solution to the system $(S_{L,0})$ in the sense of Definition \ref{weakdefinition}. The proof of Theorem \ref{existthm} is complete.
\hfill $\square$

\subsection{Uniqueness and further properties}
\textbf{Proof of Theorem \ref{continuousdepen}.}
The proof is similar to that for Theorem \ref{contidepen}. Keeping in mind that  $\delta=0$, in analogy to \eqref{3.19a}, we obtain
\begin{align}
&\frac{1}{2}\frac{\mathrm{d}}{\mathrm{d}s}\Vert\boldsymbol{\omega}_{1}(s)-\boldsymbol{\omega}_{2}(s)\Vert_{0,*}^{2}
+\Vert\nabla (\omega_{1}(s)-\omega_{2}(s))\Vert_{H}^{2}
+  \int_{\Omega}\big(\xi_1(s)-\xi_2(s)\big)\big(\omega_{1}(s)-\omega_{2}(s)\big)\,\mathrm{d}x
\notag\\
&\qquad +\big\langle\xi_{\Gamma,1}(s)-\xi_{\Gamma,2}(s),\omega_{\Gamma,1}(s)-\omega_{\Gamma,2}(s)\big\rangle_{(H^{1/2}(\Gamma))',H^{1/2}(\Gamma)}
\notag\\
&\quad\leq(K+K_{\Gamma})\Vert\boldsymbol{\omega}_{1}(s)-\boldsymbol{\omega}_{2}(s)\Vert_{\mathcal{L}^{2}}^{2}
+\Vert\boldsymbol{f}_{1}(s)-\boldsymbol{f}_{2}(s)\Vert_{(\widetilde{\mathcal{V}}^{1})'}\Vert\boldsymbol{\omega}_{1}(s) -\boldsymbol{\omega}_{2}(s)\Vert_{\widetilde{\mathcal{V}}^{1}}.
\label{3.19a-d}
\end{align}
Since $\boldsymbol{\omega}_{1}(s)-\boldsymbol{\omega}_{2}(s)\in \widetilde{\mathcal{V}}_{(0)}^{1}$ for a.a. $s\in (0,T)$,
from the Poincar\'{e} inequality \eqref{Po4} and
the interpolation inequality \eqref{int-0}, we get
\begin{align}
&\Vert\boldsymbol{\omega}_{1}(s)-\boldsymbol{\omega}_{2}(s)\Vert_{\widetilde{\mathcal{V}}^{1}}
\leq C \Vert\nabla (\omega_{1}(s)-\omega_{2}(s))\Vert_{H},
\notag\\
&\Vert\boldsymbol{\omega}_{1}(s)-\boldsymbol{\omega}_{2}(s)\Vert_{\mathcal{L}^{2}}^{2}
\leq \gamma \Vert\nabla (\omega_{1}(s)-\omega_{2}(s))\Vert_{H}^{2}
+C_\gamma \Vert\boldsymbol{\omega}_{1}(s)-\boldsymbol{\omega}_{2}(s)\Vert_{0,*}^{2},
\notag
\end{align}
for any $\gamma>0$ and a.a. $s\in (0,T)$.
Besides, it follows from  the monotonicity of $\beta$, $\beta_{\Gamma}$ that
\begin{align}
 &\int_{\Omega}\big(\xi_1(s)-\xi_2(s)\big)\big(\omega_{1}(s)-\omega_{2}(s)\big)\,\mathrm{d}x
\notag\\
&\quad +\big\langle\xi_{\Gamma,1}(s)-\xi_{\Gamma,2}(s),\omega_{\Gamma,1}(s)-\omega_{\Gamma,2}(s) \big\rangle_{(H^{1/2}(\Gamma))',H^{1/2}(\Gamma)} \geq 0,
\quad \text{for a.a.}\ s\in (0,T).
\notag
\end{align}
Inserting the above inequalities into \eqref{3.19a-d}, taking $\gamma$ sufficiently small, then
by Gronwall's lemma, the fact $\overline{m}(\boldsymbol{\varphi}_{1})=\overline{m}(\boldsymbol{\varphi}_{2})$ and the Poincar\'e inequality \eqref{Po4},
we can conclude \eqref{conti-d0}.
\hfill $\square$
\medskip

Finally, if the bulk and boundary potentials exhibit the same growth, some further properties of the asymptotic limit as $\delta\to 0$ can be obtained.
More precisely, we have

\begin{corollary}\label{Cor-d}
Let the assumptions of Theorem \ref{existthm} be satisfied. In addition, we assume
\begin{description}
\item[$\mathbf{(A8)}$] $D(\beta)= D(\beta_\Gamma)$ and there exists a constant $M\geq1$ such that
\begin{align}
\frac{1}{M}|\beta^\circ_{\Gamma}(r)|-M\leq|\beta^\circ(r)|\leq M |\beta^\circ_{\Gamma}(r)|+M,
\quad\forall\,r\in D(\beta).
\notag
\end{align}
\end{description}
\begin{itemize}
\item[(1)] Improved regularity: the
limit triplet $(\bm{\varphi}, \bm{\mu}, \bm{\xi})$ obtained in Theorem \ref{existthm} satisfies
\begin{align*}
&\varphi\in L^{2}(0,T;H^{3/2}(\Omega)),\quad\partial_{\mathbf{n}}\varphi\in L^{2}(0,T;H_{\Gamma}),\quad \psi\in L^{2}(0,T;V_{\Gamma}),\\
&\xi_{\Gamma}\in L^{2}(0,T;H_{\Gamma}),\quad\xi_{\Gamma}\in\beta_{\Gamma}(\psi)\quad\text{a.e. on }\Sigma.
\end{align*}
\item[(2)] Refined convergence (in the sense of a subsequence): it holds
\begin{align}
\xi_{\Gamma,\delta}&\to\xi\,&&\text{ weakly in }L^{2}(0,T;H_{\Gamma}),\notag\\
\delta \psi_{\delta}&\to0\,&&\text{ weakly in }L^{2}(0,T;H^{3/2}(\Gamma)),\notag\\
\partial_{\mathbf{n}}\varphi_{\delta}-\delta\Delta_{\Gamma}\psi_{\delta}\,&\to\partial_{\mathbf{n}}\varphi\,&&\text{ weakly in }L^{2}(0,T;H_{\Gamma}),
\notag
\end{align}
as $\delta\rightarrow0$ and \eqref{thetapoint} can be replaced by
$\theta=\partial_{\mathbf{n}}\varphi+\xi_{\Gamma}+\pi_{\Gamma}(\psi)-f_\Gamma$ almost everywhere on $\Sigma$.
\item[(3)] Convergence rate:  if $\overline{m}(\bm{\varphi}_0^\delta)\equiv\overline{m}(\bm{\varphi}_0)$ for $\delta\in (0,1)$,
then for all $\delta\in(0,1)$, it holds
\begin{align}
\Vert\boldsymbol{\varphi}_{\delta}-\boldsymbol{\varphi}\Vert_{L^{\infty}(0,T;\mathcal{H}_{(0)}^{-1})}
+\Vert\boldsymbol{\varphi}_{\delta}-\boldsymbol{\varphi}\Vert_{L^{2}(0,T;\widetilde{\mathcal{V}}^1)}
\leq C\big(\delta^{1/2}+ \Vert\boldsymbol{\varphi}^{\delta}_0-\boldsymbol{\varphi}_0\Vert_{0,*}
+ \|\bm{f}^\delta-\bm{f}\|_{L^2(0,T;\mathcal{L}^2)}\big).
\notag
\end{align}
\end{itemize}
\end{corollary}

Since the proof of Corollary \ref{Cor-d} follows the same arguments for \cite[Theorems 2.10, 2.12]{CFS},
we leave the details to interested readers.

\section{Asymptotic limits with respect to kinetic rate: $L\rightarrow0$ or $L\to\infty$}
\label{sec5}
\setcounter{equation}{0}
In this section, we study the asymptotic limit of solutions to problem $(S_{L,\delta})$ as $L\rightarrow0$ or $L\to\infty$, with $\delta\in (0,\infty)$ being fixed.

For any $\varepsilon\in(0,1)$, we consider the approximating problem \eqref{2.9}--\eqref{2.10}
with data $(\bm{\varphi}_0,\bm{f})$. Let $\boldsymbol{\omega}_{L, \varepsilon}$
be the unique solution given by Proposition \ref{approexist}.
We set
\begin{align}
\boldsymbol{\varphi}_{L,\varepsilon}= \boldsymbol{\omega}_{L,\varepsilon}
+\overline{m}_0\mathbf{1},\quad
\boldsymbol{\mu}_{L,\varepsilon}:=\varepsilon\boldsymbol{\omega}_{L,\varepsilon}'(t)
+\partial\Phi_{\delta}\big(\boldsymbol{\omega}_{L,\varepsilon}\big)
+\boldsymbol{\beta}_{\varepsilon}\big(\boldsymbol{\varphi}_{L,\varepsilon}\big)
+\boldsymbol{\pi}\big(\boldsymbol{\varphi}_{L,\varepsilon}\big)-\boldsymbol{f},
\notag
\end{align}
where $\overline{m}_0=\overline{m}(\bm{\varphi}_0)$.
Then $\big(\boldsymbol{\varphi}_{L,\varepsilon},\boldsymbol{\mu}_{L,\varepsilon}\big)$
is uniquely determined and satisfies \eqref{2.14}--\eqref{appromutheta} with data $(\bm{\varphi}_0,\bm{f})$.

As in the proof of Theorem \ref{weakexist}, when $\varepsilon\to 0$,
$\big(\boldsymbol{\varphi}_{L,\varepsilon},\boldsymbol{\mu}_{L,\varepsilon}, \boldsymbol{\beta}_{\varepsilon}(\boldsymbol{\varphi}_{L,\varepsilon})\big)$
converges to certain limit triplet $\big(\boldsymbol{\varphi}_{L},\boldsymbol{\mu}_{L},\boldsymbol{\xi}_{L}\big)$, which is a weak solution to problem $(S_{L,\delta})$ with data $(\bm{\varphi}_0,\bm{f})$.
Since the uniqueness of $\boldsymbol{\mu}_{L}$, $\boldsymbol{\xi}_{L}$ is not clear (cf. Theorem \ref{contidepen}), the related convergence should always be understood in the sense of a suitable subsequence. Below we will not relabel the convergent subsequence for the sake of simplicity.

The proofs of Theorem \ref{asymptotic0} and Theorem \ref{asymptoticinfinity} rely on uniform estimates with respect to the kinetic rate $L$. The first estimate comes from the mass conservation. Recalling \eqref{con-mass} and passing to the limit $\varepsilon\to 0$, we get
\begin{align}
\overline{m}(\bm{\varphi}_L(t))= \overline{m}(\bm{\varphi}_0)=\overline{m}_0,\quad \forall\,  t\in [0,T].
\label{con-mass-L}
\end{align}
Nevertheless, further uniform estimates have to be derived separately for the two different cases $L\to 0$ and $L\to \infty$.

\subsection{The case $L\rightarrow0$}

We now proceed to derive uniform estimates for $L\in(0, 1]$.

\begin{lemma}
\label{asympuniform1}
There exists a positive constant $C_{1}$, independent of $L\in(0, 1]$, such that
\begin{align}
&\Vert\boldsymbol{\omega}_{L}\Vert_{L^{\infty}(0,T;\mathcal{H}_{(0)}^{-1})}
+\Vert\boldsymbol{\omega}_{L}\Vert_{L^{2}(0,T;\mathcal{V}_{(0)}^{1})}
+ \Vert\xi_L\Vert_{L^{1}(0,T;L^{1}(\Omega))}
+\Vert\xi_{\Gamma,L}\Vert_{L^{1}(0,T;L^{1}(\Gamma))}
\leq C_1,
\label{uni-L1}
\end{align}
where $\bm{\omega}_L=\bm{\varphi}_L-\overline{m}_0\mathbf{1}$.
\end{lemma}
\begin{proof}
Recalling the derivation of \eqref{eq3.1}, we have
\begin{align}
&\Vert\boldsymbol{\omega}_{L,\varepsilon}\Vert_{L^{\infty}(0,T;\mathcal{H}_{(0)}^{-1})}^2
+\Vert\boldsymbol{\omega}_{L,\varepsilon}\Vert_{L^{2}(0,T;\mathcal{V}_{(0)}^{1})}^2
+\Vert\beta_{\varepsilon}(\varphi_{L,\varepsilon})\Vert_{L^{1}(0,T;L^{1}(\Omega))}
+\Vert\beta_{\Gamma,\varepsilon}(\psi_{L,\varepsilon})\Vert_{L^{1}(0,T;L^{1}(\Gamma))}
\notag\\
&\quad \leq C\big(\varepsilon\|\bm{\omega}_0\|_{\mathcal{L}^2}^2+\|\bm{\omega}_0\|_{0,*}^2+1),
\notag
\end{align}
where $C>0$ is independent of $L, \varepsilon$.
By the weak lower semicontinuity as $\varepsilon\to 0$, we find
\begin{align}
&\Vert\boldsymbol{\omega}_{L}\Vert_{L^{\infty}(0,T;\mathcal{H}_{(0)}^{-1})}^2
+\Vert\boldsymbol{\omega}_{L}\Vert_{L^{2}(0,T;\mathcal{V}_{(0)}^{1})}^2
+\Vert\xi_{L}\Vert_{L^{1}(0,T;L^{1}(\Omega))}
+\Vert\xi_{\Gamma,L}\Vert_{L^{1}(0,T;L^{1}(\Gamma))}
 \leq C\big(\|\bm{\omega}_0\|_{0,*}^2+1).
\label{uni-L1a}
\end{align}
Recall that for any fixed $L\in (0,\infty)$, the norms $\Vert\cdot\Vert_{\mathcal{H}^{1}}$ and $\Vert\cdot\Vert_{\mathcal{H}_{L,0}^{1}}$ are equivalent on $\mathcal{H}_{L,0}^{1}$.
This fact yields
\begin{align}
\Vert\boldsymbol{y}\Vert_{\mathcal{L}^{2}}
\leq \Vert\boldsymbol{y}\Vert_{\mathcal{H}^{1}}
\leq \widehat{C}\Vert\boldsymbol{y}\Vert_{\mathcal{H}_{1,0}^{1}}
\leq \widehat{C}\Vert\boldsymbol{y}\Vert_{\mathcal{H}_{L,0}^{1}},
\quad\forall\,\boldsymbol{y}\in\mathcal{H}_{L,0}^{1},
\label{equiv1}
\end{align}
where $\widehat{C}>0$ is independent of $L\in (0,1]$.
 As a consequence, it holds
\begin{align}
\Vert\boldsymbol{\omega}_{0}\Vert_{0,*}^{2}
&=\Vert\mathfrak{S}^{L}\boldsymbol{\omega}_{0}\Vert_{\mathcal{H}_{L,0}^{1}}^{2}
=\big(\boldsymbol{\omega}_{0},\mathfrak{S}^{L}\boldsymbol{\omega}_{0}\big)_{\mathcal{L}^{2}}\notag\\
&\leq\Vert\boldsymbol{\omega}_{0}\Vert_{\mathcal{L}^{2}}\Vert\mathfrak{S}^{L}\boldsymbol{\omega}_{0}\Vert_{\mathcal{L}^{2}}
\leq \widehat{C} \Vert\boldsymbol{\omega}_{0}\Vert_{\mathcal{L}^{2}}\Vert\mathfrak{S}^{L}\boldsymbol{\omega}_{0}\Vert_{\mathcal{H}_{L,0}^{1}} =\widehat{C}\Vert\boldsymbol{\omega}_{0}\Vert_{\mathcal{L}^{2}}\Vert\boldsymbol{\omega}_{0}\Vert_{0,*}.
\notag
\end{align}
Hence, $\Vert\boldsymbol{\omega}_{0}\Vert_{0,*}\leq \widehat{C}\Vert\boldsymbol{\omega}_{0}\Vert_{\mathcal{L}^{2}}$ and thus it is uniformly bounded for $L\in (0,1]$. This together with \eqref{uni-L1a} yields the conclusion \eqref{uni-L1}.
\end{proof}

\begin{lemma}
There exists a positive constant $C_{2}$, independent of $L\in(0, 1]$, such that
\begin{align}
& \Vert\boldsymbol{\varphi}_{L}\Vert_{L^{\infty}(0,T;\mathcal{V}^{1})}
+\|\widehat{\beta}(\varphi_{L})\|_{L^\infty(0,T;L^1(\Omega))}
+\|\widehat{\beta}_{\Gamma}(\psi_{L})\|_{L^\infty(0,T;L^1(\Omega))}
\notag\\
&\quad + \Vert\partial_t\boldsymbol{\varphi}_{L}\Vert_{L^2(0,T;\mathcal{H}_{(0)}^{-1})}
+ \Vert\mathbf{P}\boldsymbol{\mu}_{L}\Vert_{L^2(0,T;\mathcal{H}_{L,0}^{1})}
 \leq C_2.
\label{uni-L2}
\end{align}
\end{lemma}
\begin{proof}
Recalling \eqref{equ2'b} and \eqref{pmue}, passing to the limit as $\varepsilon\to 0$,
we can deduce the uniform estimate \eqref{uni-L2} from the weak lower semicontinuity,
\eqref{con-mass-L}, \eqref{uni-L1} and the Poincar\'{e} inequality \eqref{Po4}.
\end{proof}

\begin{lemma}
There exists a positive constant $C_{3}$, independent of $L\in(0, 1]$, such that
\begin{align}
\|\partial_t\bm{\varphi}_L\|_{L^2(0,T;(\mathcal{V}^1)')}\leq C_3.
\label{eq7.3-L}
\end{align}
\end{lemma}
\begin{proof}
From \eqref{uni-L2} and the definition of $\|\cdot\|_{\mathcal{H}_{L,0}^{1}}$, we find
\begin{align}
&\Vert\nabla\mu_L\Vert_{L^{2}(0,T;H)}
+\Vert\nabla_{\Gamma}\theta_L\Vert_{L^{2}(0,T;H_{\Gamma})}
+\frac{1}{\sqrt{L}}\Vert\mu_L-\theta_L\Vert_{L^{2}(0,T;H_{\Gamma})}
\leq C,
\label{eq7.2}
\end{align}
where $C>0$ is independent of $L\in (0,1]$.
Taking $\boldsymbol{z}=(z,z_\Gamma)\in\mathcal{V}^{1}$ in the weak formulation \eqref{eq3.2}, we get
\begin{align}
\big\langle \partial_t\bm{\varphi}_L(t),\boldsymbol{z}\big\rangle_{(\mathcal{V}^1)',\mathcal{V}^1}
+\int_{\Omega}\nabla\mu_L(t)\cdot\nabla z\,\mathrm{d}x
+\int_\Gamma \nabla \theta_L(t)\cdot \nabla_\Gamma z_\Gamma \,\mathrm{d}S =0,
\quad \text{for a.a.}\ t\in (0,T).
\notag
\end{align}
Then it follows that
\begin{align*}
\left|\big\langle \partial_t\bm{\varphi}_L(t),\boldsymbol{z}\big\rangle_{(\mathcal{V}^1)',\mathcal{V}} \right|
&\leq \|\nabla\mu_L(t)\|_H\|\nabla z\|_H+ \|\nabla \theta_L(t)\|_{H_\Gamma}\| \nabla_\Gamma z_\Gamma\|_{H_\Gamma}
\\
&\leq \big(\|\nabla\mu_L(t)\|_H + \|\nabla \theta_L(t)\|_{H_\Gamma}\big)\|\bm{z}\|_{\mathcal{V}^1},
\end{align*}
which together with \eqref{eq7.2} implies the estimate \eqref{eq7.3-L}.
\end{proof}

\begin{lemma}
There exists a positive constant $C_{4}$, independent of $L\in(0,1]$, such that
\begin{align}
\Vert\xi_L\Vert_{L^{2}(0,T;L^{1}(\Omega))}+\Vert\xi_{\Gamma,L}\Vert_{L^{2}(0,T;L^{1}(\Gamma))}
+\Vert\boldsymbol{\mu}_L\Vert_{L^{2}(0,T;\mathcal{H}^{1})}
\leq C_{4}.
\label{7.10}
\end{align}
\end{lemma}
\begin{proof}
Keeping the estimates \eqref{con-mass-L}, \eqref{uni-L1}, \eqref{equiv1}, \eqref{uni-L2} in mind,
we can obtain the uniform estimate \eqref{7.10} with respect to $L\in(0,1]$ by the same argument for Lemma \ref{uni3} and passing to the limit as $\varepsilon\to 0$.
\end{proof}

\begin{lemma}
\label{asympuniform3}
There exists a positive constant $C_{5}$, independent of $L\in(0,1]$, such that
\begin{align}
&\Vert\xi_L\Vert_{L^{2}(0,T;H)}+\Vert\xi_{L}\Vert_{L^{2}(0,T;H_{\Gamma})}
+\Vert\xi_{\Gamma,L}\Vert_{L^{2}(0,T;H_{\Gamma})}
+\Vert\bm{\varphi}_L\Vert_{L^{2}(0,T;\mathcal{V}^2)}
\leq C_{5}.
\notag
\end{align}
\end{lemma}
\begin{proof}
To conclude, we apply the same argument for Lemma \ref{uni4} and then pass to the limit as $\varepsilon\to 0$. The estimate for
$\Vert\bm{\varphi}_L\Vert_{L^{2}(0,T;\mathcal{V}^2)}$ follows from the estimate for $\Vert\bm{\omega}_L\Vert_{L^{2}(0,T;\mathcal{V}^2)}$, \eqref{con-mass-L}
and the Poincar\'e inequality \eqref{Po4}.
\end{proof}

\textbf{Proof of Theorem \ref{asymptotic0}.}
The existence of a limit triplet $(\boldsymbol{\varphi}^{0},\boldsymbol{\mu}^{0},\boldsymbol{\xi}^{0})$ with expected regularity properties and the (sequential) convergence results \eqref{7.20}--\eqref{7.28} are guaranteed by \eqref{con-mass-L}, Lemma \ref{asympuniform1}--\ref{asympuniform3} and the same compactness argument as in the proof of Theorem \ref{existthm}.
By \eqref{7.22}, \eqref{7.28} and the maximal monotonicity of $\beta$ and $\beta_{\Gamma}$, we find
$$
\xi^{0}\in\beta(\varphi^{0})\ \ \text{a.e. in}\ Q,\quad\xi_{\Gamma}^{0}\in\beta_{\Gamma}(\psi^{0})\ \ \text{a.e. in}\ \Sigma.
$$
Then taking limit $L\rightarrow0$ in \eqref{eq3.2}, \eqref{eq3.3}, we recover \eqref{eq3.2-0} and \eqref{eq3.3-0}.
The inequality \eqref{7.30} follows from \eqref{eq7.2}. This also implies
$$
\mu^{0}|_{\Gamma}=\theta^{0}\quad \text{a.e. on}\ \Sigma.
$$
Hence, $(\boldsymbol{\varphi}^{0},\boldsymbol{\mu}^{0},\boldsymbol{\xi}^{0})$ is a weak solution of GMS model in the sense of Definition \ref{GMS}. The uniqueness of $\bm{\varphi}^0$ has been established in \cite[Theorem 2.1]{CF15}.
The proof of Theorem \ref{asymptotic0} is complete.
\hfill $\square$

\subsection{The case $L\rightarrow\infty$}

We now derive uniform estimates for $L\geq \widehat{L}$, where $\widehat{L}\geq 1$ is a constant to be determined later.

\begin{lemma}
There exists a positive constant $C_6$, independent of $L\geq 1$, such that
\begin{align}
&\Vert\partial_t\boldsymbol{\omega}_{L}\Vert_{L^{2}(0,T;\mathcal{H}_{(0)}^{-1})}
+\Vert\boldsymbol{\varphi}_{L}\Vert_{L^{\infty}(0,T;\mathcal{V}^{1})}
+\Vert\widehat{\beta}(\varphi_{L})\Vert_{L^{\infty}(0,T;L^{1}(\Omega))}
+\Vert\widehat{\beta}_{\Gamma}(\psi_{L})\Vert_{L^{\infty}(0,T;L^{1}(\Gamma))}\leq C_6.
\label{3.4-L}
\end{align}
\end{lemma}
\begin{proof}
Testing \eqref{2.9} at time $s\in(0,T)$ by $\boldsymbol{\omega}_{L,\varepsilon}'\in\mathcal{L}_{(0)}^{2}$, with the help of \eqref{eq2.2}, we obtain
\begin{align}
&\frac{\mathrm{d}}{\mathrm{d}s}\left(\Phi_\delta\big(\boldsymbol{\omega}_{L,\varepsilon}(s)\big)
+ \int_{\Omega}\widehat{\beta}_{\varepsilon}\big(\varphi_{L,\varepsilon}(s)\big)\,\mathrm{d}x
+ \int_{\Gamma}\widehat{\beta}_{\Gamma,\varepsilon}\big(\psi_{L,\varepsilon}(s)\big)\,\mathrm{d}S \right)
+\varepsilon\Vert\boldsymbol{\omega}_{L,\varepsilon}'(s)\Vert_{\mathcal{L}^{2}}^{2}
+\Vert\boldsymbol{\omega}_{L,\varepsilon}'(s)\Vert_{0,*}^{2}
\notag\\
&\quad = -\big(\mathbf{P}\bm{\pi}(\bm{\varphi}_{L,\varepsilon}(s)),\boldsymbol{\omega}_{L,\varepsilon}'(s)\big)_{\mathcal{L}^{2}}
+\big(\mathbf{P}\boldsymbol{f}(s),\boldsymbol{\omega}_{L,\varepsilon}'(s)\big)_{\mathcal{L}^{2}},
\quad \text{for a.a.}\ s\in(0,T).
\label{equ1'-L}
\end{align}
Similar to \eqref{equ1'-ho}, it holds
\begin{align}
 \big(\mathbf{P}\boldsymbol{f}(s),\boldsymbol{\omega}_{L,\varepsilon}'(s)\big)_{\mathcal{L}^{2}}
  = \big(\mathbf{P}\boldsymbol{f}(s),\mathfrak{S}^L \boldsymbol{\omega}_{L,\varepsilon}'(s) \big)_{\mathcal{H}_{L,0}^1}
 \leq\frac{1}{4}\Vert\boldsymbol{\omega}_{L,\varepsilon}'(s)\Vert_{0,*}^{2}
+ \|\boldsymbol{f}(s)\|_{\mathcal{V}^{1}}^2,
\label{equ1'-ho-L}
\end{align}
and
\begin{align}
\big|\big(\mathbf{P}\bm{\pi}(\bm{\varphi}_{L,\varepsilon}(s)),\boldsymbol{\omega}_{L,\varepsilon}'(s)\big)_{\mathcal{L}^{2}}\big|
&=\big| \big(\mathbf{P}\bm{\pi}(\bm{\varphi}_{L,\varepsilon}(s)),\mathfrak{S}^L\boldsymbol{\omega}_{L,\varepsilon}'(s)\big)_{\mathcal{H}_{L,0}^1}\big|
\notag\\
&\leq \|\mathbf{P}\bm{\pi}(\bm{\varphi}_{L,\varepsilon}(s))\|_{\mathcal{H}_{L,0}^1}\|\boldsymbol{\omega}_{L,\varepsilon}'(s)\|_{0,*}
\notag\\
&\leq\frac{1}{4}\Vert\boldsymbol{\omega}_{L,\varepsilon}'(s)\Vert_{0,*}^{2}
+ \|\mathbf{P}\bm{\pi}(\bm{\varphi}_{L,\varepsilon}(s))\|_{\mathcal{H}_{L,0}^1}^2.
\label{equ1'-ho-L2}
\end{align}
Thanks to the additional assumption $\pi, \pi_\Gamma \in W^{1,\infty}(\mathbb{R})$, we deduce from $\mathbf{(A3)}$,
$\boldsymbol{\varphi}_{L,\varepsilon}= \boldsymbol{\omega}_{L,\varepsilon}
+\overline{m}_0\mathbf{1}$ and the Poincar\'{e} inequality \eqref{Po4} that for all $L\geq 1$, it holds
\begin{align}
\|\mathbf{P}\bm{\pi}\big(\bm{\varphi}_{L,\varepsilon}(s)\big)\|_{\mathcal{H}_{L,0}^1}^2
&\leq \|\nabla\pi\big(\varphi_{L,\varepsilon}(s)\big)\|_H^{2}
+ \|\nabla_{\Gamma}\pi_{\Gamma}\big(\psi_{L,\varepsilon}(s)\big)\|_{H_\Gamma}^{2}
+ \|\pi\big(\varphi_{L,\varepsilon}(s)\big)-\pi_{\Gamma}\big(\psi_{L,\varepsilon}(s)\big)\|_{H_\Gamma}^{2}
\notag\\
&\leq K^2\|\nabla \omega_{L,\varepsilon}(s)\|_H^{2}
+K_\Gamma^2 \|\nabla_{\Gamma}\omega_{\Gamma,L,\varepsilon}(s)\|_{H_\Gamma}^{2}
+3 K^2\|\omega_{L,\varepsilon}(s)\|_{H}^{2}
\notag\\
&\quad + 3 K_\Gamma^2\|\omega_{\Gamma,L,\varepsilon}(s)\|_{H_\Gamma}^{2}
+ 6 |\Gamma| \left(|\pi(\overline{m}_{0})|^{2}+|\pi_{\Gamma}(\overline{m}_{0})|^{2}\right)\notag\\
&\leq C \Vert\boldsymbol{\omega}_{L,\varepsilon}(s)\Vert^{2}_{\mathcal{V}_{(0)}^{1}} +C,
\label{equ1'-ho-L3}
\end{align}
where $C>0$ is independent of $L\geq 1$.
In view of \eqref{equ1'-ho-L}, \eqref{equ1'-ho-L2}, \eqref{equ1'-ho-L3} and $\mathbf{(A4)}$ we infer from \eqref{equ1'-L} and Gronwall's lemma that for all $L\geq 1$, it holds
\begin{align}
&\frac{1}{2}\Vert\boldsymbol{\omega}_{L,\varepsilon}(t)\Vert^{2}_{\mathcal{V}_{(0)}^{1}}
+\int_{\Omega}\widehat{\beta}_{\varepsilon}\big(\varphi_{L,\varepsilon}(t)\big)\,\mathrm{d}x
+\int_{\Gamma}\widehat{\beta}_{\Gamma,\varepsilon}\big(\psi_{L,\varepsilon}(t)\big)\,\mathrm{d}S
\notag\\
&\quad +\varepsilon\int_{0}^{t}\Vert\boldsymbol{\omega}_{L,\varepsilon}'(s)\Vert_{\mathcal{L}^{2}}^{2}\,\mathrm{d}s +\frac{1}{2}\int_{0}^{t}\Vert\boldsymbol{\omega}_{L,\varepsilon}'(s)\Vert_{0,*}^{2}\,\mathrm{d}s
 \leq C,\quad \forall\, t\in [0,T],
\notag
\end{align}
where $C>0$ depends on $\Vert\boldsymbol{\omega}_{0}\Vert^{2}_{\mathcal{V}_{(0)}^{1}}$, $\overline{m}_0$, $\int_{\Omega}\widehat{\beta}(\varphi_{0})\,\mathrm{d}x$, $\int_{\Gamma}\widehat{\beta}_{\Gamma}(\psi_{0})\,\mathrm{d}S$, $\|\bm{f}\|_{L^2(0,T;\mathcal{V}^1)}$, $\Omega$ and $T$.
Passing to the limit as $\varepsilon\to 0$, by the weak lower semicontinuity and $\bm{\varphi}_L=\bm{\omega}_L+\overline{m}_0\mathbf{1}$,
we obtain the estimate \eqref{3.4-L}.
\end{proof}

\begin{lemma}
There exists a positive constant $C_7$, independent of $L\geq 1$, such that
\begin{align}
&\Vert\partial_t\varphi_L\Vert_{L^{2}(0,T;V')}\leq C_7,
\label{eq7.3}\\
&\Vert\partial_t\psi_L\Vert_{L^{2}(0,T;V_{\Gamma}')}\leq C_7.
\label{eq7.4}
\end{align}
\end{lemma}
\begin{proof}
From \eqref{3.4-L} and the definition of $\|\cdot\|_{\mathcal{H}_{L,0}^{1}}$, we see that for $L\geq 1$, it holds
\begin{align}
&\Vert\nabla\mu_L\Vert_{L^{2}(0,T;H)}
+\Vert\nabla_{\Gamma}\theta_L\Vert_{L^{2}(0,T;H_{\Gamma})}
+\frac{1}{\sqrt{L}}\Vert\mu_L-\theta_L\Vert_{L^{2}(0,T;H_{\Gamma})}
\leq C,
\label{eq7.2-L}
\end{align}
where $C>0$ is independent of $L$.
For any $z\in V$, taking $\boldsymbol{z}=(z,0)\in\mathcal{H}^{1}$ in \eqref{eq3.2}, we obtain
\begin{align}
\big\langle \partial_t\varphi_L(t),z\big\rangle_{V',V}+\int_{\Omega}\nabla\mu_L(t)\cdot\nabla z\,\mathrm{d}x+\frac{1}{L}\int_{\Gamma}\big(\mu_{L}(t)-\theta_{L}(t)\big)z\,\mathrm{d}S=0,
\quad \text{for a.a.}\ t\in (0,T).
\label{time-L}
\end{align}
It follows that
\begin{align}
\Big|\big\langle \partial_t\varphi_L(t),z\big\rangle_{V',V}\Big|
&\leq\Vert\nabla\mu_L(t)\Vert_{H}\Vert\nabla z\Vert_{H}
+\frac{1}{L}\Vert\mu_L(t)-\theta_L(t)\Vert_{H_{\Gamma}}\Vert z\Vert_{H_{\Gamma}}
\notag\\
&\leq C\Big(\Vert\nabla\mu_L(t)\Vert_{H}
+\frac{1}{\sqrt{L}}\cdot\frac{1}{\sqrt{L}}\Vert\mu_L(t)-\theta_L(t)\Vert_{H_{\Gamma}}\Big)\Vert z\Vert_{V},
\notag
\end{align}
which implies
$$\Vert\partial_t\varphi_L(t)\Vert_{V'}\leq C\Big(\Vert\nabla\mu_L(t)\Vert_{H}
+\frac{1}{\sqrt{L}}\cdot\frac{1}{\sqrt{L}}\Vert\mu_L(t)-\theta_L(t)\Vert_{H_{\Gamma}}\Big).$$
Then by \eqref{eq7.2-L}, we get \eqref{eq7.3}.
Analogously, for any $z_{\Gamma}\in V_{\Gamma}$, taking $\boldsymbol{z}=(0,z_{\Gamma})\in\mathcal{H}^{1}$ in \eqref{eq3.2}, we obtain
\begin{align}
\big\langle \partial_t \psi_L(t),z_{\Gamma}\big\rangle_{V_{\Gamma}',V_{\Gamma}}
+\int_{\Gamma}\nabla_{\Gamma}\theta_L(t)\cdot\nabla_{\Gamma}z_{\Gamma}\,\mathrm{d}S
-\frac{1}{L}\int_{\Gamma}\big(\mu_L(t)-\theta_L(t)\big)z_\Gamma\,\mathrm{d}S=0.
\notag
\end{align}
By the same reasoning, we arrive at \eqref{eq7.4}.
\end{proof}

\begin{lemma}
There exists a large constant $\widehat{L}\geq 1$, such that for all $L\geq \widehat{L}$, it holds
\begin{align}
\Vert\xi_L\Vert_{L^{2}(0,T;L^{1}(\Omega))}+\Vert\xi_{\Gamma,L}\Vert_{L^{2}(0,T;L^{1}(\Gamma))}
\leq C_8\big(1+\Vert\partial_{\mathbf{n}}\varphi_L\Vert_{L^{2}(0,T;H_\Gamma)}\big),
\label{7.42}
\end{align}
where $C_8>0$ is independent of $L$.
\end{lemma}
\begin{proof}
Taking the test function $\boldsymbol{y}=\big(\mathcal{N}_{\Omega}\big(\varphi_L-\langle\varphi_L\rangle_\Omega\big), \mathcal{N}_{\Gamma}\big(\psi_L-\langle\psi_L\rangle_\Gamma\big)\big)\in\mathcal{H}^{1}$
in \eqref{eq3.2}, then we obtain
\begin{align}
0&=\big\langle\partial_{t}\varphi_L,\mathcal{N}_{\Omega}\big(\varphi_L-\langle\varphi_L\rangle_\Omega\big)\big\rangle_{V',V}
+\big\langle\partial_{t}\psi_L,\mathcal{N}_{\Gamma}\big(\psi_L-\langle\psi_L\rangle_\Gamma\big)\big\rangle_{V_{\Gamma}',V_{\Gamma}}
\notag\\
&\quad +\int_{\Omega}\nabla\mu_L\cdot\nabla\mathcal{N}_{\Omega}\big(\varphi_L-\langle\varphi_L\rangle_\Omega\big)\,\mathrm{d}x
+\int_{\Gamma}\nabla_{\Gamma}\theta_L\cdot\nabla_{\Gamma}\mathcal{N}_{\Gamma}\big(\psi_L-\langle\psi_L\rangle_\Gamma\big)\,\mathrm{d}S
\notag\\
&\quad+\frac{1}{L}\int_{\Gamma}(\mu_L-\theta_L)\left[\mathcal{N}_{\Omega}\big(\varphi_L-\langle\varphi_L\rangle_\Omega\big)
-\mathcal{N}_{\Gamma}\big(\psi_L-\langle\psi_L\rangle_\Gamma\big)\right]\,\mathrm{d}S
\notag\\
&=\big\langle\partial_{t}\varphi_L,\mathcal{N}_{\Omega}\big(\varphi_L-\langle\varphi_L\rangle_\Omega\big)\big\rangle_{V',V}
+\big\langle\partial_{t}\psi_L,\mathcal{N}_{\Gamma}\big(\psi_L-\langle\psi_L\rangle_\Gamma\big)\big\rangle_{V_{\Gamma}',V_{\Gamma}}
\notag\\
&\quad
+\int_{\Omega}\mu_L\big(\varphi_L-\langle\varphi_L\rangle_\Omega\big)\,\mathrm{d}x
+\int_{\Gamma}\theta_L\big(\psi_L-\langle\psi_L\rangle_\Gamma\big)\,\mathrm{d}S
\notag\\
&\quad +\frac{1}{L}\int_{\Gamma}(\mu_L-\theta_L)\left[\mathcal{N}_{\Omega}\big(\varphi_L-\langle\varphi_L\rangle_\Omega\big)
-\mathcal{N}_{\Gamma}\big(\psi_L-\langle\psi_L\rangle_\Gamma\big)\right]\,\mathrm{d}S
\notag\\
&=: \sum_{j=1}^5 R_j.
\label{7.33}
\end{align}
It follows from the definition of $\mathcal{N}_{\Omega}$, $\mathcal{N}_{\Gamma}$, H\"{o}lder's inequality and Poincar\'e's inequalities \eqref{Po1}, \eqref{Po2} that
\begin{align}
|R_1|+|R_2|
& \leq
|\big\langle\partial_{t}\varphi_L,\mathcal{N}_{\Omega}\big(\varphi_L-\langle\varphi_L\rangle_\Omega\big)\big\rangle_{V',V}|
+|\big\langle\partial_{t}\psi_L,\mathcal{N}_{\Gamma}\big(\psi_L-\langle\psi_L\rangle_\Gamma\big)\big\rangle_{V_{\Gamma}',V_{\Gamma}}|
\notag\\
& \leq C\|\partial_{t}\varphi_L\|_{V'}\|\nabla\varphi_L\|_H+  C\|\partial_{t}\psi_L\|_{V_\Gamma'}\|\nabla_\Gamma\psi_L\|_{H_\Gamma}.
\label{7.33-L1}
\end{align}
Next, by the trace theorem and the elliptic estimates, we find
\begin{align}
|R_5|&=\left|\frac{1}{L}\int_{\Gamma}(\mu_L-\theta_L)\left[\mathcal{N}_{\Omega}\big(\varphi_L-\langle\varphi_L\rangle_\Omega\big)
-\mathcal{N}_{\Gamma}\big(\psi_L-\langle\psi_L\rangle_\Gamma\big)\right]\mathrm{d}S\right|
\notag\\
&\leq \frac{C}{L}\Vert\mu_L-\theta_L\Vert_{H_{\Gamma}}\Big(\Vert\varphi_L-\langle\varphi_L\rangle_\Omega\Vert_{V_{0}^*}
+\Vert \psi_L-\langle\psi_L\rangle_\Gamma\Vert_{V_{\Gamma,0}^*}\Big)
\notag\\
&\leq\frac{C}{\sqrt{L}}\left(\frac{1}{\sqrt{L}}\Vert\mu_L-\theta_L\Vert_{H_{\Gamma}}\right)
\big(\Vert \varphi_L\Vert_{H} +\Vert \psi_L\Vert_{H_{\Gamma}}\big).
\label{7.40}
\end{align}
The estimates for $R_2$, $R_3$ are more involved. By \eqref{eq3.5} and \eqref{eq3.6}, we get
\begin{align}
R_2+R_3
&=\int_{\Omega}\mu_L\big(\varphi_L-\langle\varphi_L\rangle_\Omega\big)\,\mathrm{d}x
+\int_{\Gamma}\theta_L\big(\psi_L-\langle\psi_L\rangle_\Gamma\big)\,\mathrm{d}S
\notag\\
&=\int_{\Omega}|\nabla \varphi_L|^{2}\,\mathrm{d}x
+\delta \int_{\Gamma}|\nabla_{\Gamma}\psi_L|^{2}\,\mathrm{d}S
+\int_{\Omega}\xi_L\big(\varphi_L-\langle\varphi_L\rangle_\Omega\big)\,\mathrm{d}x
+\int_{\Gamma}\xi_{\Gamma,L}\big(\psi_L-\langle\psi_L\rangle_\Gamma\big)\,\mathrm{d}S
\notag\\
&\quad+\int_{\Omega}\big(\pi(\varphi_L)-f\big)\big(\varphi_L-\langle\varphi_L\rangle_\Omega\big)\,\mathrm{d}x +\int_{\Gamma}\big(\pi_{\Gamma}(\psi_L)-f_\Gamma\big)\big(\psi_L-\langle\psi_L\rangle_\Gamma\big)\,\mathrm{d}S
\notag\\
&\quad + \big(\langle\varphi_L\rangle_\Omega-\langle\psi_L\rangle_\Gamma\big)\int_{\Gamma}\partial_{\mathbf{n}}\varphi_L\,\mathrm{d}S.
\label{7.34}
\end{align}
Consider the third and fourth terms on the right-hand side of \eqref{7.34}
\begin{align}
&\int_{\Omega}\xi_L\big(\varphi_L-\langle\varphi_L\rangle_\Omega\big)\,\mathrm{d}x
+\int_{\Gamma}\xi_{\Gamma,L}\big(\psi_L-\langle\psi_L\rangle_\Gamma\big)\,\mathrm{d}S
\notag\\
&\quad=\int_{\Omega}\xi_L\big(\varphi_L-\langle\varphi_0\rangle_\Omega\big)\,\mathrm{d}x
+\int_{\Gamma}\xi_{\Gamma,L}\big(\psi_L-\langle\psi_0\rangle_\Gamma\big)\,\mathrm{d}S
\notag\\
&\qquad
+\big(\langle\varphi_0\rangle_\Omega-\langle\varphi_L\rangle_\Omega\big)\int_{\Omega}\xi_L\,\mathrm{d}x
+\big(\langle\psi_0\rangle_\Gamma-\langle\psi_L\rangle_\Gamma\big)\int_{\Gamma}\xi_{\Gamma,L}\,\mathrm{d}S.
\label{7.34a}
\end{align}
It follows from \eqref{eq7.2-L} and \eqref{time-L} that
\begin{align}
\big|\langle\varphi_0\rangle_\Omega-\langle\varphi_L(t)\rangle_\Omega\big|
&=\frac{1}{|\Omega|}\left|\int_0^t\big\langle \partial_t\varphi_L(s),1\big\rangle_{V',V}\,\mathrm{d}s\right|
\notag\\
&\leq \frac{1}{|\Omega|}\frac{1}{L}\int_{0}^{t}\|\mu_L(s)-\theta_L(s)\|_{L^1(\Gamma)}\,\mathrm{d}s
\notag\\
&\leq\frac{|\Gamma|^{1/2}|T|^{1/2}}{|\Omega|}\frac{1}{\sqrt{L}}\left(\frac{1}{\sqrt{L}}\Vert\mu_L-\theta_L\Vert_{L^{2}(0,T;H_{\Gamma})} \right)
\notag\\
&\leq\frac{C'|\Gamma|^{1/2}|T|^{1/2}}{|\Omega|}\frac{1}{\sqrt{L}}, \quad \forall\, t\in [0,T],
\label{7.36}
\end{align}
and in a similar manner,
\begin{align}
\big|\langle\psi_0\rangle_\Gamma-\langle\psi_L(t)\rangle_\Gamma\big|
\leq\frac{C''|T|^{1/2}}{|\Gamma|^{1/2}}\frac{1}{\sqrt{L}},\quad \forall\, t\in [0,T],
\label{7.37}
\end{align}
where the constants $C',C''>0$ in \eqref{7.36}, \eqref{7.37} are independent of $L$.
Since $\langle \varphi_0\rangle_\Omega, \langle \psi_0\rangle_\Gamma \in \mathrm{int}D(\beta_\Gamma)$, recalling \eqref{eq2.8}, we can first work with the approximate solutions $\bm{\varphi}_{L,\varepsilon}$ and then pass to the limit as $\varepsilon\to 0$, to get
\begin{align*}
&\int_{\Omega}\xi_L\big(\varphi_L-\langle\varphi_0\rangle_\Omega\big)\,\mathrm{d}x
\geq \delta_{0}^{(1)} \|\xi_L\|_{L^1(\Omega)}-c_{1}^{(1)}|\Omega|,\\
&\int_{\Gamma}\xi_{\Gamma,L}\big(\psi_L-\langle\psi_0\rangle_\Gamma\big)\,\mathrm{d}S
\geq \delta_{0}^{(2)} \|\xi_{\Gamma,L}\|_{L^1(\Gamma)} -c_{1}^{(2)}|\Gamma|.
\end{align*}
Set $\widehat{\delta}_0=\min\big\{\delta_{0}^{(1)}, \delta_{0}^{(2)}\big\}$, $\widehat{c}_1=\max\big\{c_{1}^{(1)},c_{1}^{(2)}\big\}$.
There exists some $\widehat{L}\geq 1$ sufficiently large such that
$$
\frac{C'|\Gamma|^{\frac{1}{2}}|T|^{\frac{1}{2}}}{|\Omega|}\frac{1}{\sqrt{L}}
\leq\frac{\widehat{\delta}_{0}}{2},\quad\frac{C''|T|^{\frac{1}{2}}}{|\Gamma|^{\frac{1}{2}}}\frac{1}{\sqrt{L}}\leq\frac{\widehat{\delta}_{0}}{2},
\qquad \forall\,L\geq \widehat{L}.
$$
Then it follows from \eqref{7.34a}--\eqref{7.37} that
\begin{align}
&\int_{\Omega}\xi_L\big(\varphi_L-\langle\varphi_L\rangle_\Omega\big)\,\mathrm{d}x
+\int_{\Gamma}\xi_{\Gamma,L}\big(\psi_L-\langle\psi_L\rangle_\Gamma\big)\,\mathrm{d}S
\geq \frac{\widehat{\delta}_{0}}{2}\left(\|\xi_L\|_{L^1(\Omega)}+\|\xi_{\Gamma,L}\|_{L^1(\Gamma)}\right)-\widehat{c}_{1}(|\Omega|+|\Gamma|).
\label{7.38}
\end{align}
Concerning the last three terms on the right-hand side of \eqref{7.34},
from $\mathbf{(A3)}$ and H\"{o}lder's inequality, we get
\begin{align}
&\left|\int_{\Omega}\big(\pi(\varphi_L)-f\big)\big(\varphi_L-\langle\varphi_L\rangle_\Omega\big)\,\mathrm{d}x \right| + \left|\int_{\Gamma}\big(\pi_{\Gamma}(\psi_L)-f_\Gamma\big)\big(\psi_L-\langle\psi_L\rangle_\Gamma\big)\,\mathrm{d}S\right|
\notag\\
&\qquad + \left|\big(\langle\varphi_L\rangle_\Omega-\langle\psi_L\rangle_\Gamma\big)\int_{\Gamma}\partial_{\mathbf{n}}\varphi_L\,\mathrm{d}S\right|
\notag\\
&\quad \leq C \|\bm{\varphi}_L\|_{\mathcal{L}^2}\big(1+\|\bm{\varphi}_L\|_{\mathcal{L}^2}+\|\bm{f}\|_{\mathcal{L}^2}
+ \Vert\partial_{\mathbf{n}}\varphi_L\Vert_{H_{\Gamma}}\big).
\label{7.39}
\end{align}
Returning to \eqref{7.33}, owing to the estimates \eqref{7.33-L1}--\eqref{7.39}, we infer from
\eqref{3.4-L}, \eqref{eq7.3}, \eqref{eq7.4}, \eqref{eq7.2-L} and $\mathbf{(A4)}$ that
\begin{align}
&\frac{\widehat{\delta}_{0}}{2}\left(\|\xi_L\|_{L^2(0,T;L^1(\Omega))}+\|\xi_{\Gamma,L}\|_{L^2(0,T;L^1(\Gamma))}\right)\notag\\
&\quad \leq CT^{1/2}\left(1+ \|\bm{\varphi}_L\|_{L^\infty(0,T;\mathcal{L}^2)}^2\right)
+C\|\bm{\varphi}_L\|_{L^\infty(0,T;\mathcal{L}^2)}\left(\|\bm{f}\|_{L^2(0,T;\mathcal{L}^2)}
+ \Vert\partial_{\mathbf{n}}\varphi_{L}\Vert_{L^{2}(0,T;H_{\Gamma})}\right)\notag\\
&\qquad + \frac{C}{\sqrt{L}}\left(\frac{1}{\sqrt{L}}\Vert\mu_L-\theta_L\Vert_{L^2(0,T;H_{\Gamma})}\right)\|\bm{\varphi}_L\|_{L^\infty(0,T;\mathcal{L}^2)}
\notag\\
&\qquad + C \left(\|\partial_{t}\varphi_L\|_{L^2(0,T;V')}+  \|\partial_{t}\psi_L\|_{L^2(0,T;V_\Gamma')}\right)\|\bm{\varphi}_L\|_{L^\infty(0,T;\mathcal{V}^1)}
\notag\\
&\quad \leq C\left(1+ \Vert\partial_{\mathbf{n}}\varphi_{L}\Vert_{L^{2}(0,T;H_{\Gamma})}\right),\qquad \forall\, L\geq \widehat{L},
 \notag
\end{align}
where $C>0$ is independent of $L\in [\widehat{L},\infty)$. Thus, the conclusion \eqref{7.42} follows.
\end{proof}

\begin{lemma}
There exists a positive constant $C_9$, independent of $L\geq \widehat{L}$, such that
\begin{align}
\Vert\boldsymbol{\mu}_L\Vert_{L^{2}(0,T;\mathcal{H}^{1})}
\leq C_9\left(1+\Vert\partial_{\mathbf{n}}\varphi_L\Vert_{L^{2}(0,T;H_{\Gamma})}\right).
\label{7.43}
\end{align}
\end{lemma}
\begin{proof}
From \eqref{eq3.5} and \eqref{eq3.6}, we find
\begin{align}
&\langle\mu_L\rangle_\Omega=\frac{1}{|\Omega|}\Big[\int_{\Omega}\big(\xi_L+\pi(\varphi_L)-f\big)\,\mathrm{d}x
-\int_{\Gamma}\partial_{\mathbf{n}}\varphi_{L}\, \mathrm{d}S\Big],
\notag\\
&\langle\theta_L\rangle_\Gamma
=\frac{1}{|\Gamma|}\int_{\Gamma}\big(\xi_{\Gamma,L}+\pi_{\Gamma}(\psi_L)-f_\Gamma+\partial_{\mathbf{n}}\varphi_L\big)\, \mathrm{d}S,
\notag
\end{align}
which together with \eqref{3.4-L}, \eqref{7.42} and $\mathbf{(A4)}$ imply
\begin{align}
\Vert \langle\mu_L\rangle_\Omega\Vert_{L^{2}(0,T)}+\Vert \langle\theta_L\rangle_\Gamma\Vert_{L^{2}(0,T)}\leq C\left(1+\Vert\partial_{\mathbf{n}}\varphi_L\Vert_{L^{2}(0,T;H_{\Gamma})}\right).
\label{7.45}
\end{align}
By \eqref{eq7.2-L}, \eqref{7.45} and Poincar\'e's inequalities \eqref{Po1}, \eqref{Po2}, we obtain \eqref{7.43}.
\end{proof}

\begin{lemma}
There exists a positive constant $C_{10}$, independent of $L\geq \widehat{L}$, such that
\begin{align}
&\Vert\xi_L\Vert_{L^{2}(0,T;H)}
+\Vert\xi_{L}\Vert_{L^{2}(0,T;H_{\Gamma})}
+\Vert\xi_{\Gamma,L}\Vert_{L^{2}(0,T;H_{\Gamma})}
+ \|\bm{\varphi}_L\|_{L^2(0,T;\mathcal{H}^2)}
\leq
C_{10}\left(1+\Vert\partial_{\mathbf{n}}\varphi_L\Vert_{L^{2}(0,T;H_{\Gamma})}\right).
\label{7.47}
\end{align}
\end{lemma}
\begin{proof}
To conclude, we apply the same argument for Lemma \ref{uni4} and then pass to the limit as $\varepsilon\to 0$. The estimate for
$\Vert\bm{\varphi}_L\Vert_{L^{2}(0,T;\mathcal{V}^2)}$ follows from the estimate for $\Vert\bm{\omega}_L\Vert_{L^{2}(0,T;\mathcal{V}^2)}$, \eqref{con-mass-L}
and the Poincar\'e inequality \eqref{Po4}.
\end{proof}

\textbf{Proof of Theorem \ref{asymptoticinfinity}.} To complete the proof, it remains to control $\Vert\partial_{\mathbf{n}}\varphi_L\Vert_{L^{2}(0,T;H_{\Gamma})}$. Thanks to the trace theorem and the Ehrling lemma, we find
\begin{align}
\Vert\partial_{\mathbf{n}}\varphi_{L}\Vert_{L^2(0,T;H_{\Gamma})}
&\leq C\Vert\varphi_{L}\Vert_{L^2(0,T;H^{2-r}(\Omega))}\quad \text{for some}\ r\in (0,1/2)\notag\\
&\leq \gamma \Vert\varphi_{L}\Vert_{L^2(0,T;H^2(\Omega))}+ C_\gamma \Vert\varphi_{L}\Vert_{L^2(0,T;V)}.
\notag
\end{align}
Hence, taking $\gamma>0$ sufficiently small, we can deduce from \eqref{3.4-L} and \eqref{7.47} that
$$
\|\bm{\varphi}_L\|_{L^2(0,T;\mathcal{H}^2)}
\leq \frac12 \Vert\varphi_{L}\Vert_{L^2(0,T;H^2(\Omega))}+ C, \quad \forall\, L\geq \widehat{L}.
$$
This yields the uniform bound of $\|\bm{\varphi}_L\|_{L^2(0,T;\mathcal{H}^2)}$ for all $L\geq \widehat{L}$ and thus $\Vert\partial_{\mathbf{n}}\varphi_L\Vert_{L^{2}(0,T;H_{\Gamma})}$ is bounded as well.

Keeping the above uniform estimates in mind, we can take the limit as $L\to \infty$ (in the sense of a subsequence) and prove Theorem \ref{asymptoticinfinity} in the same way as for Theorem \ref{asymptotic0}.
\hfill $\square$

\appendix
\section{Appendix}
\noindent We report some technical lemmas that have been frequently used
in our analysis.

First, we recall the compactness lemma of Aubin-Lions-Simon
type (see, for instance, \cite{Lions} in the case $q>1$ and \cite{Simon}
when $q=1 $).

\begin{lemma}
\label{ALS} Let $X_{0} \overset{c}{\hookrightarrow } X_{1}\subset X_{2}$
where $X_{j}$ are (real) Banach spaces ($j=0,1,2$). Let $1<p\leq \infty $, $%
1\leq q\leq \infty ~$and $I$ be a bounded subinterval of $\mathbb{R}$. Then,
the sets
\begin{equation*}
\left\{ \varphi \in L^{p}\left( I;X_{0}\right) :\partial _{t}\varphi \in
L^{q}\left( I;X_{2}\right) \right\} \overset{c}{\hookrightarrow }
L^{p}\left( I;X_{1}\right),\quad \text{ if }1<p<\infty,
\end{equation*}
and
\begin{equation*}
\left\{ \varphi \in L^{p}\left( I;X_{0}\right) :\partial _{t}\varphi \in
L^{q}\left( I;X_{2}\right) \right\} \overset{c}{\hookrightarrow } C\left(
I;X_{1}\right),\quad \text{ if }p=\infty ,\text{ }q>1.
\end{equation*}
\end{lemma}

The following Ehrling lemma can be found in \cite{Lions}.
\begin{lemma}
\label{Ehrling}
Let $B_{0}$, $B_{1}$, $B$ be three Banach spaces so that $B_{0}$ and $B_{1}$ are reflexive. Moreover, $B_{0}\hookrightarrow\hookrightarrow B\hookrightarrow B_{1}$. Then, for each $\gamma>0$, there exists a positive constant $C_{\gamma}$ depends on $\gamma$ such that
$$\Vert z\Vert_{B}\leq\gamma\Vert z\Vert_{B_{0}}+C_{\gamma}\Vert z\Vert_{B_{1}},\quad \forall\, z\in B_{0}.$$
\end{lemma}

\bigskip
\noindent \textbf{Acknowledgments.} H. Wu is a member of the Key Laboratory of Mathematics for Nonlinear Sciences (Fudan University), Ministry
of Education of China. The research of H. Wu was partially supported by NNSFC Grant No. 12071084
and the Shanghai Center for Mathematical Sciences at Fudan University.

%

%
\end{document}